\pdfoutput=1
\documentclass[twoside]{article}
\usepackage[letterpaper, left=3.5cm, right=3.5cm, top=4cm, bottom=2.5cm]{geometry}
\usepackage{graphicx}
\usepackage{amsmath,amssymb,amsthm}
\usepackage{authblk}
\usepackage[utf8]{inputenc}
\usepackage{microtype}
\usepackage{mathrsfs} 
\usepackage{enumitem}
\usepackage{calc}
\usepackage{esint}
\usepackage{fancyhdr}
\usepackage{xcolor}
\usepackage[labelfont = bf]{caption}
\usepackage{doi}
\usepackage{subfigure}
\usepackage[numbers]{natbib}
\usepackage{stmaryrd}
\usepackage{mathtools}
\usepackage{float}

\newtheorem{theorem}{Theorem}[section]
\numberwithin{equation}{section}
\newtheorem{proposition}[theorem]{Proposition}
\newtheorem{definition}[theorem]{Definition}

\newtheorem{remark}[theorem]{Remark}
\newtheorem{lemma}[theorem]{Lemma}

\newtheorem{algorithm}[theorem]{Algorithm}

\usepackage{titlesec}
\titleformat{\section}{\normalfont\scshape\centering}{\thesection.}{0.5em}{}
\titleformat*{\subsection}{\itshape}
\titleformat*{\subsubsection}{\itshape}

\providecommand{\keywords}[1]
{
	{\small\textit{Keywords:~~} #1}
}

\providecommand{\MSC}[1]
{
	{\small\textit{AMS MSC (2020):~~} #1}
}

\usepackage{hyperref}
\hypersetup{
	colorlinks=true,
	linkcolor=blue,
	citecolor = blue,
	filecolor=magenta, 
	urlcolor=cyan,
}

\providecommand{\jumptmp}[2]{#1\llbracket{#2}#1\rrbracket}
\providecommand{\jump}[1]{\jumptmp{}{#1}}

\begin{document}
	\setlength{\abovedisplayskip}{5.5pt}
	\setlength{\belowdisplayskip}{5.5pt}
	\setlength{\abovedisplayshortskip}{5.5pt}
	\setlength{\belowdisplayshortskip}{5.5pt}

	\title{\vspace{-10mm} Explicit a posteriori error representation for variational problems and application to TV-minimization}
	\author[1]{Sören Bartels\thanks{Email: \texttt{bartels@mathematik.uni-freiburg.de}}}
	\author[2]{Alex Kaltenbach\thanks{Email: \texttt{kaltenbach@math.tu-berlin.de}}}
	\date{\today}
	    \affil[1]{\small{Department of Applied Mathematics, University of Freiburg, Hermann--Herder--Stra\ss e 10, 79104 Freiburg}}
		\affil[2]{\small{Institute of Mathematics, Technical University of Berlin, Straße des 17. Juni 136, 10623 Berlin}}
	\maketitle

	\pagestyle{fancy}
	\fancyhf{}
	\fancyheadoffset{0cm}
	\addtolength{\headheight}{-0.25cm}
	\renewcommand{\headrulewidth}{0pt} 
	\renewcommand{\footrulewidth}{0pt}
	\fancyhead[CO]{\textsc{Explicit error representation and application to TV-minimization}}
	\fancyhead[CE]{\textsc{S. Bartels and A. Kaltenbach}}
	\fancyhead[R]{\thepage}
	\fancyfoot[R]{}
	
	\begin{abstract}
		In this paper, we propose a general approach for explicit a posteriori error representation for convex minimization problems using basic convex duality relations. 
        Exploiting discrete orthogonality relations in the space of element-wise constant vector fields as well as a discrete integration-by-parts formula between the Crouzeix--Raviart~and~the~\mbox{Raviart--Thomas}~\mbox{element}, all convex duality relations are transferred to a discrete level, making the explicit a posteriori error representation  --initially based on continuous arguments only--  practicable~from~a numerical point of view. In addition, 
        we provide a generalized Marini formula for the primal solution that determines a discrete primal solution in terms of a given discrete dual solution. 
        We benchmark all these concepts via the Rudin--Osher--Fatemi model. This leads to an adaptive algorithm that yields a (quasi-optimal)
        linear~convergence~rate.
	\end{abstract}

	
	\MSC{35Q68; 49M25; 49M29; 65N30; 65N50}
	
	\section{Introduction}\label{sec:intro}
		\thispagestyle{empty}

    \hspace{5mm}The numerical analysis of the approximation of variational problems
    is challenging when these are non-differentiable, degenerate, or involve
    constraints. In particular, following established concepts for linear
    elliptic partial differential equations often leads to sub-optimal results only.
    The framework of convex duality provides an attractive concept to 
    reveal hidden information and structures to obtain quasi-optimal error representation formulas
under meaningful regularity conditions. Similar to \cite{Rep99,Repin18}, we first exploit this 
idea to derive explicit computable a posteriori error estimates for~a~\textit{natural} error 
measure. Then, this general result is  transferred to a non-differentiable model problem with discontinuous solutions. As a whole, our results, similar to \cite{Rep99,Repin18}, show that
the question of developing asymptotically exact a posteriori error estimators is
rather a question of identifying optimal error quantities. However, different from \cite{Rep99,Repin18}, we also propose a general approach for making our results practicable from a numerical~point~of~view.\enlargethispage{10mm}

Given a domain $\Omega\subseteq \mathbb{R}^d$,~${d\in \mathbb{N}}$, 
a convex energy density $\phi\colon\mathbb{R}\to \mathbb{R}\cup\{+\infty\}$, a
(Lebesgue)~mea-surable energy density $\psi\colon\Omega\times\mathbb{R}\to \mathbb{R}\cup\{+\infty\}$  that is convex with respect to~the~second~argument, and a Banach space $X$ consisting of functions defined in 
$\Omega$, we denote by the minimization of the energy functional $I\colon X\to \mathbb{R}\cup\{+\infty\}$, for every $v\in X$ defined by
\begin{align}
    I(v) \coloneqq \int_{\Omega}{\phi(\nabla v)\,\mathrm{d}x} + \int_{\Omega}{\psi(\cdot, v)\,\mathrm{d}x}\,,\label{intro:primal}
\end{align}
the \textit{primal problem}.

Its (Fenchel) \textit{dual problem} consists in the maximization of the functional $D\colon Y\to \mathbb{R}\cup\{-\infty\}$, where $Y$ is a Banach space consisting of vector fields defined in 
$\Omega$, for every $y\in Y$ is defined by
\begin{align}
D(y) \coloneqq -\int_{\Omega}{\phi^*(y)\,\mathrm{d}x}  - \int_{\Omega}{\psi^*(\cdot, \mathrm{div}\,y)\,\mathrm{d}x}\,.\label{intro:dual}
\end{align}
Here, $\phi^*\colon \mathbb{R}^d\to \mathbb{R}\cup\{+\infty\}$ and $\psi^*\colon \Omega\times \mathbb{R}\to \mathbb{R}\cup\{+\infty\}$ (with respect to the second argument) denote the (Fenchel) conjugates of $\phi\colon\mathbb{R}\to \mathbb{R}\cup\{+\infty\}$ and $\psi\colon\Omega\times\mathbb{R}\to \mathbb{R}\cup\{+\infty\}$,~respectively.
Under rather general conditions, cf.\  \cite{ZeiIII,ET99}, we have the well-posedness of  the 
primal problem  and the dual problem, i.e., the existence of a minimizer $u\in X$ of \eqref{intro:primal}, i.e., a \textit{primal solution}, and of a maximizer $z\in Y$~of~\eqref{intro:dual}, i.e., a \textit{dual solution}, and the \textit{strong duality relation}
\begin{align}
\min_{v\in X} I(v) = I(u)= D(z) = \max_{y\in Y} D(y)\,.\label{intro:strong_duality}
\end{align}
Since $u\hspace{-0.1em}\in \hspace{-0.1em}X$ and $z\hspace{-0.1em}\in\hspace{-0.1em} Y$ are optimal for \eqref{intro:primal} and \eqref{intro:dual}, respectively, it holds~$0\hspace{-0.1em}\in\hspace{-0.1em} \partial I(u)$~and~${0\hspace{-0.1em}\in\hspace{-0.1em} \partial D(z)}$.
In particular, for every $v\in X$ and $y\in Y$, the quantities
\begin{align}
\rho_I^2(v,u)& \coloneqq  I(v) - I(u)\,,\label{eq:rho_I}\\
\rho_{-D}^2(y,z)& \coloneqq D(z) - D(y)\,,\label{eq:rho_D}
\end{align}
are non-negative. They define distances, if \eqref{intro:primal} and \eqref{intro:dual}, respectively, are 
strictly~convex,~and are called coercivity functionals or optimal convexity measures.

For accessible and admissible approximations $v\hspace{-0.1em}\in\hspace{-0.1em} X$ and $y\hspace{-0.1em}\in\hspace{-0.1em} Y$ of the solutions $u \hspace{-0.1em}\in \hspace{-0.1em} X$~and~${z \hspace{-0.1em}\in \hspace{-0.1em} Y}$, given the definitions \eqref{eq:rho_I} and \eqref{eq:rho_D}, the strong duality relation \eqref{intro:strong_duality} implies~the~error~identity 
\begin{align}\label{intro:error_identity}
\begin{aligned}
\rho_I^2(v,u) + \rho_{-D}^2(y,z)
 = I(v) - D(z)
 \eqqcolon  \eta^2(v,y)\,.
\end{aligned}
\end{align}
Hence, the fully computable error estimator $\eta^2\colon X\times Y\to \mathbb{R}\cup\{+\infty\}$, cf.\ \eqref{intro:error_identity}, exactly
represents the sum of the primal and dual approximation errors, i.e., of \eqref{eq:rho_I} and \eqref{eq:rho_D}.

The \hspace{-0.1mm}error \hspace{-0.1mm}representation \hspace{-0.1mm}\eqref{intro:error_identity} \hspace{-0.1mm}can \hspace{-0.1mm}be \hspace{-0.1mm}seen \hspace{-0.1mm}as \hspace{-0.1mm}a \hspace{-0.1mm}generalization \hspace{-0.1mm}of \hspace{-0.1mm}the \hspace{-0.1mm}Prager--Synge
\hspace{-0.1mm}result, \hspace{-0.1mm}cf. \hspace{-0.1mm}\cite{PraSyn47,Brae09,Braess13}, which states that for the Poisson problem, i.e., $\phi\hspace*{-0.1em}\coloneqq\hspace*{-0.1em} \frac{1}{2}\vert \cdot\vert^2\hspace*{-0.1em}\in \hspace*{-0.1em} C^1(\mathbb{R}^d)$,~${\psi\hspace*{-0.1em}\coloneqq\hspace*{-0.1em} ((t,x)^\top\hspace*{-0.1em}\mapsto\hspace*{-0.1em} -f(x)t)\colon}$ $\Omega\times\mathbb{R}\to \mathbb{R}\cup\{+\infty\}$, where $f\in L^2(\Omega)$, $X\coloneqq W^{1,2}_D(\Omega)$,~and~${Y\coloneqq W^2_N(\textup{div};\Omega)}$,~for~every~${v\in W^{1,2}_D(\Omega)}$ and ${y\in W^2_N(\textup{div};\Omega)}$ with $-\textup{div}\,y=f$ a.e.\ in $\Omega$, we have that 
\begin{align}\label{intro:prager_synge}
\begin{aligned}
\tfrac12 \|\nabla v -\nabla u\|_{L^2(\Omega;\mathbb{R}^d)}^2 + \tfrac12 \|y - z \|_{L^2(\Omega;\mathbb{R}^d)}^2 
= \tfrac12 \|\nabla v-y \|^2_{L^2(\Omega;\mathbb{R}^d)}\,.
\end{aligned}
\end{align}
The equation \eqref{intro:prager_synge} has been used by various authors to define error estimators; for a comprehensive list of references, we refer the reader to \cite{BB20}. 
Often, local procedures are devised to~construct~an~ad-missible vector field
$y\in W^2_N(\textup{div};\Omega)$ with $-\textup{div}\,y=f$ a.e.\ in $\Omega$ from~a~given~function~${v\in W^{1,2}_D(\Omega)}$. While this leads to efficient procedures
to obtain accurate error estimators, the arguments cannot be expected to transfer
to non-linear problems. Another alternative to computing approximations
for the primal and dual problems consists in using finite element methods
for which reconstruction formulas are available, e.g., using the discontinuous Crouzeix--Raviart finite element
method and the Marini formula in the case of the Poisson problem, cf.\ \cite{Mar85}.\enlargethispage{7mm}

It has recently been found (cf.\ \cite{CP20,Bar21}) that the discontinuous Crouzeix--Raviart finite~element method leads to quasi-optimal a priori error estimates for non-linear and non-differentiable~problems, while continuous finite element methods  provide only a sub-optimal
convergence~behavior. In the derivation of those results, a general 
discrete convex duality theory with Raviart--Thomas vector fields has emerged that 
also leads to reconstruction
formulas in rather general settings. As a consequence, given an approximation 
$v\in X$ or $y\in Y$, respectively, the missing one can be obtained via a simple post-processing procedure.
Then, the pair leads to the error representation formula \eqref{intro:error_identity}. It should also 
be noted that neither $v\in X$ nor $y\in Y$ needs to be optimal in a subspace
of $X$ or $Y$. By introducing appropriate residuals, any pair of admissible
approximations of~$u\in X$~and~$z\in Y$ can be used. This is particularly important for non-linear
problems,~i.e.,~non-quadratic functionals, where an exact solution~of~discrete~problems~is~neither~possible~nor~rational. 

A  \hspace{-0.1mm}difficulty \hspace{-0.1mm}in \hspace{-0.1mm}the \hspace{-0.1mm}application \hspace{-0.1mm}of \hspace{-0.1mm}the \hspace{-0.1mm}explicit \hspace{-0.1mm}a \hspace{-0.1mm}posteriori \hspace{-0.1mm}error \hspace{-0.1mm}representation 
\hspace{-0.1mm}formula~\hspace{-0.1mm}\eqref{intro:error_identity}~\hspace{-0.1mm}arises from the condition that $v\in X$ and $y\in Y$ need to be admissible for
the~functionals~\eqref{intro:primal}~and~\eqref{intro:dual}. In the case of the Poisson problem,
this arises, e.g., via element-wise constant approximations of  $f\in L^2(\Omega)$
that are the images of Raviart--Thomas vector fields under~the~divergence~operator. While data terms can be controlled by introducing appropriate 
data oscillation terms, structural peculiarities of the energy densities
$\phi\colon \mathbb{R}^d\to \mathbb{R}\cup\{+\infty\}$ and $\psi\colon \Omega\times \mathbb{R}\to \mathbb{R}\cup\{+\infty\}$ and their (Fenchel) conjugates $\phi^*\colon \mathbb{R}^d\to \mathbb{R}\cup\{+\infty\}$ and $\psi^*\colon \Omega\times \mathbb{R}\to \mathbb{R}\cup\{+\infty\}$~are~often~more~challenging.
We illustrate this 
by analyzing a non-differentiable 
problem
which   leads to a new error analysis and an adaptive refinement procedure
for the computationally~challenging~\mbox{minimization}~problem.

With $\phi = |\cdot|\in C^0(\mathbb{R}^d)$ and $\psi=((x,t)^\top\mapsto \frac{\alpha}{2}(t-g(x))^2)\colon \Omega\times \mathbb{R}\to \mathbb{R}$
 for a given function
$g\in L^2(\Omega)$, i.e., the \textit{noisy image}, and a given parameter $\alpha>0$, i.e., the \textit{fidelity parameter},
the Rudin--Osher--Fatemi (ROF) model, cf.\  \cite{ROF92}, seeks a minimizing function $u\in BV(\Omega)\cap L^2(\Omega)$, i.e., the \textit{de-noised image}, where $BV(\Omega)$ denotes the space of functions with bounded variation, 
for the functional $I\colon BV(\Omega)\cap L^2(\Omega)\to \mathbb{R}$, for every $v\in  BV(\Omega)\cap L^2(\Omega)$ defined by
\begin{align}
I(v) \coloneqq \vert \mathrm{D}v\vert(\Omega) + \tfrac{\alpha}{2} \|v-g\|_{L^2(\Omega)}^2 \,,\label{intro:ROF_primal}
\end{align}
where \hspace{-0.1mm}$\vert \mathrm{D}(\cdot)\vert(\Omega)\colon \hspace{-0.1em}BV(\Omega)\hspace{-0.1em}\to\hspace{-0.1em} [0,+\infty]$ \hspace{-0.1mm}denotes \hspace{-0.1mm}the \hspace{-0.1mm}total \hspace{-0.1mm}variation \hspace{-0.1mm}functional.
\hspace{-0.1mm}The~\hspace{-0.1mm}(Fenchel)~\hspace{-0.1mm}\mbox{(pre-)dual} problem to the minimization of the functional \eqref{intro:ROF_primal} consists in the maximization of 
the functional $D\colon W_N^2(\textup{div};\Omega)\cap L^\infty(\Omega;\mathbb{R}^d)\to \mathbb{R}\cup\{-\infty\}$, for every $y\in W_N^2(\textup{div};\Omega)\cap L^\infty(\Omega;\mathbb{R}^d)$ defined by 
\begin{align}
D(y) \coloneqq -I_{K_1(0)}(y)-\tfrac{1}{2\alpha} \|\mathrm{div}\, y+\alpha g\|_{L^2(\Omega)}^2+\tfrac{\alpha}{2}\| g\|_{L^2(\Omega)}^2\,,\label{intro:ROF_dual}
\end{align}
where 
$I_{K_1(0)}(y)\coloneqq 0$ if $\vert y\vert \leq 1$ a.e.\ in $\Omega$ and $I_{K_1(0)}(y)\coloneqq +\infty$~else.
The primal solution $u\in BV(\Omega)$ $\cap L^2(\Omega)$, i.e., the unique minimizer of \eqref{intro:ROF_primal},  and a dual solution ${z\in W_N^2(\textup{div};\Omega)\cap L^\infty(\Omega;\mathbb{R}^d)}$,~i.e., a (possibly non-unique) maximizer~of~\eqref{intro:ROF_dual}, are 
(formally) related via, cf. \cite[p.\ 284]{CCMN08}, 
\begin{align}
\begin{aligned}
z&\in \left.\begin{cases}
    \big\{\frac{\nabla u}{|\nabla u|}\big\}&\text{ if }|\nabla u|>0\\
    K_1(0)&\text{ if }|\nabla u|=0
\end{cases}\right\}&&\quad\text{ a.e.\ in }\Omega\,,\\
\textup{div}\, z &= \alpha\, (u-g)&&\quad\text{ a.e.\ in }\Omega\,.
\end{aligned}\label{intro:ROF_elg}
\end{align}
The relations \eqref{intro:ROF_elg} determine $z\in W_N^2(\textup{div};\Omega)\cap L^\infty(\Omega;\mathbb{R}^d)$ via $u\in BV(\Omega)\cap L^2(\Omega)$~and~vice~versa. 
A 
Crouzeix--Raviart finite element approximation of \eqref{intro:primal} is given by the minimization of the regularized, discrete functional
$I_{h,\varepsilon}^{cr}\colon \mathcal{S}^{1,cr}(\mathcal{T}_h)\to \mathbb{R}$, $h,\varepsilon>0$, for every $v_h\in \mathcal{S}^{1,cr}(\mathcal{T}_h)$ defined by
\begin{align*}
I_{h,\varepsilon}^{cr}(v_h) \coloneqq \|f_\varepsilon(\vert \nabla_{\!h} v_h\vert )\|_{L^1(\Omega)}
+ \tfrac{\alpha}{2} \|\Pi_h(v_h-g)\|_{L^2(\Omega)}^2 \,.
\end{align*}
Here, $\nabla_{\!h}$ is the element-wise application of the gradient operator
and $f_\varepsilon\!\in\! \smash{C^1(\mathbb{R})}$ is a regularization of the modulus $\vert \cdot\vert$, and $\Pi_h$ denotes
the (local) $L^2$-projection onto element-wise~constant~functions. 
A quasi-optimal dual Raviart--Thomas vector field $z_{h,\varepsilon}^{rt}\in \mathcal{R}T^0_N(\mathcal{T}_h)$ can be associated with a
minimizing function $u_{h,\varepsilon}^{cr}\in \mathcal{S}^{1,cr}(\mathcal{T}_h)$ of $I_{h,\varepsilon}^{cr}\colon \mathcal{S}^{1,cr}(\mathcal{T}_h)\to \mathbb{R}$ via the reconstruction formula
\begin{align}
z_{h,\varepsilon}^{rt} = \tfrac{f_\varepsilon'(\vert \nabla_{\!h} u_{h,\varepsilon}^{cr}\vert) }{\vert \nabla_{\!h} u_{h,\varepsilon}^{cr}\vert}\nabla_{\!h} u_{h,\varepsilon}^{cr}
 + \alpha \tfrac{\Pi_h (u_{h,\varepsilon}^{cr} -g)}{d}\big( \mathrm{id}_{\mathbb{R}^d}- \Pi_h \mathrm{id}_{\mathbb{R}^d}\big)\quad\text{ in }\mathcal{R}T^0_N(\mathcal{T}_h)\,.\label{intro_marini}
\end{align}
For canonical choices of $f_\varepsilon\in C^1(\mathbb{R})$, e.g., 
$f_\varepsilon =\vert \cdot\vert_\varepsilon= ((\cdot)^2+\varepsilon^2)^{1/2}$, it holds $\vert \Pi_h z_{h,\varepsilon}^{rt}\vert\le 1$~a.e.~in~$\Omega$, but not
$|z_{h,\varepsilon}^{rt}|\le 1$ a.e.\ in $\Omega$. Thus, we employ $f_\varepsilon = (1-\varepsilon)\, |\cdot|_\varepsilon$,
so that
$|f_\varepsilon'(t)|\le 1-\varepsilon$~for~all~$t\in \mathbb{R}$. The choice $\varepsilon\sim h^2$ in \eqref{intro_marini} and an additional projection step onto $K_1(0)$
lead to an accurate approximation $\overline z_{h,\varepsilon}^{rt}\in \mathcal{R}T^0_N(\mathcal{T}_h)$ of $z\in W_N^2(\textup{div};\Omega)\cap L^\infty(\Omega;\mathbb{R}^d)$, which 
satisfies $|\overline z_{h,\varepsilon}^{rt}|\le 1$~a.e.~in~$\Omega$ and, thus, represents an admissible test function that leads to the definition
of an error estimator. The resulting adaptive mesh-refinement procedure  leads to significantly
improved experimental convergence rates compared to recent related contributions, cf.~\cite{bartels15,BM20,BTW21}. More precisely, we report quasi-optimal linear convergence rates which have been obtained only~for~meshes~with quadratic grading towards a sufficiently~simple~jump~set~of~a~\mbox{piece-wise}~regular~$g$~in~\cite{BTW21}.\enlargethispage{10mm}

	\textit{This article is organized as follows:} In Section \ref{sec:preliminaries}, we introduce the~employed~notation~and~the relevant finite element spaces. In Section \ref{sec:convex_min}, we propose a general approach for explicit a posteriori error representation for convex minimization problems based on (discrete) convex duality relations. In Section \ref{sec:ROF},
    we transfer the concepts of Section \ref{sec:convex_min} to the Rudin--Osher--Fatemi model and propose a regularization scheme.  In Section \ref{sec:experiments}, we review our theoretical~findings~via~numerical~experiments.
   \newpage
	
	\section{Preliminaries}\label{sec:preliminaries}

    \subsection{Convex analysis}
    
    \hspace{5mm}For a (real) Banach space $X$, which is equipped with the norm $\|\cdot\|_X\colon X\to \mathbb{R}_{\ge 0}$, we denote its corresponding (continuous) dual space by $X^*$ equipped with the dual norm 
	$\|\cdot\|_{X^*}\colon X^*\to \mathbb{R}_{\ge 0}$, defined by $\|x^*\|_{X^*}\coloneqq \sup_{\|x\|_X\leq 1}{\langle x^*,x\rangle_X}$ for every $x^*\in X^*$, where $\langle \cdot,\cdot\rangle_X\colon X^*\times X\to \mathbb{R}$, defined by $\langle x^*,x\rangle_X\coloneqq x^*(x)$ for every $x^*\in X^*$ and $x\in X$, denotes the duality pairing.
	A functional $F\colon X\to \mathbb{R}\cup\{+\infty\}$ is called \textit{sub-differentiable} in $x\in  X$, if $ F(x)<\infty $ and if there exists $x^*\in  X^*$, called  \textit{sub-gradient}, such that for every $ y\in X $, it holds
	\begin{align}
		\langle x^*,y-x\rangle_X\leq F(y)-F(x)\,.\label{eq:subgrad}
	\end{align} 
	The  \textit{sub-differential}  $\partial F\colon X\to  2^{X^*}$ of a functional $F\colon X\to \mathbb{R}\cup\{+\infty\}$ for every $ x\in X$~is~defined~by $(\partial F)(x)\coloneqq \{x^*\in X^*\mid \eqref{eq:subgrad}\text{ holds for }x^*\}$ if $F(x)<\infty$ and $(\partial F)(x)\coloneqq \emptyset$ else. 
    
    For a given functional $F\colon X\to \mathbb{R}\cup\{\pm\infty\}$, we denote its corresponding \textit{(Fenchel)~conjugate}~by $F^*\colon X^*\to \mathbb{R}\cup\{\pm\infty\}$, which for every $x^*\in X^*$ is defined by 
    \begin{align}
        F^*(x^*)\coloneqq \sup_{x\in X}{\langle x^*,x\rangle_X-F(x)}\,.\label{def:fenchel}
    \end{align}
	If $F\colon X\to \mathbb{R}\cup\{+\infty\}$ is a proper, convex, and lower semi-continuous functional, then also~its~(Fen-chel) conjugate $F^*\colon X^*\to\mathbb{R}\cup\{+\infty\}$ is a proper, convex, and lower semi-continuous~functional,  cf.\  \cite[p.\ 17]{ET99}. 
	Furthermore, for every $x^*\in X^*$ and $x\in X$ such that 
	$ F^*(x^*)+F(x)$ is well-defined, i.e., the critical case $\infty-\infty$ does not occur, the \textit{Fenchel--Young inequality}
	\begin{align}
		\langle x^*,x\rangle_X\leq F^*(x^*)+F(x)\label{eq:fenchel_young_ineq}
	\end{align}
	applies. 
	In particular, 
	for every $x^*\in X^*$ and $x\in X$, it holds the \textit{Fenchel--Young identity}
	\begin{align}
		x^*\in (\partial F)(x)\quad\Leftrightarrow \quad	\langle x^*,x\rangle_X= F^*(x^*)+F(x)\,.\label{eq:fenchel_young_id}
	\end{align}
    \hspace{5mm}The following convexity measures for functionals play an important role in the derivation~of an explicit a posteriori error representation for convex minimization problems in Section \ref{sec:convex_min}; for further information, please refer to  \cite{bregman67,NSV00,OBGXY05,bartels15}.

  \begin{definition}[Br\'egman distance and symmetric Br\'egman distance]\label{def:convexity_measure}
        Let $X$ be a (real) Banach space and $F\colon X\to \mathbb{R}\cup\{+\infty\}$ proper, i.e., $D(F)\coloneqq \{x\in X\mid F(x)<\infty\}\neq \emptyset$.
        \begin{itemize}[noitemsep,topsep=2pt,leftmargin=!,labelwidth=\widthof{(ii)}]
            \item[(i)] The \textup{Br\'egman distance} $\sigma^2_F\colon 
        D(F)\times X\to [0,+\infty]$ for every $x\in D(F)$~and~$y\in X$~is~defined~by
        \begin{align*}
            \sigma^2_F(y,x)\coloneqq F(y)-F(x)-\sup_{x^*\in (\partial F)(x)}{\langle x^*,y-x\rangle_X}\,,
        \end{align*}
        where we use the convention $\sup(\emptyset)\coloneqq-\infty$.
        \item[(ii)] The \textup{Br\'egman distance} $\sigma^2_F\colon 
        D(F)^2\to  [0,+\infty]$ for every $x,y\in D(F)$~is~defined~by
        \begin{align*}
            \sigma_{F,s}^2(y,x)\coloneqq \sigma_F^2(y,x)+\sigma_F^2(x,y)=\inf_{x^*\in (\partial F)(x);y^*\in (\partial F)(y)}{\langle x^*-y^*,x-y\rangle_X}\,,
        \end{align*}
        where we use the convention $\inf(\emptyset)\coloneqq +\infty$.
        \end{itemize}
 \end{definition}

    \begin{definition}[Optimal convexity measure at a minimizer]\label{def:convexity_measure_optimal}
        Let $X$ be a (real) Banach~space~and $F\colon X\to \mathbb{R}\cup\{+\infty\}$ proper. Moreover, let $x\in X$ be minimal for $F\colon X\to \mathbb{R}\cup\{+\infty\}$.~Then,~the \textup{optimal convexity measure} $\rho^2_F\colon 
       X^2\to [0,+\infty]$ \textup{at} $x\in X$ for every $y\in X$~is~defined~by
        \begin{align*}
            \rho^2_F(y,x)\coloneqq F(y)-F(x)\ge 0 \,.
        \end{align*}
 \end{definition}

    \begin{remark}\label{rem:convexity_measure_optimal}
     Let $X$ be a (real) Banach space and $F\colon X\to \mathbb{R}\cup\{+\infty\}$ proper. Moreover, let $x\in X$ be minimal for $F\colon X\to \mathbb{R}\cup\{+\infty\}$. Then, due to $0\in (\partial F)(x)$, for every $y\in X$, it holds
        \begin{align*}
            \sigma^2_F(y,x)\le  \rho^2_F(y,x)\,.
        \end{align*}
    \end{remark}

    \subsection{Function spaces}
    
	\hspace{5mm}Throughout the article,  we denote by ${\Omega \subseteq \mathbb{R}^d}$,~${d \in \mathbb{N}}$, a bounded polyhedral Lipschitz domain, whose (topological) boundary is disjointly divided~into~a~closed~Dirichlet part $\Gamma_D$ and an open Neumann part $\Gamma_N$, i.e., ${\partial\Omega = \Gamma_D\cup\Gamma_N}$~and~${\emptyset = \Gamma_D\cap\Gamma_N}$. \enlargethispage{3mm}

	For $p\in \left[1,\infty\right]$ and $l\in \mathbb{N}$, we employ the standard notations\footnote{Here, $W^{\smash{-\frac{1}{p},p}}(\Gamma_N)\coloneqq  (W^{\smash{1-\frac{1}{p'},p'}}(\Gamma_N))^*$ and $W^{\smash{-\frac{1}{p},p}}(\partial\Omega)\coloneqq  (W^{\smash{1-\frac{1}{p'},p'}}(\partial\Omega))^*$.}
	\begin{align*}
		\begin{aligned}
		W^{1,p}_D(\Omega;\mathbb{R}^l)&\coloneqq  \big\{v\in L^p(\Omega;\mathbb{R}^l)&&\hspace*{-3.25mm}\mid \nabla v\in L^p(\Omega;\mathbb{R}^{l\times d}),\, \textup{tr}\,v=0\text{ in }L^p(\Gamma_D;\mathbb{R}^l)\big\}\,,\\
		W^{p}_N(\textup{div};\Omega)&\coloneqq  \big\{y\in L^p(\Omega;\mathbb{R}^d)&&\hspace*{-3.25mm}\mid \textup{div}\,y\in L^p(\Omega),\,\textup{tr}_n\,y=0\text{ in }W^{-\frac{1}{p},p}(\Gamma_N)\big\}\,,
	\end{aligned}
	\end{align*}
	$\smash{W^{1,p}(\Omega;\mathbb{R}^l)\coloneqq  W^{1,p}_D(\Omega;\mathbb{R}^l)}$ if $\smash{\Gamma_D=\emptyset}$, and $\smash{W^{p}(\textup{div};\Omega)\coloneqq  W^{p}_N(\textup{div};\Omega)}$ if $\smash{\Gamma_N=\emptyset}$,
	where~we~\mbox{denote} by $\textup{tr}\colon\smash{W^{1,p}(\Omega;\mathbb{R}^l)}\hspace{-0.05em} \to \hspace{-0.05em}\smash{L^p(\partial\Omega;\mathbb{R}^l)}$ and by $
	\textup{tr}_n(\cdot)\colon\smash{W^p(\textup{div};\Omega)}\hspace{-0.05em} \to \hspace{-0.05em}\smash{W^{-\frac{1}{p},p}(\partial\Omega)}$, the trace~and~\mbox{normal} trace operator, respectively. In particular, we always omit $\textup{tr}(\cdot)$ and $\textup{tr}_n(\cdot)$. In addition, we employ the abbreviations $L^p(\Omega)\hspace{-0.05em} \coloneqq\hspace{-0.05em}  L^p(\Omega;\mathbb{R}^1)$, ${W^{1,p}(\Omega)\hspace{-0.05em}\coloneqq \hspace{-0.05em} W^{1,p}(\Omega;\mathbb{R}^1)}$,~and~${W^{1,p}_D(\Omega)\hspace{-0.05em}\coloneqq \hspace{-0.05em} W^{1,p}_D(\Omega;\mathbb{R}^1)}$. For  (Lebesgue) measurable functions $u,v\colon \hspace{-0.1em}\Omega\hspace{-0.1em}\to\hspace{-0.1em} \mathbb{R}$ and a (Lebesgue) measurable set ${M\hspace{-0.1em}\subseteq \hspace{-0.1em}\Omega}$,~we~write
    \begin{align*}
        (u,v)_{M}\coloneqq \int_{M}{u\,v\,\mathrm{d}x}\,,
    \end{align*}
    whenever the right-hand side is well-defined. Analogously, for  (Lebesgue) measurable~vector~fields $z,y\colon \Omega\to \mathbb{R}^d$ and a (Lebesgue) measurable set $M\subseteq\Omega$, we  write ${(z,y)_{M}\coloneqq \int_{M}{z\cdot y\,\mathrm{d}x}}$. Moreover, 
    let $\vert \textup{D}(\cdot)\vert(\Omega) \colon L^1_{\textup{loc}}(\Omega) \to \mathbb{R}\cup\{+\infty\}$, for every $v\in L^1_{\textup{loc}}(\Omega)$ defined by\footnote{Here, $C_c^\infty(\Omega;\mathbb{R}^d)$ denotes the space of smooth and in $\Omega$ compactly supported vector fields.} 
	\begin{align*}
		\vert \textup{D} {v}\vert(\Omega)\coloneqq  \sup\big\{-(v,\textup{div}\,\phi)_{\Omega}\mid \phi\in C_c^\infty(\Omega;\mathbb{R}^d);
  \|\phi\|_{L^\infty(\Omega;\mathbb{R}^d)}\leq 1\big\} \,,
	\end{align*}
	denote \hspace{-0.1mm}the \hspace{-0.1mm}\textit{total \hspace{-0.1mm}variation} \hspace{-0.1mm}functional. \hspace{-0.1mm}Then, \hspace{-0.1mm}the \hspace{-0.1mm}\textit{space \hspace{-0.1mm}of \hspace{-0.1mm}functions \hspace{-0.1mm}with \hspace{-0.1mm}bounded \hspace{-0.1mm}variation} \hspace{-0.1mm}is \hspace{-0.1mm}defined by
    \begin{align*}
        BV(\Omega)\coloneqq  \big\{v\in L^1(\Omega)\mid \vert \textup{D}v\vert(\Omega)<\infty\big\}\,.
    \end{align*}

	\subsection{Triangulations}

 \hspace{5mm}Throughout the entire paper, we denote by $\{\mathcal{T}_h\}_{h>0}$, a family~of~regular, i.e., uniformly shape regular and conforming, triangulations of $\Omega\subseteq \mathbb{R}^d$, $d\in\mathbb{N}$, cf.\  \!\cite{EG21}. 
	Here,~${h>0}$~refers to the \textit{average mesh-size}, i.e., if we set $h_T\coloneqq  \textup{diam}(T)$ for all $T\in \mathcal{T}_h$,~then,~we~have~that~${h 
= \frac{1}{\textup{card}(\mathcal{T}_h)}\sum_{T\in \mathcal{T}_h}{h_T}
 }$.
	For every element $T \in \mathcal{T}_h$,
 we denote by $\rho_T>0$, the supremum of diameters of~inscribed~balls. We assume that there exists a constant $\omega_0>0$, independent of $h>0$, such that $\max_{T\in \mathcal{T}_h}{h_T}{\rho_T^{-1}}\le
 \omega_0$. The smallest such constant is called the \textit{chunkiness} of $\{\mathcal{T}_h\}_{h>0}$.  The~sets~$\mathcal{S}_h$, $\mathcal{S}_h^{i}$, $\mathcal{S}_h^{\partial}$, and $\mathcal{N}_h$  contain the sides, interior sides, boundary sides, and vertices, respectively, of the elements of $\mathcal{T}_h$. 
 We have the following relation between the average mesh-size and the number of vertices:
 \begin{align*}
     h\sim \textup{card}(\mathcal{N}_h)^{-1/d}\,.
 \end{align*}
	
	For $k\in \mathbb{N}\cup\{0\}$ and $T\in \mathcal{T}_h$, let $\mathcal{P}_k(T)$ denote the set of polynomials of maximal~degree~$k$~on~$T$. Then, for $k\in \mathbb{N}\cup\{0\}$~and $l\in \mathbb{N}$, the sets of continuous and~\mbox{element-wise}~polynomial functions or vector~fields,~respectively, are defined by
	\begin{align*}
    \mathcal{L}^k(\mathcal{T}_h)^l&\coloneqq  \big\{v_h\in L^\infty(\Omega;\mathbb{R}^l)\mid v_h|_T\in\mathcal{P}_k(T)^l\text{ for all }T\in \mathcal{T}_h\big\}\,,\\
	\mathcal{S}^k(\mathcal{T}_h)^l&\coloneqq  	\mathcal{L}^k(\mathcal{T}_h)^l\cap C^0(\overline{\Omega};\mathbb{R}^l)\,.
	\end{align*}
	For every $T\in \mathcal{T}_h$ and $S\in \mathcal{S}_h$, let $\smash{x_T\coloneqq  \frac{1}{d+1}\sum_{z\in \mathcal{N}_h\cap T}{z}\in T}$ and $\smash{x_S\coloneqq  \frac{1}{d}\sum_{z\in \mathcal{N}_h\cap S}{z}\in S}$ denote the barycenters of $T$ and $S$, respectively. The (local) $L^2$-projection operator $\Pi_h\colon L^1(\Omega;\mathbb{R}^l)\to \mathcal{L}^0(\mathcal{T}_h)^l$ onto element-wise constant functions or vector  fields, respectively, for every 
    $v\in L^1(\Omega) $, is defined by $\Pi_h v|_T\coloneqq \fint_T{v\,\mathrm{d}x}$ for all $T\in \mathcal{T}_h$. 
 The element-wise~gradient 
 $\nabla_{\!h}\colon \hspace{-0.1em}\mathcal{L}^1(\mathcal{T}_h)^l\hspace{-0.1em}\to\hspace{-0.1em} \mathcal{L}^0(\mathcal{T}_h)^{l\times d}$, for every $v_h\hspace{-0.1em}\in\hspace{-0.1em} \mathcal{L}^1(\mathcal{T}_h)^l$,~is~defined~by $\nabla_{\!h}v_h|_T\hspace{-0.1em}\coloneqq  \hspace{-0.1em}\nabla(v_h|_T)$ for~all~${T\hspace{-0.1em}\in\hspace{-0.1em} \mathcal{T}_h}$.
 	
	\subsubsection{Crouzeix--Raviart element}\enlargethispage{11mm}
	
	\qquad The Crouzeix--Raviart finite element space, cf.\ \cite{CR73}, consists of \mbox{element-wise} affine functions that are continuous at the barycenters of inner element sides, i.e.,\footnote{Here, for every inner side $S\in\mathcal{S}_h^{i}$, $\jump{v_h}_S\coloneqq  v_h|_{T_+}-v_h|_{T_-}$ on $S$, where $T_+, T_-\in \mathcal{T}_h$ satisfy $\partial T_+\cap\partial T_-=S$, and for every boundary $S\in\mathcal{S}_h^{\partial}$, $\jump{v_h}_S\coloneqq  v_h|_T$ on $S$, where $T\in \mathcal{T}_h$ satisfies $S\subseteq \partial T$.}
	\begin{align*}\mathcal{S}^{1,\textit{cr}}(\mathcal{T}_h)\coloneqq  \big\{v_h\in \mathcal{L}^1(\mathcal{T}_h)\mid \jump{v_h}_S(x_S)=0\text{ for all }S\in \mathcal{S}_h^{i}\big\}\,.
	\end{align*}
    Note that $\mathcal{S}^{1,\textit{cr}}(\mathcal{T}_h)\subseteq BV(\Omega)$. More precisely, for every $v_h\in \mathcal{S}^{1,\textit{cr}}(\mathcal{T}_h)$, cf. \cite[Theorem 1.63]{braides98}, we have that $\mathrm{D}v_h=\nabla_{\! h}v_h\otimes \mathrm{d}x+\jump{v_h}\otimes \mathrm{d}s|_{\mathcal{S}_h}$ with $\nabla_{\! h}v_h\otimes \mathrm{d}x\perp \jump{v_h}\otimes \mathrm{d}s|_{\mathcal{S}_h}$, so that, cf.\ \cite{BBHSVN23}, 
    \begin{align}\label{eq:total_variation_cr}
        \vert \mathrm{D}v_h\vert(\Omega)= \|\nabla_{\! h}v_h\|_{L^1(\Omega;\mathbb{R}^d)}+\|\jump{v_h}\|_{L^1(\mathcal{S}_h)}\,.
    \end{align}
	The Crouzeix--Raviart finite element space with homogeneous Dirichlet boundary condition~on~$\Gamma_D$ is defined by
	\begin{align*}
			\smash{\mathcal{S}^{1,\textit{cr}}_D(\mathcal{T}_h)}\coloneqq  \big\{v_h\in\smash{\mathcal{S}^{1,\textit{cr}}(\mathcal{T}_h)}\mid v_h(x_S)=0\text{ for all }S\in \mathcal{S}_h\cap \Gamma_D\big\}\,.
	\end{align*}
    A basis for  $\smash{\mathcal{S}^{1,\textit{cr}}(\mathcal{T}_h)}$ is given by functions $\varphi_S\hspace{-0.1em}\in\hspace{-0.1em} \smash{\mathcal{S}^{1,\textit{cr}}(\mathcal{T}_h)}$, $S\hspace{-0.1em}\in\hspace{-0.1em} \mathcal{S}_h$, satisfying~the~\mbox{Kronecker}~\mbox{property} $\varphi_S(x_{S'})=\delta_{S,S'}$ for all $S,S'\in \mathcal{S}_h$. A basis for  $\smash{\smash{\mathcal{S}^{1,\textit{cr}}_D(\mathcal{T}_h)}}$~is~given~by~${\varphi_S\in \smash{\mathcal{S}^{1,\textit{cr}}_D(\mathcal{T}_h)}}$,~${S\in \mathcal{S}_h\setminus\Gamma_D}$.

	\subsubsection{Raviart--Thomas element}
 
	\qquad The  Raviart--Thomas finite element space, cf.\ \cite{RT75},~consists~of~element-wise \hspace{-0.2mm}affine~\hspace{-0.2mm}vector~\hspace{-0.2mm}fields \hspace{-0.2mm}that \hspace{-0.2mm}have \hspace{-0.2mm}continuous \hspace{-0.2mm}constant \hspace{-0.2mm}normal \hspace{-0.2mm}components \hspace{-0.2mm}on \hspace{-0.2mm}inner~\hspace{-0.2mm}element~\hspace{-0.2mm}sides,~\hspace{-0.2mm}i.e.,\!\footnote{Here, for every inner side $S\in\mathcal{S}_h^{i}$, $\jump{y_h\cdot n}_S\coloneqq  \smash{y_h|_{T_+}\cdot n_{T_+}+y_h|_{T_-}\cdot n_{T_-}}$ on $S$, where $T_+, T_-\in \mathcal{T}_h$ satisfy $\smash{\partial T_+\cap\partial T_-=S}$ and for every $T\in \mathcal{T}_h$, $\smash{n_T\colon\partial T\to \mathbb{S}^{d-1}}$ denotes the outward unit normal vector field~to~$ T$, 
	and for every boundary side $\smash{S\in\mathcal{S}_h^{\partial}}$, $\smash{\jump{y_h\cdot n}_S\coloneqq  \smash{y_h|_T\cdot n}}$ on $S$, where $T\in \mathcal{T}_h$ satisfies $S\subseteq \partial T$ and $\smash{n\colon\partial\Omega\to \mathbb{S}^{d-1}}$ denotes the outward unit normal vector field to $\Omega$.}
	\begin{align*}
  \mathcal{R}T^0(\mathcal{T}_h)\coloneqq  \big\{y_h\in \mathcal{L}^1(\mathcal{T}_h)^d\mid &\,\smash{y_h|_T\cdot  n_T=\textup{const}\text{ on }\partial T\text{ for all }T\in \mathcal{T}_h\,,}\\ 
  &\smash{	\jump{y_h\cdot n}_S=0\text{ on }S\text{ for all }S\in \mathcal{S}_h^{i}\big\}\,.}
	\end{align*}
    Note that $\mathcal{R}T^{0}_N(\mathcal{T}_h)\subseteq W^\infty_N(\textup{div};\Omega)$. 
	The Raviart--Thomas finite element space with homogeneous normal component boundary condition on  $\Gamma_N$ is defined by
	\begin{align*}
		\smash{\mathcal{R}T^{0}_N(\mathcal{T}_h)}\coloneqq  \big\{y_h\in	\mathcal{R}T^0(\mathcal{T}_h)\mid y_h\cdot n=0\text{ on }\Gamma_N\big\}\,.
	\end{align*}
    A basis for  $\mathcal{R}T^0(\mathcal{T}_h)$ is given~by~vector fields $\psi_S\hspace{-0.1em}\in\hspace{-0.1em}  \mathcal{R}T^0(\mathcal{T}_h)$,  $S\hspace{-0.1em}\in\hspace{-0.1em} \mathcal{S}_h$,~satisfying~\mbox{Kronecker}~\mbox{property} $\psi_S|_{S'}\cdot n_{S'}=\delta_{S,S'}$ on $S'$ for all $S'\in \mathcal{S}_h$, where $n_S$  is the unit normal vector on $S$ pointing from $T_-$ to $T_+$ if $T_+\cap T_-=S\in \mathcal{S}_h$. A basis for $\smash{\mathcal{R}T^{0}_N(\mathcal{T}_h)}$ is given by $\psi_S\in \smash{\mathcal{R}T^{0}_N(\mathcal{T}_h)}$, ${S\in \mathcal{S}_h\setminus\Gamma_N}$.

	\subsubsection{Discrete integration-by-parts formula}
 
 	\qquad For every $v_h\in \mathcal{S}^{1,\textit{cr}}_D(\mathcal{T}_h)$ and ${y_h\in \mathcal{R}T^0_N(\mathcal{T}_h)}$, it holds the \textit{discrete integration-by-parts~formula}
	\begin{align}
		(\nabla_{\!h}v_h,\Pi_h y_h)_\Omega=-(\Pi_h v_h,\,\textup{div}\,y_h)_\Omega\,.\label{eq:pi0}
	\end{align}
 In~addition, cf.\ \cite[Section 2.4]{BW21}, 
 if a vector field $y_h\in \mathcal{L}^0(\mathcal{T}_h)^d$ satisfies for every $v_h\in \smash{\mathcal{S}^{1,cr}_D(\mathcal{T}_h)}$
 \begin{align*}
     (y_h,\nabla_{\!h} v_h)_{\Omega}=0\,,
 \end{align*}
 then, choosing $v_h=\varphi_S\in\mathcal{S}^{1,cr}_D(\mathcal{T}_h) $ for all $S\in \mathcal{S}_h\setminus \Gamma_D$, one finds that $y_h\in \mathcal{R}T^0_N(\mathcal{T}_h)$.
 Similarly, if a function $v_h\in \mathcal{L}^0(\mathcal{T}_h)$ satisfies for every $y_h\in \mathcal{R}T^0_N(\mathcal{T}_h)$
 \begin{align*}
     (v_h,\textup{div}\,y_h)_{\Omega}=0\,,
 \end{align*}
 then, choosing $y_h\hspace{-0.1em}=\hspace{-0.1em}\psi_S\hspace{-0.1em}\in\hspace{-0.1em}\mathcal{R}T^0_N(\mathcal{T}_h) $ for all $S\hspace{-0.1em}\in\hspace{-0.1em} \mathcal{S}_h\setminus \Gamma_N$, one finds that $v_h\hspace{-0.1em}\in\hspace{-0.1em} \mathcal{S}^{1,cr}_D(\mathcal{T}_h)$.~In~other~words, 
 we have the orthogonal (with respect to the inner product $(\cdot,\cdot)_{\Omega}$)  decompositions 
    \begin{align}
    \mathcal{L}^0(\mathcal{T}_h)^d&=\textup{ker}(\textup{div}|_{\smash{\mathcal{R}T^0_N(\mathcal{T}_h)}})\oplus \nabla_{\!h}(\mathcal{S}^{1,\textit{\textrm{cr}}}_D(\mathcal{T}_h))
       \,,\label{eq:decomposition.2}\\
        \mathcal{L}^0(\mathcal{T}_h)&=\textup{ker}(\nabla_{\!h}|_{\smash{\mathcal{S}^{1,cr}_D(\mathcal{T}_h)}})\oplus \textup{div}\,(\mathcal{R}T^0_N(\mathcal{T}_h)) \,.\label{eq:decomposition.1}
    \end{align}
	
	\section{Exact a posteriori error estimation for convex minimization problems} \label{sec:convex_min}

    \subsection{Continuous convex minimization problem and continuous convex duality}
	
    \hspace{5mm}Let $\phi\colon \mathbb{R}^d\to \mathbb{R}\cup\{+\infty\}$ be a proper, convex, and lower semi-continuous~function and let $\psi\colon \Omega\times\mathbb{R}\to \mathbb{R}\cup\{+\infty\}$ be a (Lebesgue) measurable function such that for a.e.\ $x\in \Omega$,~the~function $\psi(x,\cdot)\colon\Omega\times\mathbb{R}\to \mathbb{R}\cup\{+\infty\}$ is proper, convex, and lower semi-continuous. We~examine~the~convex minimization problem that seeks for a function $u\in W^{1,p}_D(\Omega) $, $p\in (1,\infty)$, that is~minimal for the functional $I\colon W^{1,p}_D(\Omega)\to \mathbb{R}\cup\{+\infty\}$, for every $v\in \smash{W^{1,p}_D(\Omega)}$ defined by
	\begin{align}
		I(v)\coloneqq \int_{\Omega}{\phi(\nabla v)\,\textup{d}x}+\int_{\Omega}{\psi(\cdot,v)\,\textup{d}x}\,.\label{primal}
	\end{align}
    In what follows, we refer to the minimization of $I \colon W^{1,p}_D(\Omega) \to\mathbb{R} \cup \{+\infty\}$ as the \textit{primal problem}.
	A \textit{(Fenchel) dual problem} to the minimization of \eqref{primal} consists in the maximization~of~the~functional $D\colon\smash{L^{p'}(\Omega;\mathbb{R}^d)}\to \mathbb{R} \cup \{ -\infty \}$, for every $y\in L^{p'}(\Omega;\mathbb{R}^d)$ defined by
	\begin{align}
		D(y)\coloneqq -\int_{\Omega}{\phi^*( y)\,\textup{d}x}-F^*(\textup{Div}\,y)\,,\label{dual}
	\end{align}
	where the distributional divergence $\textup{Div}\colon L^{p'}(\Omega;\mathbb{R}^d)\to (W^{1,p}_D(\Omega))^*$~for~every~$y\in \smash{L^{p'}(\Omega;\mathbb{R}^d)}$~and $v\hspace{-0.1em}\in\hspace{-0.1em} \smash{W^{1,p}_D(\Omega)}$ is defined by $\langle \textup{Div}\,y,v\rangle_{\smash{W^{1,p}_D(\Omega)}}\hspace{-0.1em}\coloneqq\hspace{-0.1em} -(y,\nabla v)_{\Omega}$ and
	$\smash{F^*\colon \hspace{-0.1em}L^{p'}(\Omega)\hspace{-0.1em}\to\hspace{-0.1em} \mathbb{R}\hspace{-0.1em}\cup\hspace{-0.1em}\{\pm\infty \}}$~denotes~the Fenchel conjugate to $F\colon L^p(\Omega)\to \mathbb{R}\cup\{+\infty\}$, defined by $F(v)\coloneqq \int_{\Omega}{\psi(\cdot,v)\,\textup{d}x}$~for~all~${ v\in L^p(\Omega)}$. Note~that~for~every ${y\hspace{-0.1em}\in \hspace{-0.1em}\smash{W^{p'}_N(\textup{div};\Omega)}}$, we have that $\langle \textup{Div}\,y,v\rangle_{\smash{W^{1,p}_D(\Omega)}}\hspace{-0.1em}=\hspace{-0.1em}(\textup{div}\,y, v)_{\Omega}$~for~all~${v\hspace{-0.1em}\in\hspace{-0.1em} W^{1,p}_D(\Omega)}$ and, thus, the representation
	\begin{align}
	D(y)=-\int_{\Omega}{\phi^*( y)\,\textup{d}x}-\int_{\Omega}{\psi^*(\cdot,\textup{div}\,y)\,\textup{d}x}\,.\label{eq:explicit_representation}
	\end{align}
    A \textit{weak duality relation} applies, cf.\  \cite[Proposition 1.1, {p.\ 48}]{ET99}, i.e.,
	\begin{align}
		\inf_{v\in W^{1,p}_D(\Omega)}{I(v)}\ge \sup_{y\in L^{p'}(\Omega;\mathbb{R}^d)}{D(y)}\,.\label{weak_duality}
	\end{align}
    In what follows, we
	always assume that $\phi\colon \mathbb{R}^d\to \mathbb{R}\cup\{+\infty\}$ and $\psi\colon \Omega\times\mathbb{R}\to \mathbb{R}\cup\{+\infty\}$ are such that \eqref{primal} admits at least one minimizer $u\in W^{1,p}_D(\Omega) $, called the \textit{primal solution},  \eqref{dual} at least one maximizer $z\in L^{p'}(\Omega;\mathbb{R}^d) $, called the \textit{dual solution}, and that a \textit{strong~duality~relation}~applies,~i.e.,
	\begin{align}
		I(u)= D(z)\,.\label{strong_duality}
	\end{align}
	By the Fenchel--Young inequality (cf.\  \eqref{eq:fenchel_young_ineq}),  \eqref{strong_duality} is equivalent to 
	the \textit{convex~optimality~relations}
	\begin{align}
        z\cdot\nabla u&=\phi^*(z)+\phi(\nabla u)\quad\textup{ a.e. in }\Omega\,,\label{optimality_relations.1}\\
        \textup{Div}\,z&\in \partial F(u)\,. \label{optimality_relations.2}
	\end{align}
    If $z\in \smash{W^{p'}_N(\textup{div};\Omega)}$, then the convex optimality relation \eqref{optimality_relations.2} is equivalent to 
    \begin{align}
        \textup{div}\,z\, u=\psi^*(\cdot,\textup{div}\,z)+\psi(\cdot, u)\quad\textup{ a.e. in }\Omega\,. 	\label{optimality_relations.3}
    \end{align}
    If $\phi\in C^1(\mathbb{R}^d)$, 
	then, by the Fenchel--Young identity~(cf.~\eqref{eq:fenchel_young_id}), \eqref{optimality_relations.1} is equivalent to 
	\begin{align}
	z= D\phi(\nabla u)\quad\textup{ in }\smash{L^{p'}(\Omega;\mathbb{R}^d)}\,.\label{optimality_relations.4}
	\end{align}
	Similarly, if $z\in \smash{W^{p'}_N(\textup{div};\Omega)}$ and 
	 $\psi(x,\cdot)\in C^1(\mathbb{R})$ for a.e.\ $x\in \Omega$, 
	 then \eqref{optimality_relations.3} is equivalent to 
	\begin{align}
	\textup{div}\,z=D\psi(\cdot, u)\quad\textup{ in }\smash{L^{p'}(\Omega)}\,.\label{optimality_relations.5}
	\end{align}

 The convex duality relations \eqref{weak_duality}--\eqref{optimality_relations.5} motivate introducing the \textit{primal-dual error estimator} $\eta^2\colon W^{1,p}_D(\Omega)\times L^{p'}(\Omega;\mathbb{R}^d)\to [0,+\infty]$, for every 
 $v\in W^{1,p}_D(\Omega)$ and $y\in L^{p'}(\Omega;\mathbb{R}^d)$ defined by\vspace{-1mm}\enlargethispage{5mm}
 \begin{align}
  \eta^2(v,y)\coloneqq I(v)-D(y)\,.\label{def:eta}
 \end{align}
Note that the sign of the estimator \eqref{def:eta} is a consequence of the weak duality relation \eqref{weak_duality}.

 Together with the optimal convexity measures (cf.\ Definition \ref{def:convexity_measure_optimal}) $\rho_I^2\colon W^{1,p}_D(\Omega)^2\to [0,+\infty]$ of \eqref{primal} at a primal solution $u\in W^{1,p}_D(\Omega)$ and $\rho_{-D}^2\colon \smash{L^{p'}(\Omega;\mathbb{R}^d)}\to [0,+\infty]$ of the negative of \eqref{dual} at a dual solution $z\in \smash{L^{p'}(\Omega;\mathbb{R}^d)}$, we arrive at the following explicit~a~posteriori~error~representation.\enlargethispage{3mm}
  
 \begin{theorem}[Explicit  (a posteriori) error representation]\label{thm:main}
 The  following statements apply:
 \begin{itemize}[noitemsep,topsep=2pt,leftmargin=!,labelwidth=\widthof{(ii)}]
     \item[(i)] For every $v\in W^{1,p}_D(\Omega)$ and $y\in  \smash{L^{p'}(\Omega;\mathbb{R}^d)}$, we have that
  \begin{align*}
        \smash{\rho^2_I(v,u)+\rho^2_{-D}(y,z)=\eta^2(v,y)}\,.
  \end{align*}
    \item[(ii)] For every $v\in \smash{ W^{1,p}_D(\Omega)}$ and $y\in  \smash{W^{p'}_N(\textup{div};\Omega)}$, we have that 
  \begin{align}\label{eta_explicit_representation}
      \eta^2(v,y)&= \int_{\Omega}{\phi(\nabla v)-\nabla v\cdot y+\phi^*(y)\,\mathrm{d}x}+\int_{\Omega}{\psi(\cdot, v)- v\,\mathrm{div}\,y+\psi^*(\cdot,\mathrm{div}\,y)\,\mathrm{d}x}\,.
  \end{align}
 \end{itemize}
 \end{theorem}

 \begin{remark}
     \begin{itemize}[noitemsep,topsep=2pt,leftmargin=!,labelwidth=\widthof{(ii)}]
     \item[(i)] By the Fenchel--Young inequality \eqref{eq:fenchel_young_ineq}, the integrands in the representation \eqref{eta_explicit_representation}, are non-negative and, thus, suitable as local refinement indicators.

     \item[(ii)] Appealing to Remark \ref{rem:convexity_measure_optimal}, from Theorem \ref{thm:main} (i), for every $v\in W^{1,p}_D(\Omega)$ and $y\in L^{p'}(\Omega;\mathbb{R}^d)$, it follows that 
     $\sigma_I^2(v,u)+\sigma_{-D}^2(y,z)\leq \eta^2(v,y)$.
     \end{itemize}
 \end{remark}

 \begin{proof}[Proof (of Theorem \ref{thm:main}).] \textit{ad (i).} Due to $I(u)=D(z)$, cf.\  \eqref{strong_duality},  Definition \ref{def:convexity_measure_optimal}, and \eqref{def:eta}, 
 for every $v\in W^{1,p}_D(\Omega)$~and $y\in L^{p'}(\Omega;\mathbb{R}^d)$, we have that 
  \begin{align*}
  \smash{\rho^2_I(v,u)+\rho^2_{-D}(y,z)=I(v)-I(u)+D(z)-D(y)=\eta^2(v,y)}\,.
  \end{align*}
  \textit{ad (ii).} Using \eqref{primal}, \eqref{eq:explicit_representation}, and integration-by-parts, we conclude that \eqref{eta_explicit_representation} applies. 
 \end{proof}
\begin{remark}[Examples]\label{rmk:examples}
 \begin{itemize}[noitemsep,topsep=2pt,leftmargin=!,labelwidth=\widthof{(iii)}]
  \item[(i)] In the \textup{$p$-Dirichlet problem}, cf. \cite{DR07,DK08}, i.e., ${\phi\coloneqq  \smash{\frac{1}{p}\vert \cdot\vert^p}\in C^1(\mathbb{R})}$, $p\hspace{-0.1em}\in\hspace{-0.1em} (1,\infty)$, and $\psi\hspace{-0.1em}\coloneqq \hspace{-0.1em} (\smash{(t,x)^\top}\mapsto -f(x)t)\colon \Omega\times \mathbb{R}\hspace{-0.1em}\to\hspace{-0.1em} \mathbb{R}$, where $f\hspace{-0.1em}\in\hspace{-0.1em} L^{p'}(\Omega)$,~cf.~\cite{dr-nafsa},~we~have~that
  \begin{align*}
   \smash{\rho^2_I(v,u)\sim \|F(\nabla v)-F(\nabla u)\|_{L^2(\Omega;\mathbb{R}^d)}^2\,,\qquad
   \rho^2_{-D}(y,z)\sim \|F^*(y)-F^*(z)\|_{L^2(\Omega;\mathbb{R}^d)}^2}\,,
  \end{align*}
  where $F,F^*\colon \hspace{-0.05em}\mathbb{R}^d\hspace{-0.1em}\to\hspace{-0.1em} \mathbb{R}^d$ for every $a\hspace{-0.1em}\in\hspace{-0.1em} \mathbb{R}^d$ are defined by $F(a)\hspace{-0.1em}\coloneqq \hspace{-0.1em}\smash{\vert a\vert^{\frac{p-2}{2}}}a$ and ${F^*(a)\hspace{-0.1em}\coloneqq \hspace{-0.1em} \smash{\vert a\vert^{\frac{p'-2}{2}}}a}$.

  \item[(ii)] In the \textup{obstacle problem}, cf.\  \cite{BK22Obstacle}, i.e., $\phi\coloneqq  \frac{1}{2}\vert \cdot\vert^2\in C^1(\mathbb{R})$ and $\psi\coloneqq  ({(t,x)^\top\mapsto -f(x)t}+I_{\chi(x)}(t))\colon \Omega\times \mathbb{R}\to \mathbb{R}\cup\{+\infty\}$, where $f\in L^2(\Omega)$ and $\chi\in W^{1,2}(\Omega)$ with $\chi\leq 0$~on~$\Gamma_D$,~cf.~\cite{BK22Obstacle}, where $I_{\chi(x)}(t)\coloneqq 0$ if $t\ge 0$ and $I_{\chi(x)}(t)\coloneqq +\infty$ else, we have that
  \begin{align*}
    \smash{\rho^2_I(v,u)= \tfrac{1}{2}\|\nabla v-\nabla u\|_{L^2(\Omega;\mathbb{R}^d)}^2+\langle -\Lambda,v-u\rangle_{W^{1,2}_D(\Omega)}\,,\qquad
   \rho^2_{-D}(y,z)\ge  \tfrac{1}{2}\|y-z\|_{L^2(\Omega;\mathbb{R}^d)}^2}\,,
  \end{align*}
  where $\Lambda\in (W^{1,2}_D(\Omega))^*$ is defined by $\langle \Lambda,v\rangle_{W^{1,2}_D(\Omega)}\coloneqq (f,v)_{\Omega}-(\nabla u,\nabla v)_{\Omega}$ for all $v\in W^{1,2}_D(\Omega)$.

  \item[(iii)] In an \textup{optimal design  problem}, cf.\ \cite{CL15}, i.e., $\phi\coloneqq \zeta\circ \vert\cdot\vert\in C^1(\mathbb{R})$, where $\zeta(0)\coloneqq 0$, $\zeta'(t)\coloneqq \mu_2 t$ if $t\in [0,t_1]$, $\zeta'(t)\coloneqq \mu_2 t_1$ if $t\in [t_1,t_2]$, and $\zeta'(t)\coloneqq\mu_1 t$ if $t\in [t_2,+\infty)$ for some $0<t_1<t_2$ and $0<\mu_1<\mu_2$ with $t_1\mu_2=t_2\mu_1$, and ${\psi\coloneqq  ((t,x)^\top\mapsto -f(x)t)\colon \Omega\times \mathbb{R}\to \mathbb{R}}$, where $f\in L^2(\Omega)$, cf.\ 
  \cite[Lemma 3.4]{CL15},
  we have that
  \begin{align*}
  \smash{ \rho^2_I(v,u)\ge  \tfrac{1}{2\mu}\|D\phi(\nabla v)-D\phi(\nabla u)\|_{L^2(\Omega;\mathbb{R}^d)}^2\,,\qquad
   \rho^2_{-D}(y,z)\ge  \tfrac{1}{2\mu}\|y-z\|_{L^2(\Omega;\mathbb{R}^d)}^2}\,.
  \end{align*}
  \item[(iv)] In the \textup{Rudin--Osher--Fatemi (ROF) problem}, cf.\ \cite{ROF92}, i.e., 
    $\phi\coloneqq \vert\cdot\vert\in C^0(\mathbb{R})$ and $\psi\coloneqq  ((t,x)^\top\mapsto \frac{\alpha}{2}(t-g(x))^2)\colon \Omega\times \mathbb{R}\to \mathbb{R}$, where $g\in L^2(\Omega)$, cf.\ \cite[Lemma 10.2]{Bar15}, we have that
  \begin{align*}
  \smash{ \rho^2_I(v,u)\ge  \tfrac{\alpha}{2}\|v-u\|_{L^2(\Omega)}^2\,,\qquad
   \rho^2_{-D}(y,z)\ge  \tfrac{1}{2\alpha }\|\textup{div}\,y-\textup{div}\,z\|_{L^2(\Omega)}^2}\,.
  \end{align*}
 \end{itemize}
 \end{remark}

Since the dual problem to the minimization of the negative of \eqref{dual}, in turn, consists in the maximization of the negative of \eqref{primal},
 the roles of the primal problem and the dual problem may be interchanged. An advantage of Theorem \ref{thm:main} consists in the fact that it yields reliable and efficient a posteriori error estimators for both the primal problem and the dual problem, i.e.,\enlargethispage{7.5mm}

\begin{remark}[Reliability and efficiency]
 Theorem \ref{thm:main} also shows that for each $y\in L^{p'}(\Omega;\mathbb{R}^d)$, the estimator $\eta^2_{I,y}\coloneqq  (v\mapsto \eta^2(v,y))\colon W^{1,p}_D(\Omega)\to [0,+\infty]$ 
 satisfies 
 \begin{align}\label{eq:a_posteriori_primal}
 \smash{\rho^2_I(v,u)+\rho^2_{-D}(y,z)=\eta^2_{I,y}(v)}\,,
 \end{align}
 and for each $v\in W^{1,p}_D(\Omega)$, the estimator $\eta^2_{-D,v}\coloneqq  (y\mapsto \eta^2(v,y))\colon L^{p'}(\Omega;\mathbb{R}^d)\to [0,+\infty]$~\mbox{satisfies}
 \begin{align}\label{eq:a_posteriori_dual}
 \smash{\rho^2_I(v,u)+\rho^2_{-D}(y,z)=\eta^2_{-D,v}(y)}\,.
 \end{align}
\end{remark}

For the a posteriori error estimators \eqref{eq:a_posteriori_primal} and \eqref{eq:a_posteriori_dual} for being numerically practicable, it is necessary to have a 
computationally cheap way to obtain sufficiently accurate approximation of the dual solution (for \eqref{eq:a_posteriori_primal}) and/or of the primal solution
(for \eqref{eq:a_posteriori_dual}), respectively. In Section~\ref{sec:discrete_duality}, resorting to (discrete) convex duality relations between a non-conforming Crouzeix--Raviart approximation of the primal problem and a  Raviart--Thomas approximation of the dual problem, we arrive at discrete reconstruction formulas, called \textit{generalized Marini formula},~cf.~\cite{Mar85,Bar21}.\enlargethispage{9mm}

 \subsection{Discrete convex minimization problem and discrete convex duality}\label{sec:discrete_duality}

 \hspace{5mm}Let $\psi_h\colon \Omega\times\mathbb{R}\to \mathbb{R}\cup\{+\infty\}$ denote a suitable approximation\footnote{We refrain from being too precise concerning
 what we mean with \textit{approximation} to allow~for~more~flexibility. Assumptions on both $\phi\colon \mathbb{R}^d\to \mathbb{R}\cup\{+\infty\}$ and $\psi_h\colon\Omega\times\mathbb{R}\to \mathbb{R}\cup\{+\infty\}$, $h>0$, that imply, e.g., $\Gamma$-convergence results can be found in \cite[Proposition 3.3]{Bar21}.} of $ \psi\colon \Omega\times\mathbb{R}\to \mathbb{R}\cup\{+\infty\}$ such that $\psi_h(\cdot,t)\in \mathcal{L}^0(\mathcal{T}_h)$ for all $t\in \mathbb{R}$  and for a.e.\ $x\in \Omega$, $\psi_h(x,\cdot)\colon \Omega\times\mathbb{R}\to \mathbb{R}\cup\{+\infty\}$ is a proper, convex,  and lower semi-continuous functional. Then, we examine the (discrete) convex minimization problem that seeks for a function  $u_h^{\textit{cr}}\in \mathcal{S}^{1,\textit{cr}}_D(\mathcal{T}_h)$ that is minimal for the functional $I_h^{\textit{cr}}\colon \smash{\mathcal{S}^{1,\textit{cr}}_D(\mathcal{T}_h)}\to \mathbb{R}\cup\{+\infty\}$, for every $v_h\in \smash{\mathcal{S}^{1,\textit{cr}}_D(\mathcal{T}_h)}$ defined by
	\begin{align}
		I_h^{\textit{cr}}(v_h)\coloneqq \int_{\Omega}{\phi(\nabla_{\! h} v_h)\,\textup{d}x}+\int_{\Omega}{\psi_h(\cdot,\Pi_h v_h)\,\textup{d}x}\,.\label{discrete_primal}
	\end{align}
	In what follows, we refer the minimization of $I_h^{\textit{cr}}\colon \smash{\mathcal{S}^{1,\textit{cr}}_D(\mathcal{T}_h)}\to \mathbb{R}\cup\{+\infty\}$ to as the \textit{discrete primal problem}.
	In \cite{Bar21,BW21}, it is shown that the corresponding (Fenchel) dual problem to the minimization of \eqref{discrete_primal} 
	consists in the maximization of $D_h^{\textit{rt}}\colon \mathcal{R}T^0_N(\mathcal{T}_h)\to \mathbb{R}\cup\{-\infty\}$,~for~every~${y_h\in\mathcal{R}T^0_N(\mathcal{T}_h)}$ defined by
	\begin{align}
		D_h^{\textit{rt}}(y_h)\coloneqq-\int_{\Omega}{\phi^*(\Pi_h y_h)\,\textup{d}x}-\int_{\Omega}{\psi_h^*(\cdot,\textup{div}\,y_h)\,\textup{d}x}\,.\label{discrete_dual}
	\end{align} 
	A \textit{discrete weak duality relation}, cf. \cite[Proposition 3.1]{Bar21}, applies
	\begin{align}
		\inf_{v_h\in \mathcal{S}^{1,\textit{cr}}_D(\mathcal{T}_h)}{I_h^{\textit{cr}}(v_h)}
		\ge \sup_{y_h\in \mathcal{R}T^0_N(\mathcal{T}_h)}{D_h^{\textit{rt}}(y_h)}\,.\label{eq:discrete_weak_duality}
	\end{align}
    We will always assume that $\phi\colon \mathbb{R}^d\to \mathbb{R}\cup\{+\infty\}$ and $\psi_h\colon \Omega\times \mathbb{R}\to \mathbb{R}\cup\{+\infty\}$ are such that \eqref{discrete_primal} admits at least one minimizer $u_h^{\textit{cr}}\in \mathcal{S}^{1,\textit{cr}}_D(\mathcal{T}_h)$, called the \textit{discrete primal solution},
    \eqref{discrete_dual} admits at least one maximizer  $z_h^{\textit{rt}}\in \mathcal{R}T^0_N(\mathcal{T}_h)$, called the \textit{discrete dual solution}, and that~a~\textit{discrete~strong duality relation} applies, i.e.,
    \begin{align}
	    I_h^{\textit{cr}}(u_h^{\textit{cr}})=D_h^{\textit{rt}}(z_h^{\textit{rt}})\,.\label{eq:discrete_strong_duality}
    \end{align}
    By the Fenchel--Young identity (cf.~\eqref{eq:fenchel_young_id}), \eqref{eq:discrete_strong_duality} is equivalent to the  \textit{discrete convex optimality relations}
	\begin{alignat}{2}
		\Pi_h z_h^{\textit{rt}}\cdot \nabla_{\! h} u_h^{\textit{cr}}&=\phi^*(\Pi_hz_h^{\textit{rt}})+\phi(\nabla_{\! h} u_h^{\textit{cr}})&&\quad\text{ a.e. in }\Omega\,,\label{eq:discrete_optimality_relations.1}\\
		\textup{div}\,z_h^{\textit{rt}}\,\Pi_hu_h^{\textit{cr}}& =\psi_h^*(\cdot,	\textup{div}\,z_h^{\textit{rt}})+\psi_h(\cdot,\Pi_hu_h^{\textit{cr}})&&\quad\text{ a.e. in }\Omega\,.
        \label{eq:discrete_optimality_relations.2}
	\end{alignat}
	If $\phi\in  C^1(\mathbb{R}^d)$, then, by the Fenchel--Young identity (cf.\  \eqref{eq:fenchel_young_id}), \eqref{eq:discrete_optimality_relations.1} is equivalent~to
	\begin{align}
			\Pi_h z_h^{\textit{rt}}=D\phi(\nabla_{\! h} u_h^{\textit{cr}})\quad\text{ in }\mathcal{L}^0(\mathcal{T}_h)^d\,,\label{eq:discrete_optimality_relations.3}
	\end{align}
    and if $\phi^*\in  C^1(\mathbb{R}^d)$, then, by the Fenchel--Young identity (cf.\  \eqref{eq:fenchel_young_id}), \eqref{eq:discrete_optimality_relations.2} is equivalent~to
	\begin{align}
			\nabla_{\! h} u_h^{\textit{cr}}=D\phi^*(\Pi_h z_h^{\textit{rt}})\quad\text{ in }\mathcal{L}^0(\mathcal{T}_h)^d\,.\label{eq:discrete_optimality_relations.4}
	\end{align}
	Similarly, if $\psi_h(x,\cdot)\in C^1(\mathbb{R})$ for a.e.\ $x\in \Omega$, then \eqref{eq:discrete_optimality_relations.2} is equivalent to
	\begin{align}
			\textup{div}\,z_h^{\textit{rt}}=D\psi_h(\cdot,\Pi_hu_h^{\textit{cr}})\quad\text{ in }\mathcal{L}^0(\mathcal{T}_h)\,,\label{eq:discrete_optimality_relations.5}
	\end{align}
    and if $\psi_h^*(x,\cdot)\in C^1(\mathbb{R})$ for a.e.\ $x\in \Omega$, then \eqref{eq:discrete_optimality_relations.2} is equivalent to
    \begin{align}
			\Pi_hu_h^{\textit{cr}}=D\psi_h^*(\cdot,\textup{div}\,z_h^{\textit{rt}})\quad\text{ in }\mathcal{L}^0(\mathcal{T}_h)\,.\label{eq:discrete_optimality_relations.6}
	\end{align}\newpage
    
    The relations \eqref{eq:discrete_optimality_relations.3}--\eqref{eq:discrete_optimality_relations.6} motivate the following discrete recontruction formulas for a discrete dual solution $z_h^{\textit{rt}}\in \mathcal{R}T^0_N(\mathcal{T}_h)$ from a discrete primal solution  $u_h^{\textit{cr}}\in \mathcal{S}^{1,cr}_D(\mathcal{T}_h)$~and~vice~versa, called \textit{generalized Marini formulas}, cf.\  \cite{Mar85,Bar21}.

    \begin{proposition}[Generalized Marini formulas]\label{prop:gen_marini}
        The following statements apply:
        \begin{itemize}[noitemsep,topsep=2pt,leftmargin=!,labelwidth=\widthof{(ii)}]
            \item[(i)] If $\phi\in  C^1(\mathbb{R}^d)$ and $\psi_h(x,\cdot)\in C^1(\mathbb{R})$ for a.e.\ $x\in \Omega$, then, given a minimizer $u_h^{\textit{cr}}\in \mathcal{S}^{1,cr}_D(\mathcal{T}_h)$~of \eqref{discrete_primal},
            a maximizer $z_h^{\textit{rt}}\in \mathcal{R}T^0_N(\mathcal{T}_h)$ of \eqref{discrete_dual} is given via
            	\begin{align}
	               z_h^{\textit{rt}}= D\phi(\nabla_{\! h} u_h^{\textit{cr}})+\frac{D\psi_h(\cdot, \Pi_hu_h^{\textit{cr}})}{d}\big(\textup{id}_{\mathbb{R}^d}-\Pi_h\textup{id}_{\mathbb{R}^d}\big)\quad\text{ in }\mathcal{R}T^0_N(\mathcal{T}_h)\,,\label{eq:reconstruction_formula.1}
	           \end{align}
            a discrete strong duality relation applies, i.e., \eqref{eq:discrete_strong_duality}.
            \item[(ii)] If $\phi^*\in  C^1(\mathbb{R}^d)$ and $\psi_h^*(x,\cdot)\in C^1(\mathbb{R})$ for a.e.\ $x\in \Omega$, then, given a maximizer $z_h^{\textit{rt}}\in \mathcal{R}T^0_N(\mathcal{T}_h)$ of \eqref{discrete_dual}, a minimizer $u_h^{\textit{cr}}\in \mathcal{S}^{1,cr}_D(\mathcal{T}_h)$ of \eqref{discrete_primal} is given via
            	\begin{align}
                   u_h^{\textit{cr}} = D\psi_h^*(\cdot,\textup{div}\,z_h^{\textit{rt}})+ D\phi^*(\Pi_h z_h^{\textit{rt}})\cdot\big(\textup{id}_{\mathbb{R}^d}-\Pi_h\textup{id}_{\mathbb{R}^d}\big)
	              \quad\text{ in }\mathcal{S}^{1,cr}_D(\mathcal{T}_h)\,,\label{eq:reconstruction_formula.2}
	           \end{align}
            a discrete strong duality relation applies, i.e., \eqref{eq:discrete_strong_duality}.
        \end{itemize}        
    \end{proposition}

     \begin{remark}
        It is possible to derive reconstructions formulas similar to \eqref{eq:reconstruction_formula.1} and \eqref{eq:reconstruction_formula.2} under weak conditions, e.g., resorting to a regularization argument (cf. Proposition \ref{prop:discrete_convex_duality})  or given discrete Lagrange multipliers (cf. \cite[Proposition 3.3]{BK22Obstacle}).
     \end{remark}

    \begin{proof}
        \textit{ad (i).} See \cite[Proposition 3.1]{Bar21}.\enlargethispage{5mm}

        \textit{ad (ii).} By definition, it holds $ u_h^{\textit{cr}}\in \mathcal{L}^1(\mathcal{T}_h)$ and the discrete convex optimality relation \eqref{eq:discrete_optimality_relations.6} is satisfied.
        Since $z_h^{\textit{rt}}\in \mathcal{R}T^0_N(\mathcal{T}_h)$ is maximal for \eqref{discrete_dual} as well as $\phi^*\in  C^1(\mathbb{R}^d)$~and~${\psi_h^*(x,\cdot)\in C^1(\mathbb{R})}$ for a.e.\ $x\in \Omega$, for every $y_h\in\mathcal{R}T^0_N(\mathcal{T}_h)$, we have that
        \begin{align}\label{prop:gen_marini.1}
            (D\phi^*(\Pi_h z_h^{\textit{rt}}),\Pi_hy_h)_{\Omega}+(D\psi_h^*(\cdot,\textup{div}\,z_h^{\textit{rt}}),\textup{div}\,y_h)_{\Omega}=0\,.
        \end{align}
        In \hspace{-0.1mm}particular, \hspace{-0.1mm}\eqref{prop:gen_marini.1} \hspace{-0.1mm}implies \hspace{-0.1mm}that \hspace{-0.1mm}$D\phi^*(\Pi_h z_h^{\textit{rt}})\hspace{-0.1em}\in\hspace{-0.1em} (\textup{ker}(\textup{div}|_{\mathcal{R}T^0_N(\mathcal{T}_h)}))^\perp$.
       \hspace{-0.1mm}Appealing \hspace{-0.1mm}to \hspace{-0.1mm}\cite[Lemma~\hspace{-0.1mm}2.4]{CP20}, it holds
        $(\textup{ker}(\textup{div}|_{\mathcal{R}T^0_N(\mathcal{T}_h)}))^\perp=\nabla_{\!h}(\mathcal{S}^{1,cr}_D(\mathcal{T}_h))$. Therefore, there exists 
        $v_h\in \mathcal{S}^{1,cr}_D(\mathcal{T}_h)$ such that  
        \begin{align}\label{prop:gen_marini.2}
            \nabla_{\!h} v_h= D\phi^*(\Pi_h z_h^{\textit{rt}})\quad\text{ in }\mathcal{L}^0(\mathcal{T}_h)^d\,.
        \end{align}
        Hence, for every $y_h\in\mathcal{R}T^0_N(\mathcal{T}_h)$, resorting to the discrete integration-by-parts~formula~\eqref{eq:pi0}, \eqref{prop:gen_marini.2}, \eqref{prop:gen_marini.1}, and \eqref{eq:discrete_optimality_relations.6}, we~find~that
        \begin{align*}
            \begin{aligned}
               (\Pi_hv_h-\Pi_h u_h^{cr},\textup{div}\,y_h)_{\Omega}
               =- (D\phi^*(\Pi_h z_h^{\textit{rt}}),\Pi_hy_h)_{\Omega}-(D\psi_h^*(\cdot,\textup{div}\,z_h^{\textit{rt}}),\textup{div}\,y_h)_{\Omega}=0\,.
            \end{aligned}
        \end{align*}
        In other words, for every $y_h\in\mathcal{R}T^0_N(\mathcal{T}_h)$, we have that
        \begin{align}\label{prop:gen_marini.2.1}
            \begin{aligned}
                ( v_h-u_h^{cr},\textup{div}\,y_h)_{\Omega}= (\Pi_h v_h-\Pi_h u_h^{cr},\textup{div}\,y_h)_{\Omega}=0\,.
            \end{aligned}
        \end{align}
        On the other hand, we have that $\nabla_{\! h}(v_h-u_h^{cr})=0$ in $\mathcal{L}^0(\mathcal{T}_h)^d$, i.e., $v_h-u_h^{cr}\in \mathcal{L}^0(\mathcal{T}_h)$.
        Therefore, \eqref{prop:gen_marini.2.1} in conjunction with \eqref{eq:decomposition.1} implies that
        $v_h-u_h^{cr}\in (\textup{div}\,(\mathcal{R}T^0_N(\mathcal{T}_h)))^{\perp}=\textup{ker}(\nabla_{\!h}|_{\mathcal{S}^{1,cr}_D(\mathcal{T}_h)})$. As a result, due to $v_h\in \smash{\mathcal{S}^{1,cr}_D(\mathcal{T}_h)}$, we conclude that $u_h^{cr}\in \smash{\mathcal{S}^{1,cr}_D(\mathcal{T}_h)}$ with
        \begin{align}\label{prop:gen_marini.3}
            \begin{aligned}
			     \nabla_{\! h} u_h^{\textit{cr}}&=D\phi^*(\Pi_h z_h^{\textit{rt}})&&\quad\text{ in }\mathcal{L}^0(\mathcal{T}_h)^d\,,\\
			     \Pi_hu_h^{\textit{cr}}&=D\psi_h^*(\cdot,\textup{div}\,z_h^{\textit{rt}})&&\quad\text{ in }\mathcal{L}^0(\mathcal{T}_h)\,.
            \end{aligned}
	    \end{align}
        By the Fenchel--Young identity, cf.\ \eqref{eq:fenchel_young_id}, \eqref{prop:gen_marini.3} is equivalent to 
        \begin{align}\label{prop:gen_marini.4}
            \begin{aligned}
                	\Pi_h z_h^{\textit{rt}}\cdot \nabla_{\! h} u_h^{\textit{cr}}&=\phi^*(\Pi_hz_h^{\textit{rt}})+\phi(\nabla_{\! h} u_h^{\textit{cr}})&&\quad\text{ a.e. in }\Omega\,,\\
				    \textup{div}\,z_h^{\textit{rt}}\,\Pi_hu_h^{\textit{cr}}& =\psi_h^*(\cdot,	\textup{div}\,z_h^{\textit{rt}})+\psi_h(\cdot,\Pi_hu_h^{\textit{cr}})&&\quad\text{ a.e. in }\Omega\,.
            \end{aligned}
        \end{align}
        Eventually, adding \eqref{prop:gen_marini.4}$_1$ and \eqref{prop:gen_marini.4}$_2$, subsequently, integration with respect to $x\in \Omega$, resorting~to the discrete integration-by-parts formula \eqref{eq:pi0}, and using the definitions \eqref{discrete_primal}~and~\eqref{discrete_dual}, we arrive at $I_h^{\textit{cr}}(u_h^{\textit{cr}})=D_h^{\textit{rt}}(z_h^{\textit{rt}})$, 
        which, appealing to the discrete weak duality relation \eqref{eq:discrete_weak_duality}, implies that $u_h^{\textit{cr}}\in \mathcal{S}^{1,cr}_D(\mathcal{T}_h)$ is minimal for \eqref{discrete_primal}.
    \end{proof}

\newpage
 \section{Application to the Rudin--Osher--Fatemi (ROF) model}\label{sec:ROF}

 \hspace{5mm}In this section, we transfer the concepts derived in Section \ref{sec:convex_min} to the non-differentiable Rudin--Osher--Fatemi~(ROF)~model, cf.\  \cite{ROF92}. The approximation of the ROF model has been investigated by numerous authors: A priori error estimates has been derived in \cite{BNS15,CP20,Bar21,BTW21,BKROF22}.
 A posteriori error estimates and adaptivity results can be found in \cite{bartels15,FV04,BM20,BTW21,BW22}.\enlargethispage{7mm}
 
\subsection{The continuous Rudin--Osher--Fatemi (ROF) model}\label{subsec:continuous_ROF_model}
	
	\hspace{5mm}Given a function $g\in L^2(\Omega)$, i.e., the \textit{noisy image}, and a constant parameter $\alpha>0$,~the~\textit{fidelity parameter} the Rudin--Osher--Fatemi~(ROF) model, cf.\  \cite{ROF92}, consists in the minimization of the functional $I\colon BV(\Omega)\cap L^2(\Omega)\to \mathbb{R}$, for every $v\in BV(\Omega)\cap L^2(\Omega)$ defined by
	\begin{align}
		\smash{I(v)\coloneqq  \vert \textup{D}v\vert (\Omega)+\tfrac{\alpha}{2}\|v-g\|^2_{L^2(\Omega)}}\,.\label{ROF-primal}
	\end{align}
	In \cite[Theorem 10.5 \& Theorem 10.6]{Bar15}, it has been established that there exists~a~unique~minimizer $u\hspace{-0.1em}\in\hspace{-0.1em} BV(\Omega)\cap L^2(\Omega)$ 
 \hspace{-0.1mm}of \hspace{-0.1mm}\eqref{ROF-primal}.  
 \hspace{-0.1mm}Appealing \hspace{-0.1mm}to \hspace{-0.1mm}\cite[Theorem \hspace{-0.1mm}2.2]{HK04} \hspace{-0.1mm}or \hspace{-0.1mm}\cite[Section \hspace{-0.1mm}10.1.3]{Bar15},~\hspace{-0.1mm}the~\hspace{-0.1mm}\mbox{corresponding} (Fenchel) dual problem to the minimization of \eqref{ROF-primal} consists in the maximization of the functional $D\colon W^2_N(\textup{div};\Omega) \cap L^\infty(\Omega;\mathbb{R}^d)\to \mathbb{R} \cup \{-\infty\}$, for every $\smash{y\in W^2_N(\textup{div};\Omega) \cap L^\infty(\Omega;\mathbb{R}^d)}$~defined~by
	\begin{align}
		\smash{D(y)\coloneqq -I_{K_1(0)}(y) -\tfrac{1}{2\alpha}\|\textup{div}\,y+\alpha\,g\|_{L^2(\Omega)}^2+\tfrac{\alpha}{2}\|g\|_{L^2(\Omega)}^2}
  \,,\label{ROF-dual}
	\end{align}
    where $I_{K_1(0)}\colon L^\infty(\Omega;\mathbb{R}^d)\to \mathbb{R} \cup \{\infty\}$ is defined by $I_{K_1(0)}(y) \coloneqq  0$ if $y\in L^\infty(\Omega;\mathbb{R}^d)$ with $\vert y\vert\leq 1$ a.e.\ in $\Omega$ and $I_{K_1(0)}(y)\coloneqq  \infty $ else. Apart from that, in \cite[Theorem 2.2]{HK04}, it is shown~that~\eqref{ROF-dual}~admits a maximizer ${z\in W^2_N(\textup{div};\Omega)\cap L^\infty(\Omega;\mathbb{R}^d)}$ and that a strong duality relation applies, i.e.,\vspace{-0.25mm}
	\begin{align}
		\smash{I(u)=D(z)}\,.\label{ROF_strong_duality_principle}
	\end{align} 
	Appealing to \cite[Proposition 10.4]{Bar15}, \eqref{ROF_strong_duality_principle} is equivalent to\vspace{-0.25mm} the convex optimality relations
	\begin{alignat}{2}
		\textup{div}\,z&=\alpha\, (u-g)\quad\text{ in }L^2(\Omega)\,,\label{ROF_optimality-relations.1}
  \\ -(u,\textup{div}\,z)_{\Omega}&=\vert \textup{D}u\vert(\Omega)\,.\label{ROF_optimality-relations.2}
	\end{alignat}

    Next, if we introduce, by analogy with Section \ref{sec:convex_min}, the \textit{primal-dual error~estimator}
    $\eta^2\colon BV(\Omega)\times (W^2_N(\textup{div};\Omega)\cap L^\infty(\Omega;\mathbb{R}^d))\to [0,+\infty]$, for every $v\in BV(\Omega)$ and $y\in W^2_N(\textup{div};\Omega)\cap L^\infty(\Omega;\mathbb{R}^d)$  defined by\vspace{-1mm}
    \begin{align}
        \smash{\eta^2(v,y)\coloneqq I(v)-D(y)}\,,\label{def:eta_ROF}
    \end{align}
    then the concepts of Section \ref{sec:convex_min} can be transferred to the ROF model.\enlargethispage{5mm}

 \begin{theorem}[Explicit  (a posteriori) error representation]\label{thm:main_2} The following statements apply:
 \begin{itemize}[noitemsep,topsep=2pt,leftmargin=!,labelwidth=\widthof{(ii)}]
     \item[(i)] For every $v\in BV(\Omega)$ and $y\in W^2_N(\textup{div};\Omega)\cap L^\infty(\Omega;\mathbb{R}^d)$, we have that\vspace{-0.25mm}
  \begin{align*}
  \smash{\rho^2_I(v,u)+\rho^2_{-D}(y,z)=\eta^2(v,y)}\,.
  \end{align*}
    \item[(ii)] For every $v\in BV(\Omega)$ and $y\in W^2_N(\textup{div};\Omega)\cap L^\infty(\Omega;\mathbb{R}^d)$, we have that 
  \begin{align}\label{eta_explicit_representation_ROF}
     \smash{ \eta^2(v,y)= \vert \mathrm{D}v\vert(\Omega)+(\textup{div}\,y,v)_{\Omega}+\tfrac{1}{2\alpha}\|\textup{div}\,y-\alpha\,(v-g)\|_{L^2(\Omega)}^2+I_{K_1(0)}(y)}\,.
  \end{align} 
 \end{itemize}
 \end{theorem}

 \begin{proof}
 \textit{ad (i).} Due to $I(u)=D(z)$, cf.\  \eqref{ROF_strong_duality_principle}, Definition \ref{def:convexity_measure_optimal}, and \eqref{def:eta_ROF},
 for every $v\in BV(\Omega)$ and $y\in  W^2_N(\textup{div};\Omega)\cap L^\infty(\Omega;\mathbb{R}^d)$, we have that
  \begin{align*}
   \smash{\rho^2_I(v,u)+\rho^2_{-D}(y,z)=I(v)-I(u)+D(z)-D(y)=\eta^2(v,y)}\,.
  \end{align*}
  \textit{ad (ii).} For every $v\in BV(\Omega)$ and $y\in W^2_N(\textup{div};\Omega)\cap L^\infty(\Omega;\mathbb{R}^d)$, we have that 
 \begin{align*}
 \eta^2(v,y)&=\vert \mathrm{D}v\vert(\Omega)+(\textup{div}\,y,v)_{\Omega}+\tfrac{1}{2\alpha}\|\alpha\,(v-g)\|_{L^2(\Omega)}^2\\&\quad
 -\tfrac{1}{2\alpha}2(\textup{div}\,y,\alpha\,v)_{\Omega}+\tfrac{1}{2\alpha}\|\textup{div}\,y+\alpha\, g\|_{L^2(\Omega)}^2-\tfrac{\alpha}{2}\|g\|_{L^2(\Omega)^2}^2+I_{K_1(0)}(y)
 \\&=\vert \mathrm{D}v\vert(\Omega)+(\textup{div}\,y,v)_{\Omega}+\tfrac{\alpha}{2}\|v-g\|_{L^2(\Omega)}^2\\&\quad-
 \tfrac{1}{2\alpha}\|\textup{div}\,y-\alpha\,(v-g)\|_{L^2(\Omega)}^2-\tfrac{\alpha}{2}\|v-g\|_{L^2(\Omega)}^2+I_{K_1(0)}(y)\,,
 \end{align*}
 which yields the claimed representation.
 \end{proof}

Restricting \hspace{-0.2mm}the \hspace{-0.2mm}estimator \hspace{-0.2mm}\eqref{def:eta_ROF} \hspace{-0.2mm}to \hspace{-0.2mm}subclasses \hspace{-0.2mm}of \hspace{-0.2mm}$BV(\Omega)$ \hspace{-0.2mm}and \hspace{-0.2mm}$W^2_N(\textup{div};\Omega)\cap L^\infty(\Omega;\mathbb{R}^d)$,~\hspace{-0.2mm}\mbox{respectively}, for which an appropriate integration-by-parts formula apply, e.g., \eqref{eq:pi0}, it is possible to derive alternative representations of the estimator \eqref{def:eta_ROF}, whose integrands are point-wise non-negative~and, thus, suitable as local refinement indicators.\vspace{-0.25mm}

 \begin{remark}[Alternative representations of \eqref{def:eta_ROF} and local refinement indicators]\label{rem:representations}\hphantom{\qquad\qquad}\vspace{-0.25mm}
     \begin{itemize}[noitemsep,topsep=2pt,leftmargin=!,labelwidth=\widthof{(iii)}]
         \item[(i)] For every $v\in W^{1,1}(\Omega)$ and $y\in W^2_N(\textup{div};\Omega)\cap L^\infty(\Omega;\mathbb{R}^d)$, by integration-by-parts,~it~holds\vspace{-0.25mm}
         \begin{align*}
              \eta^2(v,y)=\|\nabla v\|_{L^1(\Omega;\mathbb{R}^d)}-(\nabla v,y)_{\Omega}+\tfrac{1}{2\alpha}\|\textup{div}\,y+\alpha\,(v-g)\|_{L^2(\Omega)}^2+I_{K_1(0)}(y)\ge 0\,.
         \end{align*}

         \item[(ii)] For every $T\in \mathcal{T}_h$, we define the local refinement indicator $\eta_T^2\colon W^{1,1}(\Omega)\times  W^2_N(\textup{div};\Omega)\cap L^\infty(\Omega;\mathbb{R}^d)\to [0,+\infty]$ for every $v\in W^{1,1}(\Omega)$ and $y\in  W^2_N(\textup{div};\Omega)\cap L^\infty(\Omega;\mathbb{R}^d)$ by\vspace{-0.25mm}
        \begin{align*}
            \eta^2_{T,W}(v,y)\coloneqq \|\nabla v\|_{L^1(T;\mathbb{R}^d)}-(\nabla v,y)_T+\tfrac{1}{2\alpha}\|\textup{div}\,y+\alpha\,(v-g)\|_{L^2(T)}^2+I_{K_1(0)}(y)\ge 0\,.
        \end{align*}

         \item[(iii)] For every $v_h\in \mathcal{S}^{1,cr}(\Omega)$ and $y_h\in \mathcal{R}T^0_N(\mathcal{T}_h)$, by the representation of the total variation of Crouzeix--Raviart functions \eqref{eq:total_variation_cr} and the discrete integration-by-parts formula~\eqref{eq:pi0},~it~holds\vspace{-0.25mm}
         \begin{align*}
              \eta^2(v_h,y_h)&=\|\nabla_{\! h} v_h\|_{L^1(\Omega;\mathbb{R}^d)}+\|\jump{v_h}\|_{L^1(\mathcal{S}_h)}-(\nabla_{\! h} v_h,\Pi_h y_h)_{\Omega}\\&\quad +\tfrac{1}{2\alpha}\|\textup{div}\,y_h+\alpha\,(v_h-g)\|_{L^2(\Omega)}^2+I_{K_1(0)}(y_h)\ge 0\,.
         \end{align*}

        \item[(iv)] For every $T\in \mathcal{T}_h$, we define the discrete local refinement indicator $\eta_{T,\textit{CR}}^2\colon {\mathcal{S}^{1,cr}(\mathcal{T}_h)\times \mathcal{R}T^0_N(\mathcal{T}_h)}$ $\to [0,+\infty]$ for every $v_h\in \mathcal{S}^{1,cr}(\mathcal{T}_h)$ and $y_h\in \mathcal{R}T^0_N(\mathcal{T}_h)$ by\vspace{-0.25mm}
        \begin{align*}
            \eta^2_{T,\textit{CR}}(v_h,y_h)&\coloneqq \|\nabla v_h\|_{L^1(T;\mathbb{R}^d)}+\sum_{S\in \mathcal{S}_h;S\subseteq T}{\|\jump{v_h}\|_{L^1(S)}}-(\nabla_{\! h} v_h,\Pi_h y_h)_T\\[-0.5mm]&\quad +\tfrac{1}{2\alpha}\|\textup{div}\,y_h+\alpha\,(v_h-g)\|_{L^2(T)}^2+I_{K_1(0)}(y_h)\ge 0\,.
        \end{align*}
     \end{itemize}
 \end{remark}

    We emphasize that the primal-dual error estimator \eqref{def:eta_ROF} and the representations \eqref{eta_explicit_representation_ROF} or in Remark \ref{rem:representations} (i) \& (ii) are  well-known, cf.\ \cite{bartels15,BM20,BTW21}. However, the combination of \eqref{def:eta_ROF} with the representation of the total variation of Crouzeix--Raviart functions \eqref{eq:total_variation_cr} and the discrete integration-by-parts formula \eqref{eq:pi0} in Remark \ref{rem:representations} (iii) \& (iv), to the best of the authors' knowledge, is new and leads to significantly improved experimental convergence rates of the corresponding adaptive mesh-refinement procedure compared to the contributions \cite{bartels15,BM20,BTW21}, cf. Section \ref{sec:experiments}.\vspace{-0.25mm}\enlargethispage{15mm} 

	\subsection{The discretized Rudin--Osher--Fatemi (ROF) model}\label{subsec:discrete_ROF_model}
	\hspace{5mm}Given $g\in L^2(\Omega)$ and $\alpha>0$, with $g_h\coloneqq  \Pi_hg\in\mathcal{L}^0(\mathcal{T}_h)$, the discretized ROF model, proposed in \cite{CP20}, consists in the minimization of $I^{cr}_h\colon \mathcal{S}^{1,\textit{cr}}(\mathcal{T}_h)\to \mathbb{R}$, for every $v_h\in \mathcal{S}^{1,\textit{cr}}(\mathcal{T}_h)$ defined by 
	\begin{align}
		I^{cr}_h(v_h)\coloneqq  \|\nabla_{\!h}v_h\|_{L^1(\Omega;\mathbb{R}^d)}+\tfrac{\alpha}{2}\|\Pi_hv_h-\alpha\, g_h\|^2_{L^2(\Omega)}\,.\label{discrete-ROF-primal}
	\end{align}
    Note that the functional \eqref{discrete-ROF-primal} defines a non-conforming approximation of the functional \eqref{ROF-primal}, as, e.g., jump terms of across inner element sides are not included. This, however, turned out to be essential in the derivation of optimal a priori error estimate in \cite{CP20,Bar21}.
    Since the functional \eqref{discrete-ROF-primal} is proper, strictly convex, weakly coercive, and lower semi-continuous,
        the direct method in the calculus of variations, cf.\ \cite{Dac08}, yields the existence of a unique minimizer $u_h^{cr}\in \mathcal{S}^{1,\textit{cr}}(\mathcal{T}_h)$, called the \textit{discrete primal solution}. Appealing to \cite{CP20,Bar21},  the corresponding (Fenchel) dual problem to the minimization of \eqref{discrete-ROF-primal} consists in the maximization of the functional $D_h^{rt}\colon\mathcal{R}T^0_N(\mathcal{T}_h)\to \mathbb{R}\cup\{-\infty\}$, for every $y_h\in \mathcal{R}T^0_N(\mathcal{T}_h)$ defined by 
	\begin{align}
		D_h^{rt}(y_h)\coloneqq  -I_{K_1(0)}(\Pi_hy_h)-\tfrac{1}{2\alpha}\|\textup{div}\,y_h+\alpha\,g_h\|_{L^2(\Omega)}^2+\tfrac{\alpha}{2}\|g_h\|_{L^2(\Omega)}^2\,.\label{discrete-ROF-dual}
	\end{align}
    Appealing \hspace{-0.1mm}to \hspace{-0.1mm}Theorem \hspace{-0.1mm}\ref{thm:discrete_ROF_strong_duality} \hspace{-0.1mm}(below), \hspace{-0.1mm}there \hspace{-0.1mm}exists \hspace{-0.1mm}a \hspace{-0.1mm}maximizer \hspace{-0.1mm}$z_h^{rt}\hspace{-0.1em}\in\hspace{-0.1em} \mathcal{R}T^0_N(\mathcal{T}_h)$ \hspace{-0.1mm}of \hspace{-0.1mm}\eqref{discrete-ROF-dual},~\hspace{-0.1mm}which~\hspace{-0.1mm}satisfies $\vert \Pi_h z_h^{rt}\vert \leq 1$ a.e.\ in $\Omega$, a
    discrete strong duality relation applies, i.e.,
    \begin{align}
            I^{cr}_h(u_h^{cr})= D_h^{rt}(z_h^{rt})\,,\label{discrete_ROF_strong_duality_principle}
    \end{align}
    and the discrete convex optimality relations
    \begin{alignat}{2}
            \textup{div}\, z_h^{rt}&=\alpha\, (\Pi_h u_h^{cr}-g_h)&&\quad\textup{ in }\mathcal{L}^0(\mathcal{T}_h)\,,\label{discrete_ROF_optimality-relations.1}\\
            \Pi_hz_h^{rt}\cdot \nabla_{\!h} u_h^{cr}&=\vert \nabla_{\!h} u_h^{cr}\vert&&\quad\textup{ in }\mathcal{L}^0(\mathcal{T}_h)\,.\label{discrete_ROF_optimality-relations.2}
    \end{alignat}\newpage

%
    
    \subsection{The regularized, discretized Rudin--Osher--Fatemi model}

    \hspace{5mm}To approximate a discrete minimizer $u_h^{cr}\in \mathcal{S}^{1,\textit{cr}}(\mathcal{T}_h)$ of \eqref{discrete-ROF-primal}, it is common to approximate 
    the modulus function by strictly convex regularizations. In this connection, for every $\varepsilon\in (0,1)$, we define a special regularization $f_\varepsilon\colon \mathbb{R}\to \mathbb{R}_{\ge 0}$ of the modulus function, for every $t\in \mathbb{R}$, via
 \begin{align}\label{def:reg}
  f_\varepsilon(t)\coloneqq  (1-\varepsilon)\, \vert t\vert_\varepsilon\,,\qquad
  \vert t\vert_\varepsilon\coloneqq (t^2+\varepsilon^2)^{\frac{1}{2}}\,,
 \end{align} 
 where $\vert \cdot\vert_\varepsilon\colon \mathbb{R}\to \mathbb{R}_{\ge 0}$ is commonly referred to as the \textit{standard regularization}.\enlargethispage{7mm}
 
 Let us collect the most important properties of the regularization \eqref{def:reg}.

 \begin{lemma}\label{lem:properties}
  For every $\varepsilon\in (0,1)$, the following statements apply:
  \begin{itemize}[noitemsep,topsep=2pt,leftmargin=!,labelwidth=\widthof{(iii)}]
  \item[(i)] $f_\varepsilon\in C^1(\mathbb{R})$ with $f_\varepsilon'(0)=0$.
  \item[(ii)] For every $t\in \mathbb{R}$, it holds $-\varepsilon\, \vert t\vert-\varepsilon^2\leq f_\varepsilon(t)-\vert t\vert\leq \varepsilon\, (1-\vert t\vert)$.  
  \item[(iii)] For every $t\in \mathbb{R}$, it holds $\vert f_\varepsilon'(t)\vert \leq 1-\varepsilon$.
  \item[(iv)] For every $s\in \mathbb{R}$, it holds 
  \begin{align*}
        f_\varepsilon^*(s)\coloneqq\begin{cases}
           \smash{-\varepsilon \,((1-\varepsilon)^2-\vert s\vert^2)^{\frac{1}{2}}}&\quad\text{ if } \vert s\vert\leq  1-\varepsilon\\
          +\infty&\quad\text{ if } \vert s\vert> 1-\varepsilon
      \end{cases}\,.
  \end{align*}
  \end{itemize}
 \end{lemma}

 \begin{remark}
     The main reason to consider the regularization $f_\varepsilon\colon \mathbb{R}\to \mathbb{R}_{\ge 0}$ instead of the standard regularization $\vert \cdot\vert_\varepsilon\colon \mathbb{R}\to \mathbb{R}_{\ge 0}$ consists in the property (iii) in Lemma \ref{lem:properties}.~This~additional slope reduction enables us later to construct a sufficiently accurate, admissible approximation of the dual solution using an additional projection step, cf.\ Remark \ref{rem:proj} (below) and Section \ref{sec:experiments} (below).
 \end{remark}

 \begin{proof}
  \textit{ad (i).} The claimed regularity $f_\varepsilon\in C^1(\mathbb{R})$ is evident. Since for every $t\in \mathbb{R}$, it holds
  \begin{align}
      f_\varepsilon'(t)=(1-\varepsilon)\,\tfrac{t}{(t^2+\varepsilon^2)^{\frac{1}{2}}}\,,\label{lem:properties.1}
  \end{align}
  we have that $f_\varepsilon'(0)=0$.
  
  \textit{ad (ii).} For every $t\in \mathbb{R}$, due to $0\leq \vert t\vert_\varepsilon-\vert t\vert\leq \varepsilon$, we have that
  \begin{align*}
    -\varepsilon\, \vert t\vert-\varepsilon^2\leq  -\varepsilon\, \vert t\vert_\varepsilon \leq f_\varepsilon(t)-\vert t\vert=\varepsilon-\varepsilon\,\vert t\vert_\varepsilon\leq \varepsilon\, (1-\vert t\vert)\,.
  \end{align*}

  \textit{ad (iii).} Immediate consequence of the representation \eqref{lem:properties.1}.

  \textit{ad (iv).} Due to \cite[Proposition 13.20 (i)]{BC11}, for every $s\in \mathbb{R}$ and $\varepsilon\in (0,1)$, we have that
  \begin{align*}
      f_\varepsilon^*(s)=((1-\varepsilon)\,\vert \cdot\vert_\varepsilon)^*(s)=(1-\varepsilon)\,(\vert \cdot\vert_\varepsilon)^*\big(\tfrac{s}{1-\varepsilon}\big)\,.
  \end{align*}
  Since for every $s\in \mathbb{R}$ and $\varepsilon\in (0,1)$, it holds 
  \begin{align*}
      (\vert \cdot\vert_\varepsilon)^*\big(s)=
      \begin{cases}
           -\varepsilon \,(1-\vert s\vert^2)^{\frac{1}{2}}&\quad\text{ if } \vert s\vert\leq  1\\
          +\infty&\quad\text{ if } \vert s\vert> 1
      \end{cases}\,,
  \end{align*}
  we conclude that
  the claimed representation of the Fenchel conjugate applies.
  \end{proof}
  
    Given $g\in  L^2(\Omega)$, $\alpha> 0$, and an element-wise constant regularization parameter $\varepsilon_h\in  \mathcal{L}^0(\mathcal{T}_h)$ with $0<\varepsilon_h<1$ a.e.\ in $\Omega$,~for~${g_h\coloneqq \Pi_hg\in \mathcal{L}^0(\mathcal{T}_h)}$,~the regularized, discrete ROF model consists in the minimization of the functional  ${I^{cr}_{h,\varepsilon_h}\colon \hspace{-0.1em}\mathcal{S}^{1,\textit{cr}}(\mathcal{T}_h)\hspace{-0.1em}\to\hspace{-0.1em} \mathbb{R}}$, for every $v_h\in \mathcal{S}^{1,\textit{cr}}(\mathcal{T}_h)$ defined by 
	\begin{align}\label{discrete-ROF-primal_reg}
		I^{cr}_{h,\varepsilon_h}(v_h)\coloneqq  \|f_{\varepsilon_h}(\vert \nabla_{\!h}v_h\vert)\|_{L^1(\Omega)}+\tfrac{\alpha}{2}\|\Pi_hv_h-g_h\|^2_{L^2(\Omega)}\,.
	\end{align}
    Since the functional \eqref{discrete-ROF-primal_reg} is proper, strictly convex, weakly coercive, and lower semi-continuous,
        the direct method in the calculus of variations, cf.\ \cite{Dac08}, yields the existence of a unique minimizer $u_{h,\varepsilon_h}^{cr}\in \mathcal{S}^{1,\textit{cr}}(\mathcal{T}_h)$, called the \textit{regularized, discrete primal solution}. 
    Appealing to $(f_{\varepsilon_h}\circ \vert \cdot\vert)^*=f_{\varepsilon_h}^*\circ \vert \cdot\vert$, cf.\ \cite[Example 13.7]{BC11}, the corresponding (Fenchel) dual problem to the minimization of \eqref{discrete-ROF-primal} consists in the maximization of functional $D_{h,\varepsilon_h}^{rt}\colon\mathcal{R}T^0_N(\mathcal{T}_h)\to \mathbb{R}\cup\{-\infty\}$, for every $y_h\in \mathcal{R}T^0_N(\mathcal{T}_h)$ defined by 
	\begin{align}
		D_{h,\varepsilon_h}^{rt}(y_h)\coloneqq  -\int_{\Omega}{f_{\varepsilon_h}^*(\vert \Pi_hy_h\vert )\,\mathrm{d}x}-\tfrac{1}{2\alpha}\|\textup{div}\,y_h+\alpha\,g_h\|_{L^2(\Omega)}^2+\tfrac{\alpha}{2}\|g_h\|_{L^2(\Omega)}^2\,.\label{discrete-ROF-dual_reg}
	\end{align}
    
    The following proposition clarifies the well-posedness of the dual regularized, discretized ROF model, i.e., the existence of a maximizer of \eqref{discrete-ROF-dual_reg}. It also yields a discrete reconstruction formula for a maximizer of \eqref{discrete-ROF-dual_reg} from a minimizer of \eqref{discrete-ROF-primal_reg} and proves discrete~strong~duality.

     \begin{proposition}\label{prop:discrete_convex_duality}
 The following statements apply:
 \begin{itemize}[noitemsep,topsep=1pt,labelwidth=\widthof{(ii)},leftmargin=!]
 \item[(i)] A discrete weak duality relation applies, i.e.,
 \begin{align}
  \inf_{v_h\in \mathcal{S}^{1,\textit{cr}}_D(\mathcal{T}_h)}{I_{h,\varepsilon_h}^{\textit{cr}}(v_h)}\ge \sup_{y_h\in \mathcal{R}T^0_N(\mathcal{T}_h)}{D_{h,\varepsilon_h}^{\textit{rt}}(y_h)}\,.\label{discrete_weak_duality_reg}
 \end{align}
  \item[(ii)] The discrete flux $z_h^{\textit{rt}}\in \mathcal{L}^1(\mathcal{T}_h)$, defined via the generalized Marini formula
  \begin{align}
        z_{h,\varepsilon_h}^{rt}\coloneqq  \tfrac{f_{\varepsilon_h}'(\vert \nabla_{\!h} u_{h,\varepsilon_h}^{cr}\vert)}{\vert \nabla_{\!h} u_{h,\varepsilon_h}^{cr}\vert}\nabla_{\!h} u_{h,\varepsilon_h}^{cr}+\alpha\tfrac{\Pi_h u_{h,\varepsilon_h}^{cr}-g_h}{d}\big(\textup{id}_{\mathbb{R}^d}-\Pi_h\textup{id}_{\mathbb{R}^d}\big)\,,\label{def:reg_marini}
    \end{align}
 satisfies $z_{h,\varepsilon_h}^{\textit{rt}}\in \mathcal{R}T^0_N(\mathcal{T}_h)$ and the discrete convex optimality relations
 \begin{alignat}{2}
		\textup{div}\,z_{h,\varepsilon_h}^{rt}&=\alpha\,(\Pi_hu_{h,\varepsilon_h}^{cr}-g_h)&&\quad\text{ in }\mathcal{L}^0(\mathcal{T}_h)\,,\label{eq:pDirichletOptimalityCR1.1}\\
 \Pi_h z_{h,\varepsilon_h}^{rt}&=\tfrac{f_{\varepsilon_h}'(\vert \nabla_{\! h} u_{h,\varepsilon_h}^{\textit{cr}}\vert )}{\vert \nabla_{\! h} u_{h,\varepsilon_h}^{\textit{cr}}\vert }\nabla_{\! h} u_{h,\varepsilon_h}^{\textit{cr}}&&\quad\text{ in }\mathcal{L}^0(\mathcal{T}_h)^d\,.\label{eq:pDirichletOptimalityCR1.2}
	\end{alignat}
 \item[(iii)] The discrete flux $z_h^{\textit{rt}}\in \mathcal{R}T^0_N(\mathcal{T}_h)$ is a maximizer of \eqref{discrete-ROF-dual_reg} and discrete~strong~duality applies, i.e., 
\begin{align*}
        I^{cr}_{h,\varepsilon_h}(u_{h,\varepsilon_h}^{cr})=D_{h,\varepsilon_h}^{rt}(z_{h,\varepsilon_h}^{rt})\,.
    \end{align*}
 \end{itemize}
 \end{proposition}

  Note that, by the Fenchel--Young identity, cf.\ \cite[Proposition 5.1, p. 21]{ET99}, \eqref{eq:pDirichletOptimalityCR1.2}~is~equivalent~to
	\begin{align}\label{eq:pDirichletOptimalityCR2}
	  	\Pi_h z_{h,\varepsilon_h}^{rt}\cdot \nabla_{\!h} u_{h,\varepsilon_h}^{cr}&=f_{\varepsilon_h}^*(\vert \Pi_h z_{h,\varepsilon_h}^{rt}\vert )+f_\varepsilon (\vert \nabla_{\!h} u_{h,\varepsilon_h}^{cr}\vert)\quad\text{ in }\mathcal{L}^0(\mathcal{T}_h)\,.
	\end{align}

 \begin{remark}\label{rem:proj}
    Appealing to Lemma \ref{lem:properties} (iii), we have that $\vert \Pi_h z_{h,\varepsilon_h}^{rt}\vert\leq 1-\varepsilon_h$ a.e.\ in $\Omega$. Therefore,
     if $\|\Pi_hu_{h,\varepsilon_h}^{cr}-g_h\|_{L^\infty(\Omega)}\leq c_0$ for some $c_0>0$, which can be expected by discrete maximum principles, then, choosing  
     $\varepsilon_h\coloneqq \frac{\alpha c_0}{d}h$, yields that
     $\|z_{h,\varepsilon_h}^{rt}\|_{L^\infty(\Omega;\mathbb{R}^d)}\leq 1$.~However,~choices~like~${\varepsilon_h\sim h}$ let us expect convergence rates not better than $\mathcal{O}(h^{1/2})$, cf. Proposition~\ref{prop:reg_conv}~(i)~(below). In order to allow for the convergence rate $\mathcal{O}(h)$, one needs to choose $\varepsilon_h\sim h^2$. But,~in~this~case, we cannot guarantee that $\|z_{h,\varepsilon_h}^{rt}\|_{L^\infty(\Omega;\mathbb{R}^d)}\leq 1$, so that we instead consider the scaled vector field $\overline{z}_{h,\varepsilon_h}^{rt}\coloneqq z_{h,\varepsilon_h}^{rt}(\max\{1,\|z_{h,\varepsilon_h}^{rt}\|_{L^\infty(\Omega;\mathbb{R}^d)}\})^{-1}\in \mathcal{R}T^0_N(\mathcal{T}_h)$, which is still a sufficiently accurate approximation of the dual solution, as indicated by the numerical experiments, cf. Section \ref{sec:experiments}.
 \end{remark}
 
    \begin{proof}
 \textit{ad (i).} Using element-wise that $f_{\varepsilon_h}=f_{\varepsilon_h}^{**}$, the definition of the convex conjugate,~cf.~\eqref{def:fenchel}, and the discrete integration-by-parts formula \eqref{eq:pi0},  we find that
 \begin{align*}
  &\inf_{v_h\in \mathcal{S}^{1,\textit{cr}}_D(\mathcal{T}_h)}{I_{h,\varepsilon_h}^{\textit{cr}}(v_h)}=\inf_{v_h\in \mathcal{S}^{1,\textit{cr}}_D(\mathcal{T}_h)}{
  \|f_{\varepsilon_h}^{**}(\vert \nabla_{\! h} v_h\vert)\|_{L^1(\Omega)}+\tfrac{\alpha}{2}\|\Pi_h v_h-g_h\|_{L^2(\Omega)}^2}
 \\
 &\qquad=
  \inf_{v_h\in \mathcal{S}^{1,\textit{cr}}_D(\mathcal{T}_h)}{ \sup_{\overline{y}_h\in \mathcal{L}^0(\mathcal{T}_h)^d}{-\int_{\Omega}{f_{\varepsilon_h}^{*}(\vert \overline{y}_h \vert)\,\mathrm{d}x}+(\overline{y}_h,\nabla_{\! h} v_h)_\Omega+\tfrac{\alpha}{2}\|\Pi_h v_h-g_h\|_{L^2(\Omega)}^2}}
 \\&\qquad\ge 
  \inf_{v_h\in \mathcal{S}^{1,\textit{cr}}_D(\mathcal{T}_h)}{ \sup_{y_h\in \mathcal{R}T^0_N(\mathcal{T}_h)}{-\int_{\Omega}{f_{\varepsilon_h}^{*}(\vert \Pi_h y_h \vert)\,\mathrm{d}x}-(\textup{div}\,y_h,\Pi_h v_h)_\Omega+\tfrac{\alpha}{2}\|\Pi_h v_h-g_h\|_{L^2(\Omega)}^2}}
 \\&\qquad\ge 
  \sup_{y_h\in \mathcal{R}T^0_N(\mathcal{T}_h)}{-\int_{\Omega}{f_{\varepsilon_h}^{*}(\vert \Pi_h y_h \vert)\,\mathrm{d}x}-\sup_{\overline{v}_h\in \mathcal{L}^0(\mathcal{T}_h)}{(\textup{div}\,y_h,\overline{v}_h)_\Omega-\tfrac{\alpha}{2}\|\overline{v}_h-g_h\|_{L^2(\Omega)}^2}}
 \\&\qquad=
  \sup_{y_h\in \mathcal{R}T^0_N(\mathcal{T}_h)}{-\int_{\Omega}{f_{\varepsilon_h}^{*}(\vert \Pi_h y_h \vert)\,\mathrm{d}x}-\tfrac{1}{2\alpha}\|\textup{div}\,y_h+\alpha\, g_h\|_{L^2(\Omega)}^2+\tfrac{\alpha}{2}\|g_h\|_{L^2(\Omega)}^2}
 \\&\qquad=
  \sup_{y_h\in \mathcal{R}T^0_N(\mathcal{T}_h)}{D_{h,\varepsilon_h}^{\textit{rt}}(y_h)}\,,
 \end{align*}
which is the claimed discrete weak duality relation.

 \textit{ad (ii).} By Lemma \ref{lem:properties}, the minimality of $u_{h,\varepsilon_h}^{\textit{cr}}\in \mathcal{S}^{1,cr}(\mathcal{T}_h)$ for \eqref{discrete-ROF-primal_reg}, for every $v_h\in \mathcal{S}^{1,cr}(\mathcal{T}_h)$, yields that
 \begin{align}\label{prop:discrete_convex_duality.1}
     \Big(f_{\varepsilon_h}'(\vert \nabla_{\! h} u_{h,\varepsilon_h}^{\textit{cr}}\vert )\tfrac{\nabla_{\! h} u_{h,\varepsilon_h}^{\textit{cr}}}{\vert \nabla_{\! h} u_{h,\varepsilon_h}^{\textit{cr}}\vert },\nabla_{\! h} v_h\Big)_\Omega+\alpha\,(\Pi_hu_{h,\varepsilon_h}^{\textit{cr}}-g_h,\Pi_h v_h)_\Omega=0\,.
 \end{align}
 By definition, the discrete flux $z_{h,\varepsilon_h}^{\textit{rt}}\in \mathcal{L}^1(\mathcal{T}_h)^d$, defined by \eqref{def:reg_marini}, satisfies the discrete convex optimality condition \eqref{eq:pDirichletOptimalityCR1.2} and $\textup{div}\,(z_{h,\varepsilon_h}^{\textit{rt}}|_T)=\alpha\,(\Pi_hu_{h,\varepsilon_h}^{\textit{cr}}-g_h)|_T$ in $T$ for all $T\in \mathcal{T}_h$. 
    Choosing $v_h=1\in \mathcal{S}^{1,cr}(\mathcal{T}_h)$ in \eqref{prop:discrete_convex_duality.1}, we find that $\int_{\Omega}{\alpha\,(\Pi_hu_{h,\varepsilon_h}^{\textit{cr}}-g_h)\,\mathrm{d}x}=0$.
 Hence, since for $\Gamma_D=\emptyset$ the divergence operator $\textup{div}\colon 
 \mathcal{R}T^0_N(\mathcal{T}_h)\to \mathcal{L}^0(\mathcal{T}_h)/\mathbb{R}$ is surjective, there exists
  $y_h\in \mathcal{R}T^0_N(\mathcal{T}_h)$ such that $\textup{div}\, y_h=\alpha\,(\Pi_hu_{h,\varepsilon_h}^{\textit{cr}}-g_h) $ in $\mathcal{L}^0(\mathcal{T}_h)$. Then, we have that $\textup{div}\,((z_{h,\varepsilon_h}^{\textit{rt}}-y_h)|_T)=0$ in $T$ for all $T\in \mathcal{T}_h$, i.e., $z_{h,\varepsilon_h}^{\textit{rt}}-y_h\in \mathcal{L}^0(\mathcal{T}_h)^d$. In addition, for every $v_h\in \mathcal{S}^{1,\textit{cr}}(\mathcal{T}_h)$, it holds
 \begin{align*}
  \begin{aligned}
  (\Pi_h y_h,\nabla_{\! h} v_h)_\Omega&=-(\textup{div}\, y_h,\Pi_h v_h)_\Omega\\&=-\alpha\,(\Pi_hu_{h,\varepsilon_h}^{\textit{cr}}-g_h,\Pi_h v_h)_\Omega\\&=\Big(f_{\varepsilon_h}'(\vert \nabla_{\! h} u_{h,\varepsilon_h}^{\textit{cr}}\vert )\tfrac{\nabla_{\! h} u_{h,\varepsilon_h}^{\textit{cr}}}{\vert \nabla_{\! h} u_{h,\varepsilon_h}^{\textit{cr}}\vert },\nabla_{\! h} v_h\Big)_\Omega
 \\& =(\Pi_h z_{h,\varepsilon_h}^{\textit{rt}},\nabla_{\! h} v_h)_\Omega\,.
 \end{aligned}
 \end{align*}
 In other words, for every $v_h\in \mathcal{S}^{1,\textit{cr}}(\mathcal{T}_h)$, it holds
 \begin{align*}
  (y_h-z_{h,\varepsilon_h}^{\textit{rt}},\nabla_{\! h} v_h)_\Omega=(\Pi_h y_h-\Pi_h z_{h,\varepsilon_h}^{\textit{rt}},\nabla_{\! h} v_h)_\Omega=0\,,
 \end{align*}
 i.e., $y_h-z_{h,\varepsilon_h}^{\textit{rt}}\in \nabla_{\! h}(\mathcal{S}^{1,\textit{cr}}_D(\mathcal{T}_h))^{\perp}$. By the decomposition \eqref{eq:decomposition.2}, we have that $\nabla_{\! h}(\mathcal{S}^{1,\textit{cr}}_D(\mathcal{T}_h))^{\perp}=\textup{ker}(\textup{div}|_{\mathcal{R}T^0_N(\mathcal{T}_h)})\subseteq \mathcal{R}T^0_N(\mathcal{T}_h)$. 
 As a result, it holds $y_h-z_{h,\varepsilon_h}^{\textit{rt}}\in \mathcal{R}T^0_N(\mathcal{T}_h)$.~Due~to~${y_h\in \mathcal{R}T^0_N(\mathcal{T}_h)}$, we conclude that $z_{h,\varepsilon_h}^{\textit{rt}}\in \mathcal{R}T^0_N(\mathcal{T}_h)$. In particular, now from
 $\textup{div}\,(z_{h,\varepsilon_h}^{\textit{rt}}|_T)=\alpha\,(\Pi_hu_{h,\varepsilon_h}^{\textit{cr}}-g_h)|_T$ in $T$ for all $T\in \mathcal{T}_h$, it follows the discrete optimality condition
  \eqref{eq:pDirichletOptimalityCR1.1}.

 \textit{ad (iii).} Using \eqref{eq:pDirichletOptimalityCR2}, \eqref{eq:pDirichletOptimalityCR1.1}, and the discrete integration-by-parts formula \eqref{eq:pi0},~we~find~that
 \begin{align*}
  I_{h,\varepsilon_h}^{\textit{cr}}(u_{h,\varepsilon_h}^{\textit{cr}})&=
  \|f_{\varepsilon_h}(\vert \nabla_{\! h} u_{h,\varepsilon_h}^{\textit{cr}}\vert)\|_{L^1(\Omega)}+\tfrac{\alpha}{2}\|\Pi_h u_{h,\varepsilon_h}^{\textit{cr}}-g_h\|_{L^2(\Omega)}^2
 \\& =-\int_{\Omega}{f_{\varepsilon_h}^{*}(\vert \Pi_h z_{h,\varepsilon_h}^{\textit{rt}} \vert)\,\mathrm{d}x}+(\Pi_h z_{h,\varepsilon_h}^{\textit{rt}},\nabla_{\! h} u_{h,\varepsilon_h}^{\textit{cr}})_\Omega+\tfrac{1}{2\alpha }\|\textup{div}\,z_{h,\varepsilon_h}^{\textit{rt}}\|_{L^2(\Omega)}^2
 \\& =-\int_{\Omega}{f_{\varepsilon_h}^{*}(\vert \Pi_h z_{h,\varepsilon_h}^{\textit{rt}} \vert)\,\mathrm{d}x}-( \textup{div}\,z_{h,\varepsilon_h}^{\textit{rt}},\Pi_hu_{h,\varepsilon_h}^{\textit{cr}})_\Omega+\tfrac{1}{2\alpha }\|\textup{div}\,z_{h,\varepsilon_h}^{\textit{rt}}\|_{L^2(\Omega)}^2
 \\& =-\int_{\Omega}{f_{\varepsilon_h}^{*}(\vert \Pi_h z_{h,\varepsilon_h}^{\textit{rt}} \vert)\,\mathrm{d}x}-\tfrac{1}{\alpha }( \textup{div}\,z_{h,\varepsilon_h}^{\textit{rt}},\textup{div}\,z_{h,\varepsilon_h}^{\textit{rt}}+\alpha\, g_h)_\Omega+\tfrac{1}{2\alpha }\|\textup{div}\,z_{h,\varepsilon_h}^{\textit{rt}}\|_{L^2(\Omega)}^2
  \\&=-\int_{\Omega}{f_{\varepsilon_h}^{*}(\vert \Pi_h z_{h,\varepsilon_h}^{\textit{rt}} \vert)\,\mathrm{d}x}-\tfrac{1}{2\alpha }\|\textup{div}\,z_{h,\varepsilon_h}^{\textit{rt}}+\alpha\, g_h\|_{L^2(\Omega)}^2
  \\&=D_{h,\varepsilon_h}^{\textit{rt}}(z_{h,\varepsilon_h}^{\textit{rt}})\,,
 \end{align*}
    which is the claimed discrete strong duality relation and, thus, appealing to the discrete weak duality relation \eqref{discrete_weak_duality_reg}, proves the maximality of $z_{h,\varepsilon_h}^{\textit{rt}}\in \mathcal{R}T^0_N(\mathcal{T}_h)$ for \eqref{discrete-ROF-dual_reg}.
 \end{proof}

The \hspace{-0.1mm}following \hspace{-0.1mm}proposition \hspace{-0.1mm}describes \hspace{-0.1mm}the \hspace{-0.1mm}approximative \hspace{-0.1mm}behavior \hspace{-0.1mm}the \hspace{-0.1mm}regularized, \hspace{-0.1mm}discretized~\hspace{-0.1mm}ROF problem towards the (unregularized) discretized ROF problem, given uniform convergence~(to~zero) of the element-wise constant regularization parameter $\varepsilon_h\in \mathcal{L}^0(\mathcal{T}_h)$. In what follows, in the convergence $\|\varepsilon_h\|_{L^\infty(\Omega)}\to 0$, 
the average mesh-size $h>0$ is always fixed.\enlargethispage{2mm}

\begin{proposition}\label{prop:reg_conv}
    If $\|\varepsilon_h\|_{L^\infty(\Omega)}<1$, then the following statements apply:
    \begin{itemize}[noitemsep,topsep=2pt,leftmargin=!,labelwidth=\widthof{(iii)}]
        \item[(i)] It holds $\tfrac{\alpha}{2}\|\Pi_h u_{h,\varepsilon_h}^{cr}-\Pi_hu_h^{cr}\|_{L^2(\Omega)}^2
  \leq \tfrac{\|\varepsilon_h\|_{L^\infty(\Omega)}}{1-\|\varepsilon_h\|_{L^\infty(\Omega)}}\,(\tfrac{\alpha}{2}\,\|g\|_{L^2(\Omega)}^2+2\,\vert \Omega\vert)$.
        \item[(ii)] $\textup{div}\, z_{h,\varepsilon_h}^{rt}\to \alpha\,(\Pi_hu_h^{cr}-g_h)$ in $\mathcal{L}^0(\mathcal{T}_h)$ $(\|\varepsilon_h\|_{L^\infty(\Omega)}\to 0)$.

        \item[(iii)] $f_{\varepsilon_h}^*(\vert \Pi_h z_{h,\varepsilon_h}^{rt}\vert )\to 0$ in $\mathcal{L}^0(\mathcal{T}_h)$ $(\|\varepsilon_h\|_{L^\infty(\Omega)}\to 0)$.

        \item[(iv)] $f_{\varepsilon_h} (\vert \nabla_{\!h} u_{h,\varepsilon_h}^{cr}\vert)\to \nabla_{\!h} u_h^{cr}$ in $\mathcal{L}^0(\mathcal{T}_h)$ $(\|\varepsilon_h\|_{L^\infty(\Omega)}\to 0)$.
    \end{itemize}
 \end{proposition}

 \begin{proof}

 \textit{ad (i).} Using both the strong convexity of $I_h^{cr}\colon \mathcal{S}^{1,cr}(\mathcal{T}_h)\to \mathbb{R}\cup\{+\infty\}$ and Lemma~\ref{lem:properties}~(ii), 
 we obtain
 \begin{align}\label{prop:reg_conv.1}
    \begin{aligned}
   \tfrac{\alpha}{2}\|\Pi_h u_{h,\varepsilon_h}^{cr}-\Pi_hu_h^{cr}\|_{L^2(\Omega)}^2&\leq I_h^{cr}(u_{h,\varepsilon_h}^{cr})-I_h^{cr}(u_h^{cr})\\&\leq 
   \tfrac{1}{1-\|\varepsilon_h\|_{L^\infty(\Omega)}} I_{h,\varepsilon_h}^{cr}(u_{h,\varepsilon_h}^{cr})+\tfrac{\|\varepsilon_h\|_{L^\infty(\Omega)}^2}{1-\|\varepsilon_h\|_{L^\infty(\Omega)}}\vert \Omega\vert -I_h^{cr}(u_h^{cr})
   \\&\leq 
   \tfrac{1}{1-\|\varepsilon_h\|_{L^\infty(\Omega)}} I_{h,\varepsilon_h}^{cr}(u_h^{cr})+\tfrac{\|\varepsilon_h\|_{L^\infty(\Omega)}^2}{1-\|\varepsilon_h\|_{L^\infty(\Omega)}}\vert \Omega\vert-I_h^{cr}(u_h^{cr})
   \\&\leq 
   \tfrac{1}{1-\|\varepsilon_h\|_{L^\infty(\Omega)}} ( I_h^{cr}(u_h^{cr})
+2\,\|\varepsilon_h\|_{L^\infty(\Omega)}\,\vert\Omega\vert)-I_h^{cr}(u_h^{cr})
   \\&=
   \tfrac{\|\varepsilon_h\|_{L^\infty(\Omega)}}{1-\|\varepsilon_h\|_{L^\infty(\Omega)}} \,(I_h^{cr}(u_h^{cr})+2\,\vert\Omega\vert)\,.
    \end{aligned}
  \end{align}
  Since, by the minimality of $u_h^{cr}\in \mathcal{S}^{1,cr}(\mathcal{T}_h)$ for \eqref{discrete-ROF-primal} and the  $L^2$-stability of $\Pi_h\colon L^2(\Omega)\to \mathcal{L}^0(\mathcal{T}_h)$,  it holds
  \begin{align}\label{prop:reg_conv.2}
  I_h^{cr}(u_h^{cr})\leq I_h^{cr}(0)=\tfrac{\alpha}{2}\|g_h\|_{L^2(\Omega)}^2\leq \tfrac{\alpha}{2}\|g\|_{L^2(\Omega)}^2\,,
  \end{align}
  from \eqref{prop:reg_conv.3} we conclude the claimed error estimate.

  \textit{ad (ii).} From claim (i), it follows that
  \begin{align}
      \Pi_h u_{h,\varepsilon_h}^{cr}\to \Pi_hu_h^{cr}\quad \text{ in } \mathcal{L}^0(\mathcal{T}_h)\quad (\|\varepsilon_h\|_{L^\infty(\Omega)}\to 0)\,.\label{prop:reg_conv.3}
  \end{align}
  Thus, using \eqref{prop:reg_conv.3}, from $\textup{div}\, z_{h,\varepsilon_h}^{rt}=\alpha\, ( \Pi_h u_{h,\varepsilon_h}^{cr}-g_h)$ in $\mathcal{L}^0(\mathcal{T}_h)$, cf.\ \eqref{eq:pDirichletOptimalityCR1.1},  we conclude that
  \begin{align*}
    \textup{div}\, z_{h,\varepsilon_h}^{rt}\to \alpha\,(\Pi_hu_h^{cr}-g_h)\quad\textup{ in }\mathcal{L}^0(\mathcal{T}_h)\quad(\|\varepsilon_h\|_{L^\infty(\Omega)}\to 0)\,.
    \end{align*}
    
    \textit{ad (iii).} Due to $\Pi_h z_{h,\varepsilon_h}^{rt}=\frac{f_{\varepsilon_h}'(\vert \nabla_{\!h} u_{h,\varepsilon_h}^{cr}\vert)}{\vert \nabla_{\!h} u_{h,\varepsilon_h}^{cr}\vert} \nabla_{\!h} u_{h,\varepsilon_h}^{cr} $ and Lemma \ref{lem:properties} (iii),~we~have~that
    \begin{align}
        \vert \Pi_h z_{h,\varepsilon_h}^{rt}\vert =\vert f_{\varepsilon_h}'(\vert \nabla_{\!h} u_{h,\varepsilon_h}^{cr}\vert)\vert \leq 1-\varepsilon_h\quad\text{ a.e.\ in }\Omega\,.\label{prop:reg_conv.5}
    \end{align}
    Therefore, using  Lemma \ref{lem:properties} (iv)  together with \eqref{prop:reg_conv.5}, we conclude that
    \begin{align*}
       \left. \begin{aligned}
        \vert f_{\varepsilon_h}^*(\vert \Pi_h z_{h,\varepsilon_h}^{rt}\vert )\vert& =
        \varepsilon_h\,((1-\varepsilon_h)^2-\vert \Pi_h z_{h,\varepsilon_h}^{rt}\vert ^2)^{\frac{1}{2}}
        \\&\leq \varepsilon_h\,(1-\varepsilon_h)\leq \varepsilon_h
        \end{aligned}\quad\right\}\quad\text{ a.e.\ in }\Omega\,,
    \end{align*}
    which implies that $f_{\varepsilon_h}^*(\vert \Pi_h z_{h,\varepsilon_h}^{rt}\vert )\to 0$ in $\mathcal{L}^0(\mathcal{T}_h)$ $(\|\varepsilon_h\|_{L^\infty(\Omega)}\to 0)$.

    \textit{ad (iv).} Due to \eqref{prop:reg_conv.2}, $(u_{h,\varepsilon_h}^{cr})_{\|\varepsilon_h\|_{L^\infty(\Omega)}\to 0}\subseteq \mathcal{S}^{1,cr}(\mathcal{T}_h)$ is bounded. The finite-dimensionality~of $\mathcal{S}^{1,cr}(\mathcal{T}_h)$ and the Bolzano--Weierstraß theorem yield a subsequence   ${(u_{h,\varepsilon_h'}^{cr})_{\|\varepsilon_h'\|_{L^\infty(\Omega)}\to 0}\subseteq \mathcal{S}^{1,cr}(\mathcal{T}_h)}$ and a function $\tilde{u}_h^{cr}\in \mathcal{S}^{1,cr}(\mathcal{T}_h)$ such that
    \begin{align}\label{prop:reg_conv.6}
        u_{h,\varepsilon_h'}^{cr}\to \tilde{u}_h^{cr}\quad\text{ in } \mathcal{S}^{1,cr}(\mathcal{T}_h)\quad(\|\varepsilon_h'\|_{L^\infty(\Omega)}\to 0)\,.
    \end{align}
    From \eqref{prop:reg_conv.6} it is readily derived that
    \begin{align*}
         f_{\varepsilon_h'} (\vert \nabla_{\!h} u_{h,\varepsilon_h'}^{cr}\vert)\to \nabla_{\!h} \tilde{u}_h^{cr}\quad\text{ in }\mathcal{L}^0(\mathcal{T}_h)\quad(\|\varepsilon_h'\|_{L^\infty(\Omega)}\to 0)\,.
    \end{align*}
    Consequently, for every $v_h\in \mathcal{S}^{1,cr}(\mathcal{T}_h)$, we find that
    \begin{align*}
        I_h^{cr}(\tilde{u}_h^{cr})&=\lim_{\|\varepsilon_h'\|_{L^\infty(\Omega)}\to 0}{I_{h,\varepsilon_h'}^{cr}(u_{h,\varepsilon_h'}^{cr})}\\&\leq \lim_{\|\varepsilon_h'\|_{L^\infty(\Omega)}\to 0}{I_{h,\varepsilon_h'}^{cr}(v_h)}\\&=I_h^{cr}(v_h)\,.
    \end{align*}
    Thus, due to the uniqueness of $u_h^{cr}\in \mathcal{S}^{1,cr}(\mathcal{T}_h)$ as a minimizer of \eqref{discrete-ROF-primal}, we get $\tilde{u}_h^{cr}=u_h^{cr}$~in~$\mathcal{S}^{1,cr}(\mathcal{T}_h)$. Since this argumentation remains valid for each subsequence of $(u_{h,\varepsilon_h}^{cr})_{\|\varepsilon_h\|_{L^\infty(\Omega)}\to 0}\subseteq \mathcal{S}^{1,cr}(\mathcal{T}_h)$, the standard subsequence principle implies that $f_{\varepsilon_h} (\vert \nabla_{\!h} u_{h,\varepsilon_h}^{cr}\vert)\hspace{-0.1em}\to\hspace{-0.1em} \nabla_{\!h} u_h^{cr}$ in $\mathcal{L}^0(\mathcal{T}_h)$~${(\|\varepsilon_h\|_{L^\infty(\Omega)}\hspace{-0.1em}\to\hspace{-0.1em} 0)}$.
 \end{proof}

    The approximation properties of the regularized, discrete ROF model \eqref{discrete-ROF-primal_reg} (and  \eqref{discrete-ROF-dual_reg}) towards the (unregularized) discrete ROF model \eqref{discrete-ROF-primal} (and  \eqref{discrete-ROF-dual_reg}) enable us to transfer the discrete convex duality relations established in Proposition \ref{prop:discrete_convex_duality}, which apply mainly due to the differentiability of the regularized, discrete ROF model, to the non-differentiable discrete~ROF~model. To the best of the authors' knowledge, the following discrete convex duality relations for the (unregularized) discrete ROF model \eqref{discrete-ROF-primal} 
    seem to be new.\enlargethispage{7mm}

    \begin{theorem}\label{thm:discrete_ROF_strong_duality}
        There exists a vector field $z_h^{rt}\in \mathcal{R}T^0_N(\mathcal{T}_h)$ with $\vert \Pi_h z_h^{rt}\vert \leq 1$ a.e.\  in $\Omega$ and the following properties:
        \begin{itemize}[noitemsep,topsep=2pt,leftmargin=!,labelwidth=\widthof{(iii)}]
            \item[(i)] For a not relabeled subsequence, it holds 
            \begin{align*}
                z_{h,\varepsilon_h}^{rt}\to z_h^{rt}\quad \text{ in } \mathcal{R}T^0_N(\mathcal{T}_h)\quad (\|\varepsilon_h\|_{L^\infty(\Omega)}\to 0)\,.
            \end{align*}
            \item[(ii)] There hold the following discrete convex optimality relations:
            \begin{align*}
            \begin{aligned}
                \textup{div}\, z_h^{rt}&=\alpha\, (\Pi_h u_h^{cr}-g_h)&&\quad\text{ in }\mathcal{L}^0(\mathcal{T}_h)\,,\\
                \Pi_hz_h^{rt}\cdot \nabla_{\!h} u_h^{cr}&=\vert \nabla_{\!h} u_h^{cr}\vert&&\quad\text{ in }\mathcal{L}^0(\mathcal{T}_h)\,.
                \end{aligned}
            \end{align*}
            
            \item[(iii)] The discrete flux $z_h^{rt}\in \mathcal{R}T^0_N(\mathcal{T}_h)$ is maximal for $D_h^{rt}\colon \mathcal{R}T^0_N(\mathcal{T}_h)\to \mathbb{R}$ and discrete strong duality applies, i.e., 
            \begin{align*}
                I_h^{cr}(u_h^{cr})=D_h^{rt}(z_h^{rt})\,.
            \end{align*}
        \end{itemize}
    \end{theorem}

    \begin{proof}
        \textit{ad (i).} Due to Proposition \ref{prop:reg_conv}  (ii) and \eqref{prop:reg_conv.5}, the sequence ${(z_{h,\varepsilon_h}^{rt})_{\|\varepsilon_h\|_{L^\infty(\Omega)}\to 0}\subseteq \mathcal{R}T^0_N(\mathcal{T}_h)}$ is bounded. Thus, by the finite-dimensionality of $\mathcal{R}T^0_N(\mathcal{T}_h)$, the Bolzano--Weierstraß theorem yields a not relabeled subsequence and a vector field $z_h^{rt}\in \mathcal{R}T^0_N(\mathcal{T}_h)$ such that
        \begin{align}\label{thm:discrete_ROF_strong_duality.1}
            z_{h,\varepsilon_h}^{rt}\to z_h^{rt}\quad \text{ in }\mathcal{R}T^0_N(\mathcal{T}_h)\quad (\|\varepsilon_h\|_{L^\infty(\Omega)}\to 0)\,.
        \end{align}
        Due to the continuity of $\Pi_h\colon L^1(\Omega)\to \mathcal{L}^0(\mathcal{T}_h)$ and $\mathcal{R}T^0_N(\mathcal{T}_h)\hookrightarrow L^1(\Omega)$, from \eqref{thm:discrete_ROF_strong_duality.1}, we obtain
        \begin{align}\label{thm:discrete_ROF_strong_duality.2}
            \Pi_h z_{h,\varepsilon_h}^{rt}\to \Pi_h z_h^{rt}\quad \text{ in }\mathcal{L}^0(\mathcal{T}_h)\quad (\|\varepsilon_h\|_{L^\infty(\Omega)}\to 0)\,.
        \end{align}
        From $\vert \Pi_h z_{h,\varepsilon_h}^{rt}\vert \leq 1-\varepsilon_h$ a.e.\ in $\Omega$, cf.\ \eqref{prop:reg_conv.5}, and \eqref{thm:discrete_ROF_strong_duality.2}, we obtain $\vert \Pi_h z_h^{rt}\vert \leq 1$ a.e.\  in $\Omega$, i.e.,
        \begin{align}\label{thm:discrete_ROF_strong_duality.3}
            I_{K_1(0)}(\Pi_h z_h^{rt})=0\,.
        \end{align}
        
        \textit{ad (ii).} Using  Proposition \ref{prop:reg_conv}, \eqref{eq:pDirichletOptimalityCR1.1}, and \eqref{eq:pDirichletOptimalityCR2}, we find that
        \begin{align*}
            \left.\begin{aligned}
            \textup{div}\,z_h^{rt}&=\lim_{\|\varepsilon_h\|_{L^\infty(\Omega)}\to 0}{\textup{div}\,z_{h,\varepsilon_h}^{rt}}\\&=\lim_{\|\varepsilon_h\|_{L^\infty(\Omega)}\to 0}{\alpha\,(\Pi_hu_{h,\varepsilon_h}^{cr}-g_h)}\\&=\alpha\,(\Pi_h u_h^{cr}-g_h)\end{aligned}\quad\right\}\quad\text{ a.e. in }\Omega\,,\\
            \intertext{as well as}
            \left.\begin{aligned}
            \Pi_h z_h^{rt}\cdot \nabla_{\!h} u_h^{cr}&=\lim_{\|\varepsilon_h\|_{L^\infty(\Omega)}\to 0}{\Pi_h z_{h,\varepsilon_h}^{rt}\cdot \nabla_{\!h} u_{h,\varepsilon_h}^{cr}}
            \\&=\lim_{\|\varepsilon_h\|_{L^\infty(\Omega)}\to 0}{f_{\varepsilon_h}^*(\vert  \Pi_h z_{h,\varepsilon_h}^{rt}\vert )+f_{\varepsilon_h}(\vert \nabla_{\!h} u_{h,\varepsilon_h}^{cr}\vert)}
            \\&=\vert\nabla_{\!h} u_h^{cr}\vert
            \end{aligned}\quad\right\}\quad\text{ a.e. in }\Omega \,,
        \end{align*}
        i.e., the claimed discrete convex optimality conditions.

        \textit{ad (iii).} 
        Using Proposition \ref{prop:reg_conv} and \eqref{thm:discrete_ROF_strong_duality.3}, we find that
        \begin{align*}
            I_h^{cr}(u_h^{cr})&=\lim_{\|\varepsilon_h\|_{L^\infty(\Omega)}\to 0}{I_{h,\varepsilon_h}^{cr}(u_{h,\varepsilon_h}^{cr})}\\&=\lim_{\|\varepsilon_h\|_{L^\infty(\Omega)}\to 0}{D_{h,\varepsilon_h}^{rt}(z_{h,\varepsilon_h}^{rt})}
            \\&=D_h^{rt}(z_h^{rt})\,,
        \end{align*}
        i.e., the claimed discrete strong duality relation.
    \end{proof}\newpage

 	\section{Numerical experiments}\label{sec:experiments}\enlargethispage{5mm}
	
	\hspace{5mm}In this section, we review the theoretical findings of Section \ref{sec:ROF} via numerical experiments. To compare \hspace{-0.1mm}approximations \hspace{-0.1mm}to \hspace{-0.1mm}an \hspace{-0.1mm}exact \hspace{-0.1mm}solution, \hspace{-0.1mm}we \hspace{-0.1mm}impose \hspace{-0.1mm}Dirichlet \hspace{-0.1mm}boundary \hspace{-0.1mm}conditions~\hspace{-0.1mm}on~\hspace{-0.1mm}${\Gamma_D\!=\!\partial\Omega}$, though an existence theory is difficult to establish, in general. However, the concepts derived in Section \ref{sec:ROF} carry over verbatimly with $\Gamma_N=\emptyset$ provided that the existence~of~a~minimizer~is~given. All experiments were conducted deploying the finite element software package \textsf{FEniCS}  (version 2019.1.0), cf.\  \cite{LW10}. All graphics were generated using the \textsf{Matplotlib} library (version 3.5.1),~cf.~\cite{Hun07}, and the \textsf{Vedo} library (version 2023.4.4), cf.\ \cite{vedo}.
	
	\subsection{Implementation details regarding the optimization procedure}
 
	\hspace{5mm}All computations are based on the regularized, discrete ROF problem \eqref{discrete-ROF-primal_reg}.~This~is~motivated by the fact that appealing to Proposition \ref{prop:reg_conv} (i), in order to bound~the~error~${\|u-\Pi_h u_h^{cr}\|_{L^2(\Omega)}}$, it suffices to determine the error $\|u-\Pi_h u_{h,\varepsilon_h}^{cr}\|_{L^2(\Omega)}$. The iterative~minimization~of \eqref{discrete-ROF-primal_reg} is realized using a semi-implicit discretized $L^2$-gradient flow from \cite{BDN18} (see also \cite[Section~5]{Bar21}) modified with a residual stopping criterion guaranteeing the necessary accuracy~in~the~optimization~procedure.

	\begin{algorithm}[Semi-implicit discretized $L^2$-gradient flow]\label{algorithm}
		Let $g_h, \varepsilon_h\in \mathcal{L}^0(\mathcal{T}_h)$ be such that $\varepsilon_h>0$ a.e.\ in $\Omega$ and $\|\varepsilon_h\|_{L^\infty(\Omega)}<1$, and choose $\tau,\varepsilon_{stop}^h> 0$. Moreover, let
		$u^0_h\in \mathcal{S}^{1,\textit{cr}}_D(\mathcal{T}_h)$. Then, for every $k\in \mathbb{N}$:
		\begin{description}[noitemsep,topsep=2pt,labelwidth=\widthof{\textit{(ii)}},leftmargin=!,font=\normalfont\itshape]
			\item[(i)] Compute the iterate $\smash{u_h^k\in \mathcal{S}^{1,\textit{cr}}_D(\mathcal{T}_h)}$ such that for every $\smash{v_h\in \mathcal{S}^{1,\textit{cr}}_D(\mathcal{T}_h)}$,~it~holds
			\begin{align}
  (d_{\tau} u_h^k,v_h)_{\Omega}+\Big(\tfrac{f_{\varepsilon_h}'(\vert \nabla_{\!h}u_h^{k-1}\vert )}{\vert \nabla_{\!h}u_h^{k-1}\vert}\nabla_{\!h}u_h^k ,\nabla_{\!h}v_h \Big)_{\Omega}+\alpha\, (\Pi_hu_h^k-g_h,\Pi_hv_h)_{\Omega}=0\,,\label{alg:step_1}
			\end{align}
			where $d_{\tau} u_h^k\coloneqq \frac{1}{\tau}(u_h^k-u_h^{k-1})$ denotes the backward difference quotient.
		\item[(ii)] Compute the residual $\smash{r_h^k\in \mathcal{S}^{1,\textit{cr}}_D(\mathcal{T}_h)}$ such that for every $\smash{v_h\in\mathcal{S}^{1,\textit{cr}}_D(\mathcal{T}_h)}$, it holds
		\begin{align}
			(r_h^k,v_h)_{\Omega}=\Big(\tfrac{f_{\varepsilon_h}'(\vert \nabla_{\!h}u_h^k\vert )}{\vert \nabla_{\!h}u_h^k\vert} \nabla_{\!h}u_h^k,\nabla_{\!h}v_h \Big)_{\Omega}+\alpha\, (\Pi_hu_h^k-g_h,\Pi_hv_h)_{\Omega}\,.\label{alg:step_2}
		\end{align}
		Stop if $\|r_h^k\|_{L^2(\Omega)}\leq \varepsilon_{stop}^h$; otherwise, increase $k\!\to\! k+1$~and~continue~with~step~\textit{(i)}.
		\end{description}
	\end{algorithm}

 Appealing to \cite[Remark 5.5]{Bar21}, the iterates $u_h^k\in  \mathcal{S}^{1,\textit{cr}}_D(\mathcal{T}_h)$, $k\in \mathbb{N}$, the residuals $r_h^k\in \mathcal{S}^{1,\textit{cr}}_D(\mathcal{T}_h)$, $k\in \mathbb{N}$, generated by Algorithm~\ref{algorithm}, and the minimizer $u_{h,\varepsilon_h}^{cr}\in \mathcal{S}^{1,\textit{cr}}_D(\mathcal{T}_h)$ of \eqref{discrete-ROF-primal_reg} satisfy
	\begin{align}
		\|u_{h,\varepsilon_h}^{cr}-u_h^k\|_{L^2(\Omega)}\leq 2\,\|r_h^k\|_{L^2(\Omega)}\,.\label{residual}
	\end{align}
	In consequence, if we choose as a stopping criterion that $\|r_h^{k^*}\|_{L^2(\Omega)}\hspace{-0.1em}\leq\hspace{-0.1em}\varepsilon_{\textit{stop}}^h\hspace{-0.1em}\coloneqq \hspace{-0.1em}c_{\textit{stop}}\,h$ for $k^*\in \mathbb{N}$, where $c_{\textit{stop}}\hspace{-0.1em}>\hspace{-0.1em}0$ does not depend on $h\hspace{-0.1em}>\hspace{-0.1em}0$, then, owing to Proposition \ref{prop:reg_conv} (i) and \eqref{residual},~we~have~that
    \begin{align*}
        \|\Pi_h(u_h^{cr}-u_h^{k^*})\|_{L^2(\Omega)}^2\leq \tfrac{\|\varepsilon_h\|_{L^\infty(\Omega)}}{1-\|\varepsilon_h\|_{L^\infty(\Omega)}}\,(2\,\|g\|_{L^2(\Omega)}^2+\tfrac{8}{\alpha}\,\vert \Omega\vert)+8\,c_{\textit{stop}}^2\,h^2\,.
    \end{align*}
    If $\|\varepsilon_h\|_{L^\infty(\Omega)}\leq c_{\textit{reg}}\, h^2$, where $c_{\textit{reg}}\in (0,1)$, then, we arrive at $\|\Pi_h(u_h^{cr}-u_h^{k^*})\|_{L^2(\Omega)}=\mathcal{O}(h)$. 
    Thus, to bound the error $\|u-\Pi_hu_h^{cr}\|_{L^2(\Omega)} $ experimentally, it is sufficient to compute $\smash{\|u-\Pi_hu_h^{k^*}\|_{L^2(\Omega)}}$.

    The \hspace{-0.1mm}following \hspace{-0.1mm}proposition \hspace{-0.1mm}proves \hspace{-0.1mm}the \hspace{-0.1mm}well-posedness, \hspace{-0.1mm}stability, \hspace{-0.1mm}and \hspace{-0.1mm}convergence \hspace{-0.1mm}of \hspace{-0.1mm}Algorithm~\hspace{-0.1mm}\ref{algorithm}.\hspace{-0.5mm}

 \begin{proposition}\label{prop:stability}
 Let the assumptions of Algorithm \ref{algorithm} be satisfied and let $\varepsilon_h\in \mathcal{L}^0(\mathcal{T}_h)$ such that $\varepsilon_h>0$ a.e.\ in $\Omega$ and $\smash{\|\varepsilon_h\|_{L^\infty(\Omega)}}<1$. Then, the following statements apply:
 \begin{itemize}[noitemsep,topsep=2pt,leftmargin=!,labelwidth=\widthof{(iii)}]
    \item[(i)] Algorithm \ref{algorithm} is well-posed, i.e., for every $k\hspace{-0.15em}\in\hspace{-0.15em} \mathbb{N}$, given the most-recent iterate ${u_h^{k-1}\hspace{-0.15em}\in\hspace{-0.15em} \smash{\mathcal{S}^{1,cr}_D(\mathcal{T}_h)}}$, there exists a unique iterate $u_h^k\in \smash{\mathcal{S}^{1,cr}_D(\mathcal{T}_h)}$ solving \eqref{alg:step_1}.
  \item[(ii)] Algorithm \ref{algorithm} is unconditionally strongly stable, i.e., for every  $L\in \mathbb{N}$,~it~holds
  \begin{align*}
  I_{h,\varepsilon_h}^{cr}(u_h^L)+\tau \sum_{k=1}^L{\|d_\tau u_h^k\|_{L^2(\Omega)}^2}\leq I_{h,\varepsilon_h}^{cr}(u_h^0)\,.
  \end{align*}

  \item[(iii)] Algorithm \ref{algorithm} terminates after a finite number of steps, i.e., there exists $k^*\in \mathbb{N}$ such that $\|r_h^{k^*}\|_{L^2(\Omega)}\leq \varepsilon_{stop}^h$.\enlargethispage{6mm}
 \end{itemize}
 \end{proposition}

 The proof of Proposition \ref{prop:stability} (ii) is essentially based on the following inequality.

 \begin{lemma}\label{lem:stability}
  For every $\varepsilon\in (0,1)$ and $a,b\in \mathbb{R}^d$, it holds\enlargethispage{10mm}
  \begin{align*}
  \tfrac{f_\varepsilon'(\vert a\vert)}{\vert a\vert } b\cdot(b-a)\ge f_\varepsilon(\vert b\vert)-f_\varepsilon(\vert a\vert)+\tfrac{1}{2}\tfrac{f_\varepsilon'(\vert a\vert)}{\vert a\vert}\vert b-a\vert^2\,.
  \end{align*}
 \end{lemma}

 \begin{proof}
    Follows from \cite[Appendix A.2]{Bar21}, since $f_\varepsilon\hspace{-0.1em}\in \hspace{-0.1em} C^1(\mathbb{R}_{\ge 0})$ and $(t\hspace{-0.1em}\mapsto\hspace{-0.1em} f_\varepsilon'(t)/t)\hspace{-0.1em}\in\hspace{-0.1em} C^0(\mathbb{R}_{\ge 0})$~is~positive and non-decreasing for all $\varepsilon\in (0,1)$.
 \end{proof}

 \begin{proof}[Proof (of Proposition \ref{prop:stability}).]
 \textit{ad (i).} Since  $\frac{f_\varepsilon'(t)}{t}\ge  0$ for all $\varepsilon\in (0,1)$ and  $t\ge 0$,~the~\mbox{well-posedness} of Algorithm \ref{algorithm} is a direct consequence of the Lax--Milgram lemma.
 
 \textit{ad (ii).}
 Let $L\hspace{-0.1em}\in\hspace{-0.1em} \mathbb{N}$ be arbitrary. Then,
  for every $k\hspace{-0.1em}\in\hspace{-0.1em} \{1,\dots,L\}$, choosing ${v_h\hspace{-0.1em}=\hspace{-0.1em}d_\tau u_h^k\hspace{-0.1em}\in\hspace{-0.1em} \mathcal{S}^{1,cr}_D(\mathcal{T}_h)}$ in \eqref{alg:step_1}, we find that
  \begin{align}\label{prop:stability.1}
  \|d_\tau u_h^k\|_{L^2(\Omega)}^2+\Big(\tfrac{f_{h,\varepsilon_h}'(\vert \nabla_{\!h}u_h^{k-1}\vert )}{\vert \nabla_{\!h}u_h^{k-1}\vert}\nabla_{\!h}u_h^k,\nabla_{\!h} d_\tau u_h^k\Big)_{\Omega}+\alpha\, (\Pi_hu_h^k-g_h,\Pi_h d_\tau u_h^k)_{\Omega}\,.
  \end{align}
  Appealing to Lemma \ref{lem:stability} with $a=\nabla_{\!h}u_h^{k-1}|_T\in \mathbb{R}^d$ and $b=\nabla_{\!h} u_h^k|_T\in \mathbb{R}^d$ applied for all $T\in \mathcal{T}_h$, for every $k\in \{1,\dots,L\}$, we have that
  \begin{align}\label{prop:stability.2}
  \tfrac{f_{h,\varepsilon_h}'(\vert \nabla_{\!h}u_h^{k-1}\vert )}{\vert \nabla_{\!h}u_h^{k-1}\vert}\nabla_{\!h}u_h^k\cdot\nabla_{\!h} d_\tau u_h^k\ge d_\tau f_{h,\varepsilon_h}(\vert \nabla_{\!h}u_h^k\vert )\quad\text{ a.e.\ in }\Omega\,.
  \end{align}
  In addition, since $d_\tau g_h=0$, for every $k\in \{1,\dots,L\}$, we have that
  \begin{align}\label{prop:stability.3}
  \begin{aligned}
  (\Pi_hu_h^k-g_h)\Pi_h d_\tau u_h^k =(\Pi_hu_h^k-g_h)d_\tau(\Pi_h u_h^k-g_h)
  =\tfrac{d_\tau}{2} \vert \Pi_hu_h^k-g_h\vert^2\,.
  \end{aligned}
  \end{align}
  Using \eqref{prop:stability.2} and \eqref{prop:stability.3} in \eqref{prop:stability.1}, for every $k\in \{1,\dots,L\}$, 
  we arrive at
  \begin{align}\label{prop:stability.4}
  \|d_\tau u_h^k\|_{L^2(\Omega)}^2+d_\tau I_{h,\varepsilon_h}^{cr}(u_h^k)\leq 0\,.
  \end{align}
  Summation of \eqref{prop:stability.4} with respect to $k\hspace{-0.1em}\in\hspace{-0.1em}\{1,\dots,L\}$, using ${\sum_{k=1}^L{\hspace{-0.1em}d_\tau I_{h,\varepsilon_h}^{cr}(u_h^k)}\hspace{-0.1em}=\hspace{-0.1em}I_{h,\varepsilon_h}^{cr}(u_h^L)\hspace{-0.1em}-\hspace{-0.1em}I_{h,\varepsilon_h}^{cr}(u_h^0)}$, yields the claimed stability estimate.

 \textit{ad (iii).} Due to (i), we have that $\|d_\tau u_h^k\|_{L^2(\Omega)}^2\to 0$ $(k\to \infty)$, i.e., by the finite-dimensionality of $\smash{\mathcal{S}^{1,cr}_D(\mathcal{T}_h)}$ and the equivalence of norms, it holds
 \begin{align}\label{prop:stability.5}
  u_h^k-u_h^{k-1}\to 0\quad \text{ in }\mathcal{S}^{1,cr}_D(\mathcal{T}_h)\quad (k\to \infty)\,.
 \end{align}
 In addition, due to (i), we have that $I_{h,\varepsilon_h}^{cr}(u_h^k)\leq I_{h,\varepsilon_h}^{cr}(u_h^0)$, which, using Lemma \ref{lem:properties}, implies~that
 $(u_h^k)_{k\in \mathbb{N}}\subseteq \mathcal{S}^{1,cr}_D(\mathcal{T}_h)$ is bounded. Due to the finite-dimensionality of $\mathcal{S}^{1,cr}_D(\mathcal{T}_h)$,~the~\mbox{Bolzano--Weier}-straß theorem yields a subsequence $(u_h^{k_l})_{l\in \mathbb{N}}\subseteq \mathcal{S}^{1,cr}_D(\mathcal{T}_h)$ and a function $\tilde{u}_h\in \mathcal{S}^{1,cr}_D(\mathcal{T}_h)$~such~that
 \begin{align}\label{prop:stability.6}
  u_h^{k_l}\to \tilde{u}_h\quad\text{ in }\mathcal{S}^{1,cr}_D(\mathcal{T}_h)\quad (l\to \infty)\,.
 \end{align}
 Due to \eqref{prop:stability.5}, from \eqref{prop:stability.6}, we deduce that
 \begin{align}\label{prop:stability.7}
  u_h^{k_l-1}\to \tilde{u}_h\quad\text{ in }\mathcal{S}^{1,cr}_D(\mathcal{T}_h)\quad (l\to \infty)\,.
 \end{align}
 As a result, using \eqref{prop:stability.5}--\eqref{prop:stability.7}, by passing for $l\to \infty$ in \eqref{alg:step_1}, for every $v_h\in \mathcal{S}^{1,cr}_D(\mathcal{T}_h)$, we obtain
 \begin{align}\label{prop:stability.8}
  \Big(\tfrac{f_{h,\varepsilon_h}'(\vert \nabla_{\!h}\tilde{u}_h\vert )}{\vert \nabla_{\!h}\tilde{u}_h\vert}\nabla_{\!h}\tilde{u}_h ,\nabla_{\!h}v_h \Big)_{\Omega}+\alpha \,(\Pi_h\tilde{u}_h-g_h,\Pi_hv_h)_{\Omega}=0\,,
 \end{align}
 and, by uniqueness, $\smash{\tilde{u}_h=u_{h,\varepsilon_h}^{cr}}$.
 Hence, using \eqref{prop:stability.5} and \eqref{prop:stability.8},  for every $v_h\in \mathcal{S}^{1,cr}_D(\mathcal{T}_h)$,~we~obtain
 \begin{align*}
  \big(r_h^{k_l},v_h\big)_{\Omega}&=\Big(\tfrac{f_{h,\varepsilon_h}'(\vert \nabla_{\!h}u_h^{k_l}\vert )}{\vert \nabla_{\!h}u_h^{k_l}\vert}\nabla_{\!h}u_h^{k_l},\nabla_{\!h}v_h \Big)_{\Omega}+\alpha \,(\Pi_hu_h^{k_l}-g_h,\Pi_hv_h)_{\Omega}
  \\&\to \Big(\tfrac{f_{h,\varepsilon_h}'(\vert \nabla_{\!h}u_{h,\varepsilon_h}^{cr}\vert )}{\vert \nabla_{\!h}u_{h,\varepsilon_h}^{cr}\vert}\nabla_{\!h}u_{h,\varepsilon_h}^{cr} ,\nabla_{\!h}v_h \Big)_{\Omega}+\alpha \,(\Pi_hu_{h,\varepsilon_h}^{cr}-g_h,\Pi_hv_h)_{\Omega}=0\quad(l\to \infty)\,,
 \end{align*}
 i.e., $r_h^{k_l}\rightharpoonup 0$ in $\mathcal{S}^{1,cr}_D(\mathcal{T}_h)$ $(l\to \infty)$, and, thus, by the finite-dimensionality of $\mathcal{S}^{1,cr}_D(\mathcal{T}_h)$,  $r_h^{k_l}\to 0$ in $\mathcal{S}^{1,cr}_D(\mathcal{T}_h)$ $(l\to \infty)$, which implies that $r_h^{k_l}\to 0$ in $L^2(\Omega)$ $(l\to \infty)$. As this~\mbox{argumentation}~remains valid for each subsequence of $(r_h^k)_{k\in \mathbb{N}}\subseteq \mathcal{S}^{1,cr}_D(\mathcal{T}_h)$, the standard convergence principle yields that $r_h^k\to 0$ in $L^2(\Omega)$ $(k\to \infty)$. In particular, there exists $k^*\in \mathbb{N}$  such that $\|r_h^{k^*}\|_{L^2(\Omega)}\leq \varepsilon^h_{\textit{stop}}$.
 \end{proof}

    \subsection{Implementation details regarding the adaptive mesh refinement procedure}\label{subsec:num_a_posteriori}\vspace{-1mm}\enlargethispage{8mm}
    
    \hspace{5mm}Before we present numerical experiments, we briefly outline the  details of the implementations regarding the adaptive mesh refinement procedure. 
    In general, we follow the \textit{adaptive algorithm}, cf.\ \cite{AO00,CKNS08,Ver13}:
    
	\begin{algorithm}[AFEM]\label{alg:afem}
		Let $\varepsilon_{\textup{STOP}}>0$, $\theta\in (0,1]$, and  $\mathcal{T}_0$ an initial  triangulation~of~$\Omega$, and 
        choose $\varepsilon_0\in \mathcal{L}^0(\mathcal{T}_0)$ such that $\varepsilon_0>0$ a.e.\ in $\Omega$ and $\|\varepsilon_0\|_{L^\infty(\Omega)}<1$.
  Then,~for~every~$i\in \mathbb{N}\cup\{0\}$:
	\begin{description}[noitemsep,topsep=1pt,labelwidth=\widthof{\textit{('Estimate')}},leftmargin=!,font=\normalfont\itshape]
		\item[('Solve')]\hypertarget{Solve}{}
		Approximate the regularized, discrete primal solution $u_i^{cr}\coloneqq u_{h_i,\varepsilon_i}^{cr} \in\mathcal{S}^{1,cr}_D(\mathcal{T}_i)$~mini-mizing \eqref{discrete-ROF-primal_reg}.  Post-process $u_i^{cr}\hspace{-0.1em}\in\hspace{-0.1em} \smash{\mathcal{S}^{1,\textit{cr}}_D(\mathcal{T}_i)}$ to obtain 
        an approximation ${\overline{u}_i^{cr}\hspace{-0.1em}\in\hspace{-0.1em} \mathcal{S}^{1,\textit{cr}}_D(\mathcal{T}_i)}$ with $\overline{u}_i^{cr}=0$ on $\partial\Omega$ and 
  a regularized, discrete dual solution $z_i^{rt}\coloneqq z_{h_i,\varepsilon_i}^{rt}\in \mathcal{R}T^0_N(\mathcal{T}_i)$ maximizing \eqref{discrete-ROF-dual_reg}. Then, define 
  \begin{align}\label{eq:projection_step}
      \overline{z}_i^{rt}\coloneqq \tfrac{z_i^{rt}}{\max\{1,\|z_i^{rt}\|_{L^\infty(\Omega;\mathbb{R}^d)}\}} \in \mathcal{R}T^0_N(\mathcal{T}_i)\,.
  \end{align}
		\item[('Estimate')]\hypertarget{Estimate}{} Compute the local refinement indicators $\smash{(\eta^2_{T,\textit{CR}}(\overline{u}_i^{cr},\overline{z}_i^{rt}))_{T\in \mathcal{T}_i}}$,~cf.~Remark~\ref{rem:representations}~(iv). If $\eta^2(\overline{u}_i^{cr},\overline{z}_i^{rt}) \leq \varepsilon_{\textup{STOP}}$, cf. Remark \ref{rem:representations} (iii), then \textup{STOP}; otherwise, continue~with step \textit{('Mark')}.
		\item[('Mark')]\hypertarget{Mark}{}  Choose a minimal (in terms of cardinality) subset $\mathcal{M}_i\subseteq\mathcal{T}_i$ such that
		\begin{align*}
			\sum_{T\in \mathcal{M}_i}{\eta_{T,\textit{CR}}^2(\overline{u}_i^{cr},\overline{z}_i^{rt})}\ge \theta^2\sum_{T\in \mathcal{T}_i}{\eta_{T,\textit{CR}}^2(\overline{u}_i^{cr},\overline{z}_i^{rt})}\,.
		\end{align*}
		\item[('Refine')]\hypertarget{Refine}{} Perform a conforming refinement of $\mathcal{T}_i$ to obtain $\mathcal{T}_{i+1}$~such~that~each $T\in \mathcal{M}_i$  is refined in $\mathcal{T}_{i+1}$.
        Then, construct $\varepsilon_{i+1}\in \mathcal{L}^0(\mathcal{T}_{i+1})$ such that
       $\varepsilon_{i+1}>0$ a.e.\ in $\Omega$ and  $\|\varepsilon_{i+1}\|_{L^\infty(\Omega)}<c_i\,h_{i+1}^2$. 
		Increase~$i\mapsto i+1$~and~continue~with~step~('Solve').
	\end{description}
	\end{algorithm}

    	\begin{remark}
			\begin{description}[noitemsep,topsep=1pt,labelwidth=\widthof{\textit{(iii)}},leftmargin=!,font=\normalfont\itshape]
				\item[(i)] The regularized, discrete  primal solution $u_i^{cr}\in\mathcal{S}^{1,\textit{cr}}_D(\mathcal{T}_i)$ in step (\hyperlink{Solve}{'Solve'}) is computed using 
                the semi-implicit discretized $L^2$-gradient flow, cf.\ Algorithm \ref{algorithm}, for fixed step-size ${\tau=1.0}$, stopping criterion $\varepsilon_{stop}^{h_i}\coloneqq \frac{h_i}{\sqrt{20}}$, and initial condition $u_i^0=0\in \mathcal{S}_D^{1,cr}(\mathcal{T}_i)$. Appealing to Proposition \ref{prop:stability} (ii), Algorithm \ref{algorithm} is unconditionally strongly stable, so that employing the fixed step-size ${\tau=1.0}$ is a reasonable choice. 
                The stopping~criterion~${\varepsilon_{stop}^{h_i}\hspace{-0.1em}\coloneqq \hspace{-0.1em}\frac{h_i}{\sqrt{20}}}$ ensures (cf.\  the argumentation below Algorithm~\ref{algorithm})~that~the final iterate $u_{h_i}^{k^*}\in \mathcal{S}^{1,cr}_D(\mathcal{T}_i)$ is a sufficiently accurate approximation of the discrete primal solution, in the sense 
                that its accuracy does not violate the best possible~linear~convergence~rate, cf. Remark \ref{rem:optimal_rate} (below).
                
                \item[(ii)] As an approximation $\overline{u}_i^{cr}\in \mathcal{S}^{1,\textit{cr}}_D(\mathcal{T}_i)$ with $\overline{u}_i^{cr}=0$ on $\partial\Omega$, we employ 
                \begin{align}
                    \overline{u}_i^{cr}\coloneqq \begin{cases}
                        u_i^{cr}&\text{ if }u_i^{cr}=0\text{ on }\partial\Omega\,,\\
                        I_k^{\partial} u_i^{cr}&\text{ else}\,,
                    \end{cases}
                \end{align}
                where the operator $I_i^{\partial}\colon \mathcal{S}^{1,\textit{cr}}(\mathcal{T}_i)\to \mathcal{S}^{1,\textit{cr}}_D(\mathcal{T}_i)$ for every $v_{h_i}\in\mathcal{S}^{1,\textit{cr}}(\mathcal{T}_i) $ is defined by
                \begin{align}
                    I_i^{\partial}v_i\coloneqq \sum_{S\in \mathcal{S}_{h_i};S\cap \partial\Omega=\emptyset}{v_{h_i}(x_S)\,\varphi_S}\,.
                \end{align}
                \item[(iii)] Note that the particular choices in (ii) are only due to the imposed homogeneous Dirichlet boundary condition. In the case $\Gamma_D\hspace{-0.1em}=\hspace{-0.1em}\emptyset$,  the choice $\overline{u}_i^{cr}\hspace{-0.1em}\coloneqq \hspace{-0.1em}u_i^{cr}\hspace{-0.1em}\in\hspace{-0.1em} \mathcal{S}^{1,cr}(\mathcal{T}_i)$ is always~admissible.
				\item[(iv)] If not otherwise specified, we employ the parameter $\theta=\smash{\frac{1}{2}}$ in (\hyperlink{Estimate}{'Mark'}).
				\item[(v)] To find the set $\mathcal{M}_i\subseteq \mathcal{T}_i$ in step (\hyperlink{Mark}{'Mark'}), we~deploy~the~D\"orfler marking strategy,~cf.~\cite{Doe96}.
				\item[(vi)] The \hspace*{-0.1mm}(minimal) \hspace*{-0.1mm}conforming \hspace*{-0.1mm}refinement \hspace*{-0.1mm}of \hspace*{-0.1mm}$\mathcal{T}_i$ \hspace*{-0.1mm}with \hspace*{-0.1mm}respect \hspace*{-0.1mm}to \hspace*{-0.1mm}$\mathcal{M}_i$~\hspace*{-0.1mm}in~\hspace*{-0.1mm}step \hspace*{-0.1mm}(\hyperlink{Refine}{'Refine'})~\hspace*{-0.1mm}is~\hspace*{-0.1mm}\mbox{obtained}~by deploying the \textit{red}-\textit{green}-\textit{blue}-refinement algorithm, cf.~\cite{Ver13}.
                \item[(vii)] For the construction of the adaptively modified regularization parameter $\varepsilon_i\in \mathcal{L}^0(\mathcal{T}_i)$ in step (\hyperlink{Refine}{'Refine'}), we employ separately the following two cases:
                \begin{align}
                 \varepsilon_i\coloneqq   \begin{cases}
                        \tfrac{\alpha}{d} \vert \Pi_{h_{i-1}} u_{i-1}^{cr}-g_{h_i}\vert  h_i^2 +  h_i^3&(\hypertarget{local}{\textcolor{blue}{local}})\,,\\
                        h_i^2&(\hypertarget{global}{\textcolor{blue}{global}})\,.
                    \end{cases}
                \end{align}
			\end{description}
	\end{remark}

    \newpage
    \subsection{Example with Lipschitz continuous dual solution}

    \hspace{5mm}We examine an example from \cite{BTW21}. In this example, we let $\Omega=(-1,1)^d$, $\Gamma_D=\partial\Omega$, $d\in \{2,3\}$, $r=\smash{\frac{1}{2}}$, $\alpha =10$, and $g=\chi_{B_r^d(0)}\in BV(\Omega)\cap L^\infty(\Omega)$. Then, the primal solution $u\in BV(\Omega)\cap L^\infty(\Omega)$~and a~dual~solution $z\in W^2(\textup{div};\Omega)\cap L^\infty(\Omega;\mathbb{R}^d)$, for a.e.\ $x\in \Omega$ are~defined~by
    \begin{align}\label{one_disk_primal_solution}
        u(x)&\coloneqq (1-\tfrac{d}{\alpha r})\,g(x)\,,\qquad
        z(x)\coloneqq\begin{cases}
            -\tfrac{x}{r}&\textup{ if }\vert x\vert < r\,,\\
            -\tfrac{rx}{\vert x\vert^d}&\textup{ if }\vert x\vert \geq r\,.
        \end{cases}
    \end{align}
    Note that $z\in W^{1,\infty}(\Omega;\mathbb{R}^d)$, so that, appealing to \cite{CP20,Bar21}, uniform mesh-refinement (i.e., $\theta=1$ in Algorithm \ref{alg:afem}) is expected to yield the quasi-optimal  convergence rate $\mathcal{O}(h^{\frac{1}{2}})$.

    \textit{2D Case.}
    The coarsest triangulation $\mathcal{T}_0$ of Figure \ref{fig:OneDisk_triang} (initial triangulation of Algorithm~\ref{alg:afem}) consists of $16$ halved squares. More precisely, Figure \ref{fig:OneDisk_triang} displays
    the triangulations $\mathcal{T}_i$,~${i\in \{0,15,20,25\}}$, generated by  Algorithm \ref{alg:afem}
    using either the adaptively modified $\varepsilon_i\in \mathcal{L}^0(\mathcal{T}_i)$, cf.\ (\hyperlink{local}{local}), or the global choice $\varepsilon_i\coloneqq h_i^2$, cf.\ (\hyperlink{global}{global}). For both choices,
    a refinement towards the circle $\partial B_r^2(0)$, i.e., the jump set $J_u$ of the exact solution $u\in BV(\Omega)\cap L^\infty(\Omega)$, cf.\ \eqref{one_disk_primal_solution}, is reported.
    This behavior is also seen in Figure \ref{fig:OneDisk_solution}, where the regularized, discrete primal~solution~$u_{15}^{\textit{\textrm{cr}}}\in \mathcal{S}^{1,\textit{\textrm{cr}}}_D(\mathcal{T}_{15})$, the (local)
    $L^2$-projection onto element-wise constant functions 
    $\Pi_{h_{15}} u_{15}^{\textit{\textrm{cr}}}\in \mathcal{L}^0(\mathcal{T}_{15})$, and
    the (local) $L^2$-projections onto element-wise affine functions of 
    the modulus of the regularized, discrete dual solution $z_{15}^{\textit{\textrm{rt}}}\in \mathcal{R}T^0_N(\mathcal{T}_{15})$ and of the projected regularized, discrete dual solution $\overline{z}_{15}^{\textit{\textrm{rt}}}\in \mathcal{R}T^0_N(\mathcal{T}_{15})$ are plotted. Figure \ref{fig:OneDisk_triang}, in addition, shows that using the adaptively modified $\varepsilon_i\in \mathcal{L}^0(\mathcal{T}_i)$, cf.\  (\hyperlink{local}{local}), the refinement is more concentrated at the jump set  $J_u$ of the exact solution $u\in BV(\Omega)\cap L^\infty(\Omega)$, cf.\ \eqref{one_disk_primal_solution}. However, in Figure \ref{fig:OneDisk_rate} it is seen that (\hyperlink{local}{local}) does not result in an improved error decay, but an error decay comparable to  (\hyperlink{global}{global}). In addition, 
    Figure \ref{fig:OneDisk_rate} demonstrates that Algorithm~\ref{alg:afem} improves the experimental convergence rate of about $\mathcal{O}(h^{\frac{1}{2}})$ predicted by \cite{CP20,Bar21} for uniform mesh-refinement to the quasi-optimal rate $\mathcal{O}(h)$, cf.~Remark~\ref{rem:optimal_rate}~(below).~In~addition, Figure \ref{fig:OneDisk_rate} indicates the primal-dual error estimator~is~reliable~and~efficient~with~respect~to~the~error~quantity
    \begin{align}\label{eq:reduced_rho}
        \tilde{\rho}^2(\overline{u}_i^{\textit{\textrm{cr}}},\overline{z}_i^{\textit{\textrm{rt}}})\coloneqq \tfrac{\alpha}{2}\|\overline{u}_i^{\textit{\textrm{cr}}}-u\|^2_{L^2(\Omega)}+\tfrac{1}{2\alpha}\|\textup{div}\,\overline{z}_i^{\textit{\textrm{rt}}}-\textup{div}\,z\|^2_{L^2(\Omega)}\,,\quad i\in \mathbb{N}\,,
    \end{align}
    which, appealing to Remark \ref{rmk:examples} (iv), is a lower bound for sum of the optimal convexity measures.\vspace{-1mm}

    \begin{figure}[H]
        \centering
          \hspace*{-2mm}\includegraphics[width=15cm]{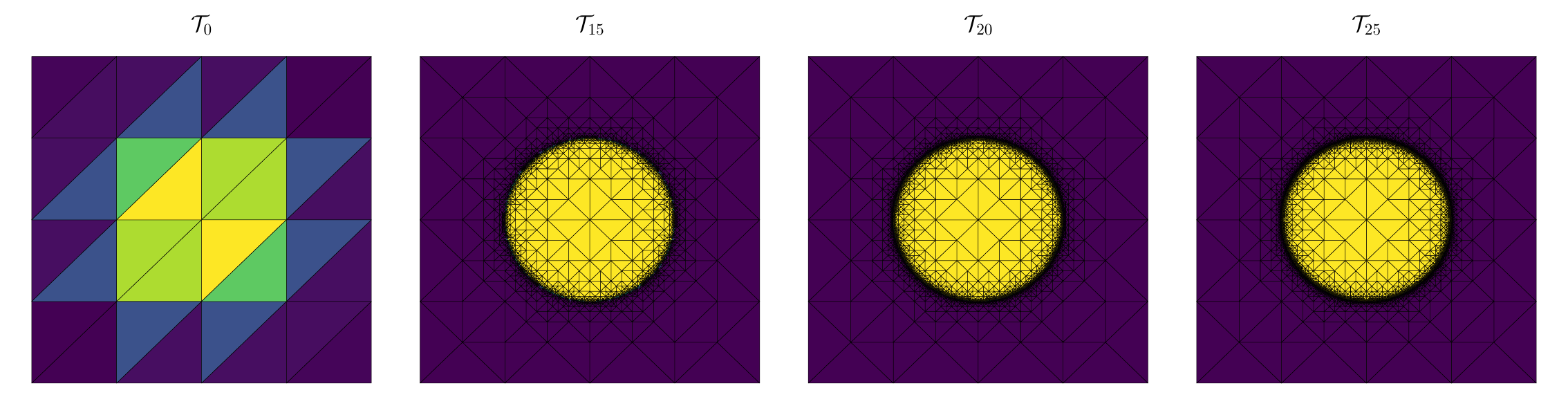}\vspace{-2.5mm}
        \hspace*{-2mm}\includegraphics[width=15cm]{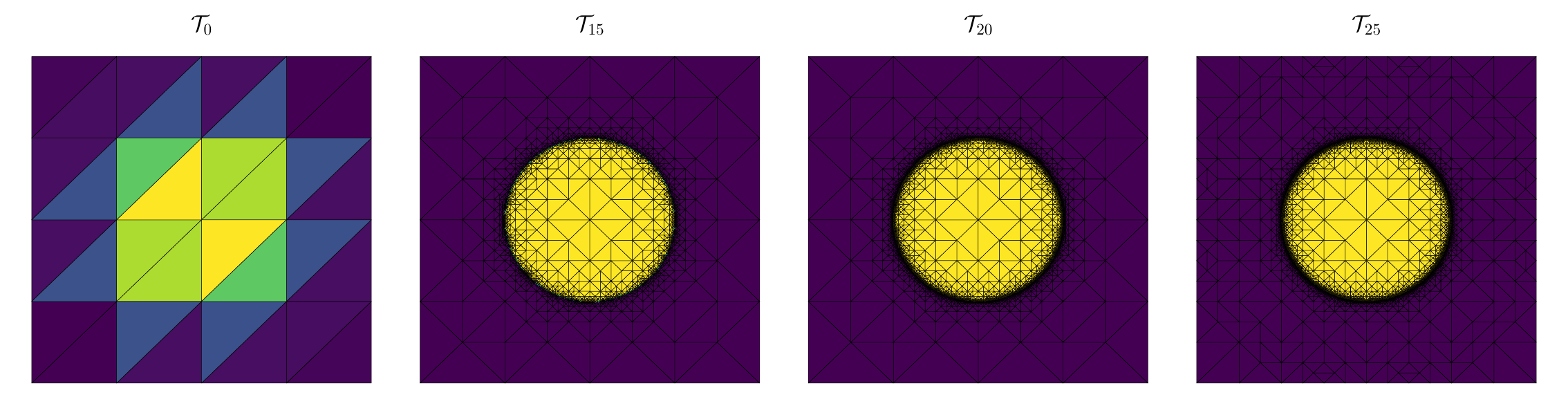}\vspace{-2.5mm}
        \caption{Initial triangulation $\mathcal{T}_0$ and 
        adaptively refined meshes $\mathcal{T}_i$, $i\in \{0,15,20,25\}$, generated by the adaptive Algorithm \ref{alg:afem} (TOP: obtained using (\protect\hyperlink{local}{local}); BOTTOM: obtained using (\protect\hyperlink{global}{global})).}
        \label{fig:OneDisk_triang}
    \end{figure}

    \begin{figure}[H]
        \centering
        \hspace*{-5mm}\includegraphics[width=14cm]{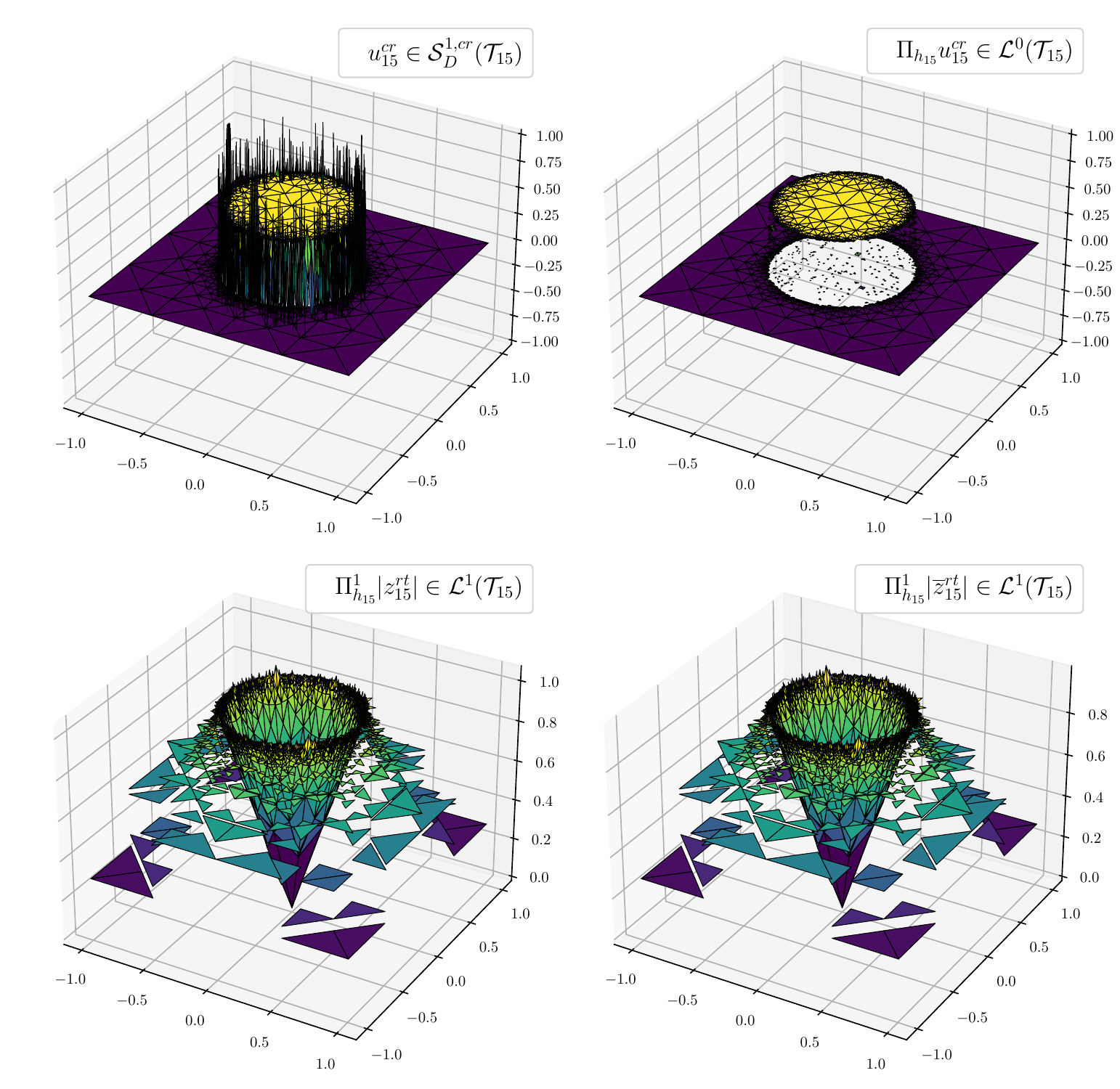}\vspace{-2.5mm}
        \caption{UPPER LEFT: Plot of $u_{15}^{\textit{\textrm{cr}}}\hspace{-0.1em}\in \hspace{-0.1em}\mathcal{S}^{1,\textit{\textrm{cr}}}_D(\mathcal{T}_{15})$, UPPER RIGHT: Plot of $\Pi_{h_{15}}u_{15}^{\textit{\textrm{cr}}}\hspace{-0.1em}\in\hspace{-0.1em} \mathcal{L}^0(\mathcal{T}_{15})$; LOWER LEFT: Plot of  $\Pi_{h_{15}}^1\vert z_{15}^{\textit{\textrm{rt}}}\vert \hspace{-0.1em}\in \hspace{-0.1em}\mathcal{L}^1(\mathcal{T}_{15})$; LOWER RIGHT:  Plot of  $\Pi_{h_{15}}^1\vert \overline{z}_{15}^{\textit{\textrm{rt}}}\vert\in \mathcal{L}^1(\mathcal{T}_{15})$; each obtained using (\protect\hyperlink{local}{local}).}
        \label{fig:OneDisk_solution}
    \end{figure}\vspace{-5mm}\enlargethispage{7mm}

    \begin{figure}[H]
        \centering
        \hspace*{-2mm}\includegraphics[width=14.5cm]{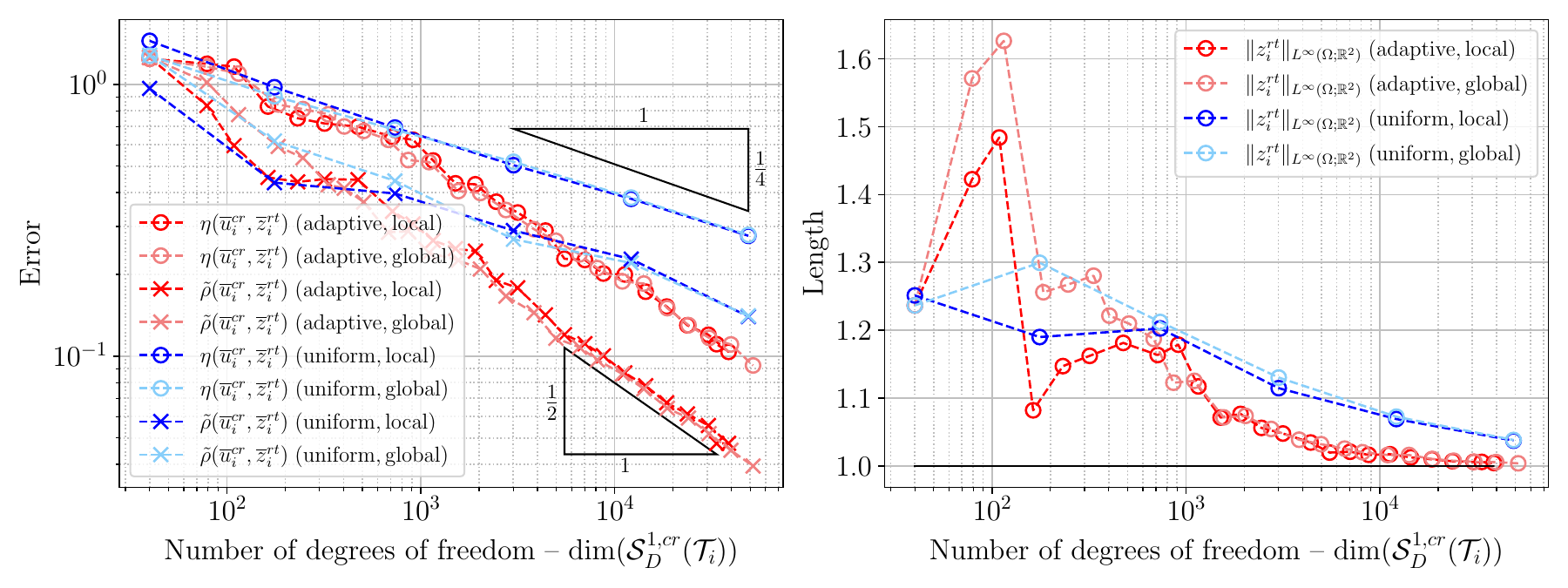}\vspace{-2.5mm}
        \caption{LEFT: Plots of $\eta(\overline{u}_i^{\textit{\textrm{cr}}},\overline{z}_i^{\textit{\textrm{rt}}})$ and $\tilde{\rho}(\overline{u}_i^{\textit{\textrm{cr}}},\overline{z}_i^{\textit{\textrm{rt}}})$ using adaptive mesh refinement for $i=0,\dots,25$ and uniform mesh refinement for $i=0,\dots, 5$; RIGHT: Plots of  $\|\overline{z}_i^{\textit{\textrm{rt}}}\|_{L^\infty(\Omega;\mathbb{R}^2)}$  using adaptive mesh refinement for $i=0,\dots,25$ and uniform mesh refinement for  $i=0,\dots, 5$.}
        \label{fig:OneDisk_rate}
    \end{figure}

    \textit{3D Case.} The initial triangulation  $\mathcal{T}_0$  of Algorithm~\ref{alg:afem} consists of $27$  cubes each divided into six tetrahedrons. Using either the adaptively modified $\varepsilon_i\in \mathcal{L}^0(\mathcal{T}_i)$, cf.\ (\hyperlink{local}{local}), or the global choice $\varepsilon_i\coloneqq h_i^2$, cf.\ (\hyperlink{global}{global}), we report similar results to~the~2D~case:~for~both~choices,
    a refinement towards the sphere $\partial B_r^3(0)$, i.e., the jump set $J_u$ of the exact solution ${u\in BV(\Omega)\cap L^\infty(\Omega)}$,~cf.~\eqref{one_disk_primal_solution},~is~re-ported, \hspace{-0.15mm}which \hspace{-0.15mm}can \hspace{-0.15mm}be \hspace{-0.15mm}seen
    \hspace{-0.15mm}in \hspace{-0.15mm}Figure \hspace{-0.15mm}\ref{fig:OneDisk3D_solution}, \hspace{-0.15mm}where \hspace{-0.15mm}the \hspace{-0.15mm}regularized,~\hspace{-0.15mm}discrete~\hspace{-0.15mm}primal~\hspace{-0.15mm}solution~$\hspace{-0.1mm}{u_{10}^{\textit{\textrm{cr}}}\hspace{-0.15em}\in \hspace{-0.15em}\mathcal{S}^{1,\textit{\textrm{cr}}}_D(\mathcal{T}_{10})}$ and 
    the (local) $L^2$-projection onto element-wise affine functions of 
    the modulus of the regularized, discrete dual solution $z_{10}^{\textit{\textrm{rt}}}\in \mathcal{R}T^0_N(\mathcal{T}_{10})$ are plotted. 
    Figure \ref{fig:OneDisk_rate} shows that the adaptive Algorithm~\ref{alg:afem} improves the experimental convergence rate of about $\mathcal{O}(\smash{h^{\frac{1}{2}}})$ predicted by \cite{CP20,Bar21} for uniform mesh-refinement to the quasi-optimal rate $\mathcal{O}(h)$, cf. Remark \ref{rem:optimal_rate} (below).\vspace{-0.4cm}
    
    \begin{figure}[H]
        \centering
        \hspace*{-5mm}
        \includegraphics[width=5.75cm]{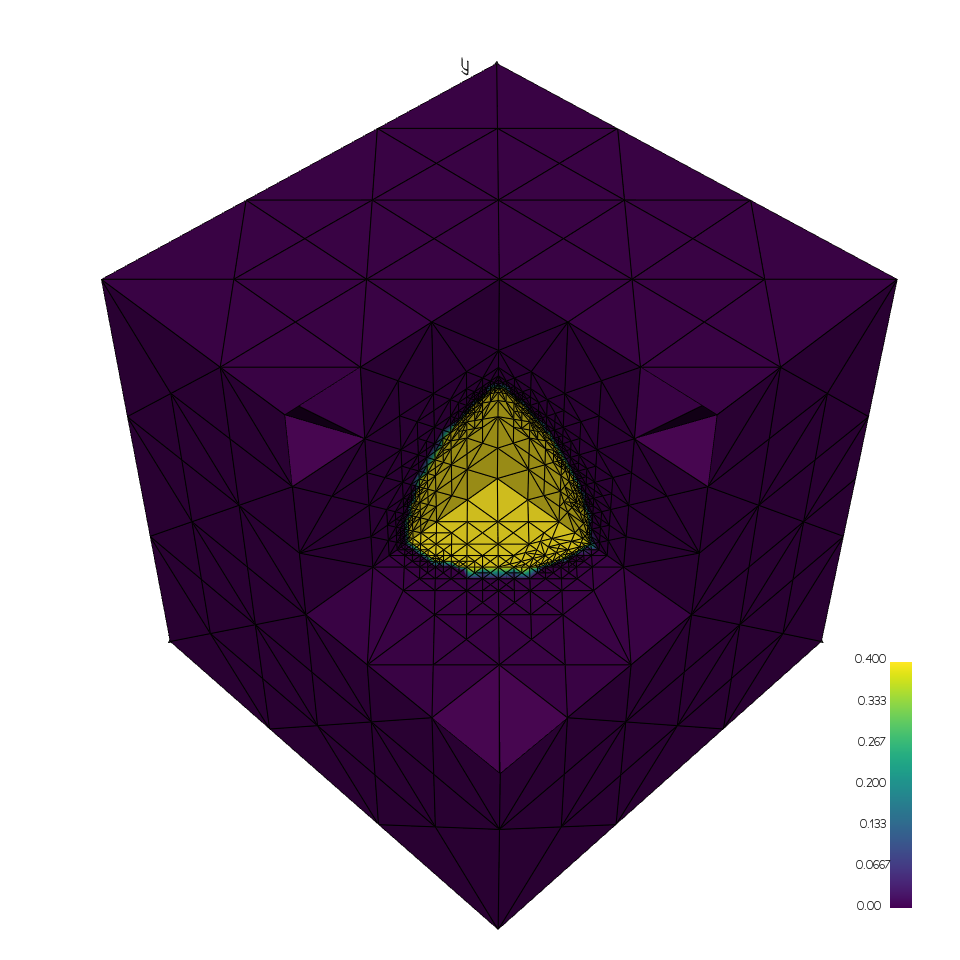}
        \hspace{5mm}\includegraphics[width=5.75cm]{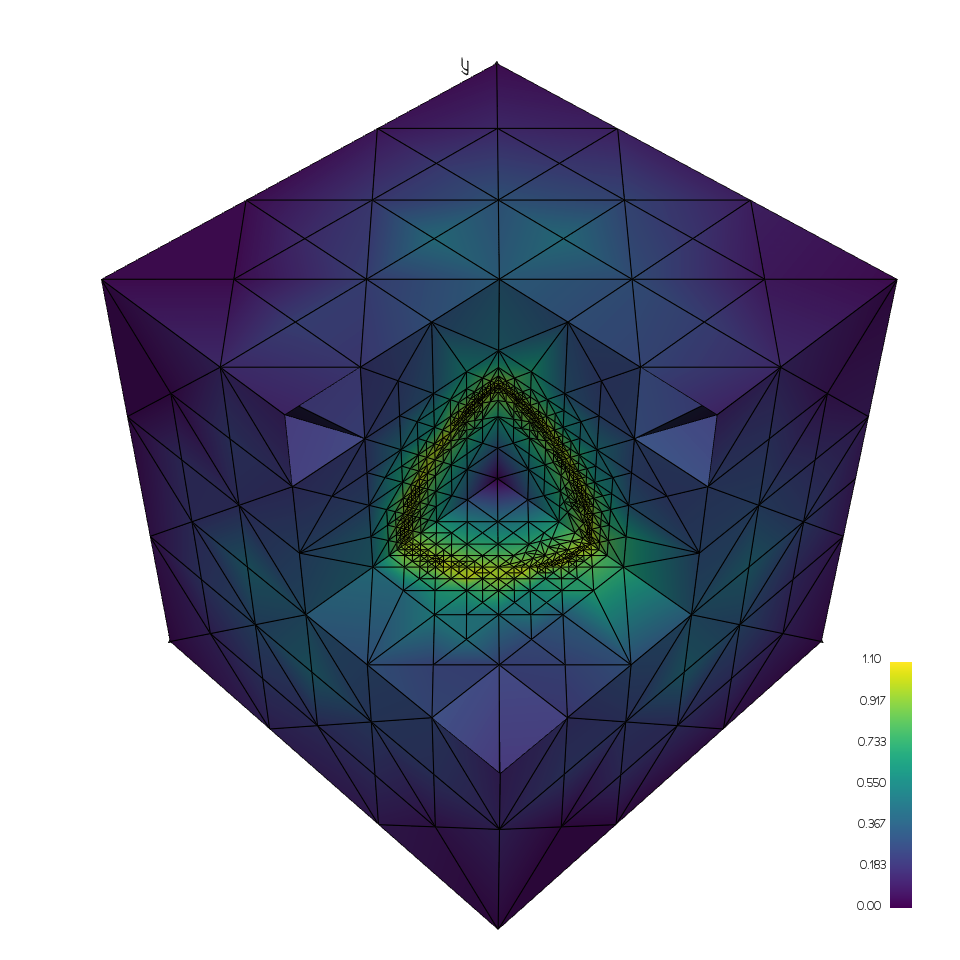}\vspace{-1mm}
        \caption{LEFT: Plot of $u_{10}^{\textit{\textrm{cr}}}\in\mathcal{S}^{1,\textit{\textrm{cr}}}_D(\mathcal{T}_{10})$; RIGHT: Plot of ${\Pi_{h_{10}}^1\vert z_{10}^{\textit{\textrm{rt}}}\vert \in \mathcal{L}^1(\mathcal{T}_{10})}$; each obtained using (\protect\hyperlink{local}{local}).}
        \label{fig:OneDisk3D_solution}
    \end{figure}\vspace{-5mm}\enlargethispage{12.5mm}

    \begin{figure}[H]
        \centering
        \hspace*{-2mm}\includegraphics[width=14.5cm]{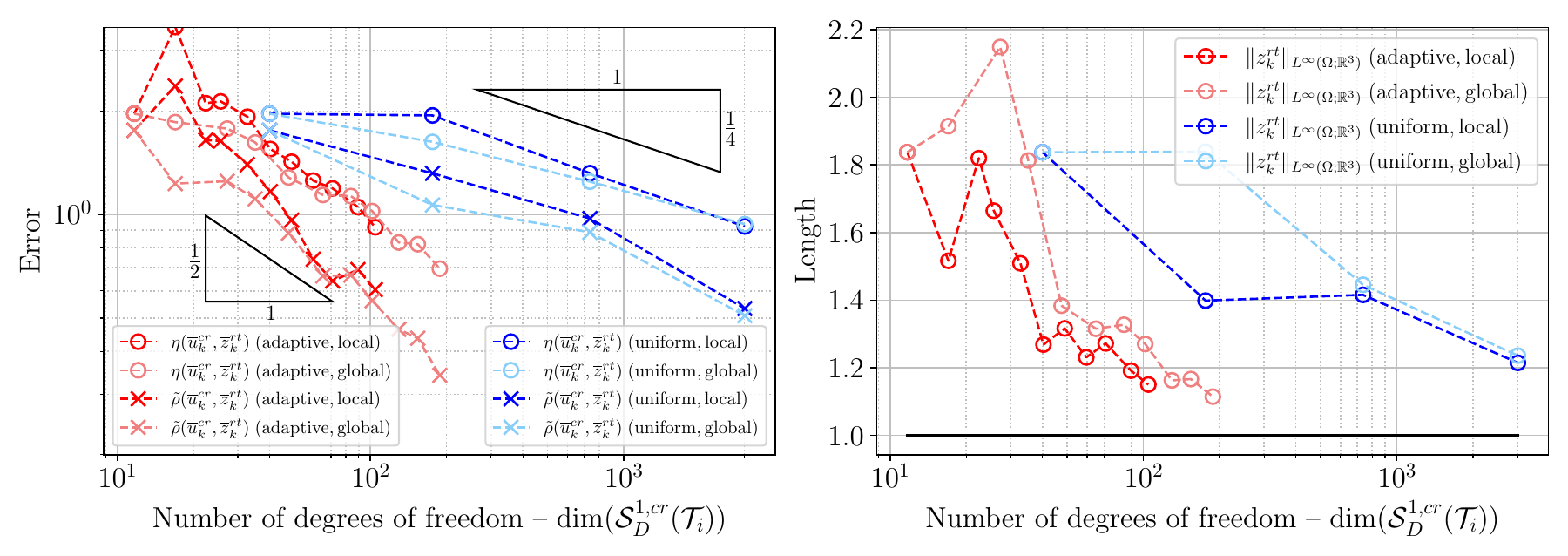}\vspace{-2.5mm}
        \caption{LEFT: Plots of $\eta(\overline{u}_i^{\textit{\textrm{cr}}},\overline{z}_i^{\textit{\textrm{rt}}})$ and $\tilde{\rho}(\overline{u}_i^{\textit{\textrm{cr}}},\overline{z}_i^{\textit{\textrm{rt}}})$ using adaptive mesh refinement for $i=0,\dots,10$ and uniform mesh refinement for $i=0,\dots, 3$; RIGHT: Plots of  $\|\overline{z}_i^{\textit{\textrm{rt}}}\|_{L^\infty(\Omega;\mathbb{R}^3)}$  using adaptive mesh refinement for $i=0,\dots,10$ and uniform mesh refinement for  $i=0,\dots, 3$.}
        \label{fig:OneDisk3D_rate}
    \end{figure}\vspace{-3mm}

    \begin{remark}[A Comment on the optimality of linear convergence rates]\label{rem:optimal_rate}  In one dimension, the $L^2$-best-approximation error of the sign function on quasi-uniform
partitions is of order $\mathcal{O}(\smash{h^{\frac{1}{2}}})$, cf.\ \cite[Example 10.5]{Bar15}. More generally, using that the
intersection $BV(\Omega) \cap L^\infty(\Omega)$ is contained in
fractional Sobolev spaces $W^{s,2}(\Omega)$ for all $s<1/2$,
cf.~\cite[Lemma 38.1]{Tart07-book}, one cannot expect a higher convergence rate
than $\mathcal{O}(\smash{h^{\frac{1}{2}}})$ for generic, essentially bounded functions~of~bounded~variation. For triangulations that are graded towards the jump
sets of certain discontinuous functions with a quadratic grading
strength, i.e., the local mesh-size satisfies
$h_T \sim h^2$ for all elements $T\in \mathcal{T}_h$ at the discontinuity set, with the average mesh-size $h\sim\textup{card}(\mathcal{N}_h)^{-1/d}$, a linear
convergence rate $\mathcal{O}(h)$ has been established in~\cite{BTW21}. Since our
error estimates not only bound squared $L^2$-errors but also control
squares of $L^p$-norms of non-linear error quantities involving~derivatives,~cf.~\mbox{\cite[Remark~5.4]{BTW21}}, a higher convergence rate than linear cannot be expected.
In view of these aspects, the linear convergence rate $\mathcal{O}(h)$ for
the devised adaptive strategy is quasi-optimal.
    \end{remark}

    \pagebreak
    
    \subsection{Example without Lipschitz continuous dual solution}\enlargethispage{3mm}

    \hspace{5mm}We examine an example from \cite{BTW21}. In this example, we let $\Omega=(-1.5,1.5)^2$, $\Gamma_D=\partial\Omega$, $r=\smash{\frac{1}{2}}$, $\alpha =10$, and $g=\chi_{B_r^2(r\mathrm{e}_1)}-\chi_{B_r^2(-r\mathrm{e}_1)}\in BV(\Omega)\cap L^\infty(\Omega)$. Then, the primal solution $u\in BV(\Omega)\cap L^\infty(\Omega)$ and a dual solution $z\in W^2(\textup{div};\Omega)\cap L^\infty(\Omega;\mathbb{R}^2)$, for a.e. $x\in \Omega$ are~defined~by
    \begin{align}\label{two_disk_primal_solution}
        u(x)\coloneqq (1-\tfrac{2}{\alpha r})\,g(x)\,,\qquad 
        z(x)\coloneqq\begin{cases}
            \mp\tfrac{x\mp r \mathrm{e}_1}{r}&\textup{ if }\vert x\mp r \mathrm{e}_1\vert < r\,,\\
            \mp\tfrac{r(x\mp r \mathrm{e}_1)}{\vert x\mp r \mathrm{e}_1\vert^2}&\textup{ if }\vert x\mp r \mathrm{e}_1\vert \geq r\,.
        \end{cases}
    \end{align}
    Note that $z\notin W^{1,\infty}(\Omega;\mathbb{R}^2)$, so that we cannot refer to \cite{CP20,Bar21} in order~to~expect~uniform~mesh-refinement to yield the convergence rate $\mathcal{O}(h^{\frac{1}{2}})$.
    However, since $z|_{\Omega^{\pm}}\in  W^{1,\infty}(\Omega^{\pm};\mathbb{R}^2)$, where $\Omega^+\hspace{-0.1em}\coloneqq \hspace{-0.1em} \Omega\cap (\mathbb{R}_{>0}\times\mathbb{R})$ and  $\Omega^-\hspace{-0.1em}\coloneqq\hspace{-0.1em} \Omega\cap (\mathbb{R}_{<0}\times\mathbb{R})$, and since the coarsest triangulation $\mathcal{T}_0$~of~Figure~\ref{fig:TwoDisk_triang} and, hence, also all resulting refinements $\mathcal{T}_i$, $i\in \mathbb{N}$, of $\mathcal{T}_0$ resolve $J_z\coloneqq \Omega\cap (\{0\}\times\mathbb{R})$, i.e., the jump set of
 $z\in W^2(\textup{div};\Omega)\cap L^\infty(\Omega;\mathbb{R}^2)$, in the sense that $J_z\subseteq \bigcup_{S\in \mathcal{S}_{h_i}}{S}$ for all $i\in \mathbb{N}$,
  referring to \cite[Theorem 4.5]{BKROF22}, we can expect uniform mesh-refinement to yield the convergence rate $\mathcal{O}(h^{\smash{\frac{1}{2}}})$.

The coarsest triangulation $\mathcal{T}_0$ of Figure \ref{fig:TwoDisk_triang} (initial triangulation of Algorithm \ref{alg:afem})~consists~of~$16$ halved squares. More precisely, Figure \ref{fig:OneDisk_triang} displays
    the triangulations $\mathcal{T}_i$, $i\!\in\! \{0,15,20,25\}$,~generated by  Algorithm \ref{alg:afem}
    using either the adaptively modified $\varepsilon_i\!\in\! \mathcal{L}^0(\mathcal{T}_i)$,~cf.~(\hyperlink{local}{local}), or the global choice $\varepsilon_i\coloneqq h_i^2$, cf. (\hyperlink{global}{global}). For both choices,
    a refinement towards ${\partial B_r^2(r\mathrm{e}_1)\cup \partial B_r^2(-r\mathrm{e}_1)}$, i.e., the jump set $J_u$ of the exact solution $u\in BV(\Omega)\cap L^\infty(\Omega)$, cf.\ \eqref{two_disk_primal_solution}, is reported.
    This behavior is also seen in Figure \ref{fig:TwoDisk_solution}, where the regularized, discrete primal solution~$u_{15}^{\textit{\textrm{cr}}}\in \mathcal{S}^{1,\textit{\textrm{cr}}}_D(\mathcal{T}_{15})$, the (local)
    $L^2$-projection onto element-wise constant functions 
    $\Pi_{h_{15}} u_{15}^{\textit{\textrm{cr}}}\in \mathcal{L}^0(\mathcal{T}_{15})$, and 
    the (local) $L^2$-projections onto element-wise affine functions of 
    the modulus of the regularized, discrete dual~solution $z_{15}^{\textit{\textrm{rt}}}\in \mathcal{R}T^0_N(\mathcal{T}_{15})$ and of the scaled regularized, discrete dual solution~${\overline{z}_{15}^{\textit{\textrm{rt}}}\in \mathcal{R}T^0_N(\mathcal{T}_{15})}$~are~plotted. Figure \ref{fig:TwoDisk_triang}, in addition, shows that employing the adaptively modified regularization~\mbox{parameter}, cf.\  (\hyperlink{local}{local}), the refinement is more concentrated at the jump set  $J_u$ of the exact solution $u\in BV(\Omega)\cap L^\infty(\Omega)$,~cf.~\eqref{two_disk_primal_solution}. However, in Figure \ref{fig:TwoDisk_rate} it can be seen that (\hyperlink{local}{local}) does not result in an improved error decay, but an error decay comparable to  (\hyperlink{global}{global}). In addition, Figure \ref{fig:TwoDisk_rate} demonstrates that  Algorithm \ref{alg:afem} improves the experimental convergence rate of about $\mathcal{O}(h^{\smash{\frac{1}{2}}})$ predicted by \cite[Theorem~4.5]{BKROF22} for uniform mesh-refinement to the quasi-optimal~rate~$\mathcal{O}(h)$,~cf.~\mbox{Remark}~\ref{rem:optimal_rate}. In addition, Figure \ref{fig:TwoDisk_rate} indicates the primal-dual error estimator is both reliable and efficient~with~respect~to~the~error~quantity~\eqref{eq:reduced_rho}.\vspace{-1mm}

    \begin{figure}[H]
        \centering
        \hspace*{-2mm}\includegraphics[width=15cm]{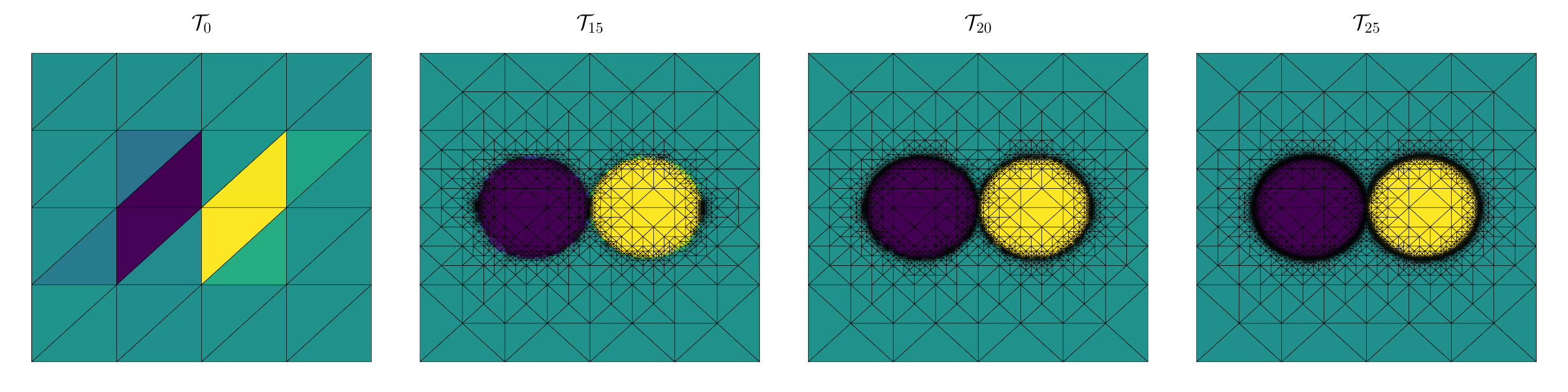}\vspace{-2.5mm}
        \hspace*{-2mm}\includegraphics[width=15cm]{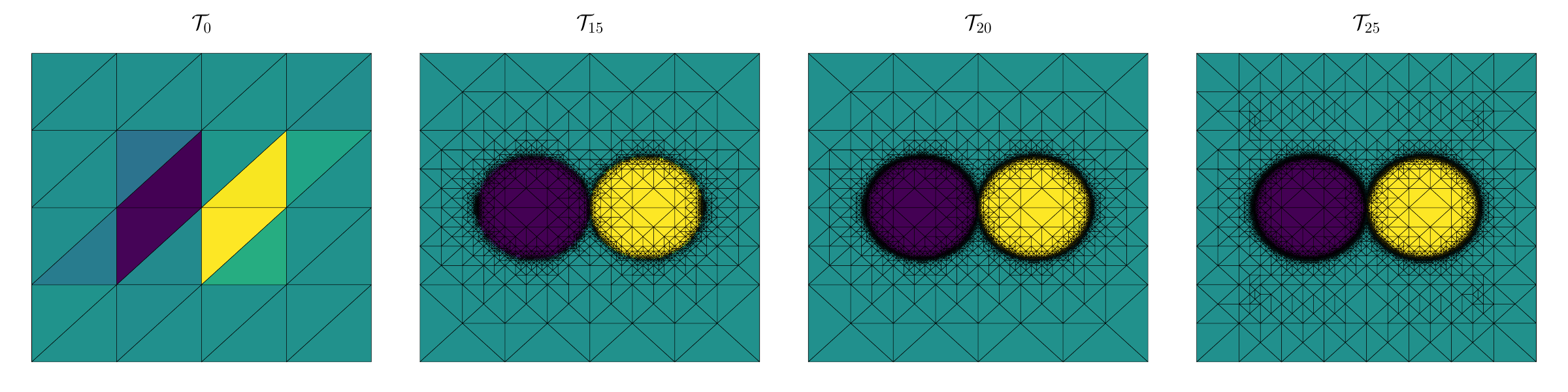}\vspace{-2.5mm}
        \caption{Initial triangulation $\mathcal{T}_0$ and 
        adaptively refined meshes $\mathcal{T}_i$, $i\in \{0,10,20,25\}$, generated by the adaptive Algorithm \ref{alg:afem} (TOP: obtained using (\protect\hyperlink{local}{local}); BOTTOM: obtained using (\protect\hyperlink{global}{global})).}
        \label{fig:TwoDisk_triang}
    \end{figure}

    \begin{figure}[H]
        \centering
        \hspace*{-5mm}\includegraphics[width=14cm]{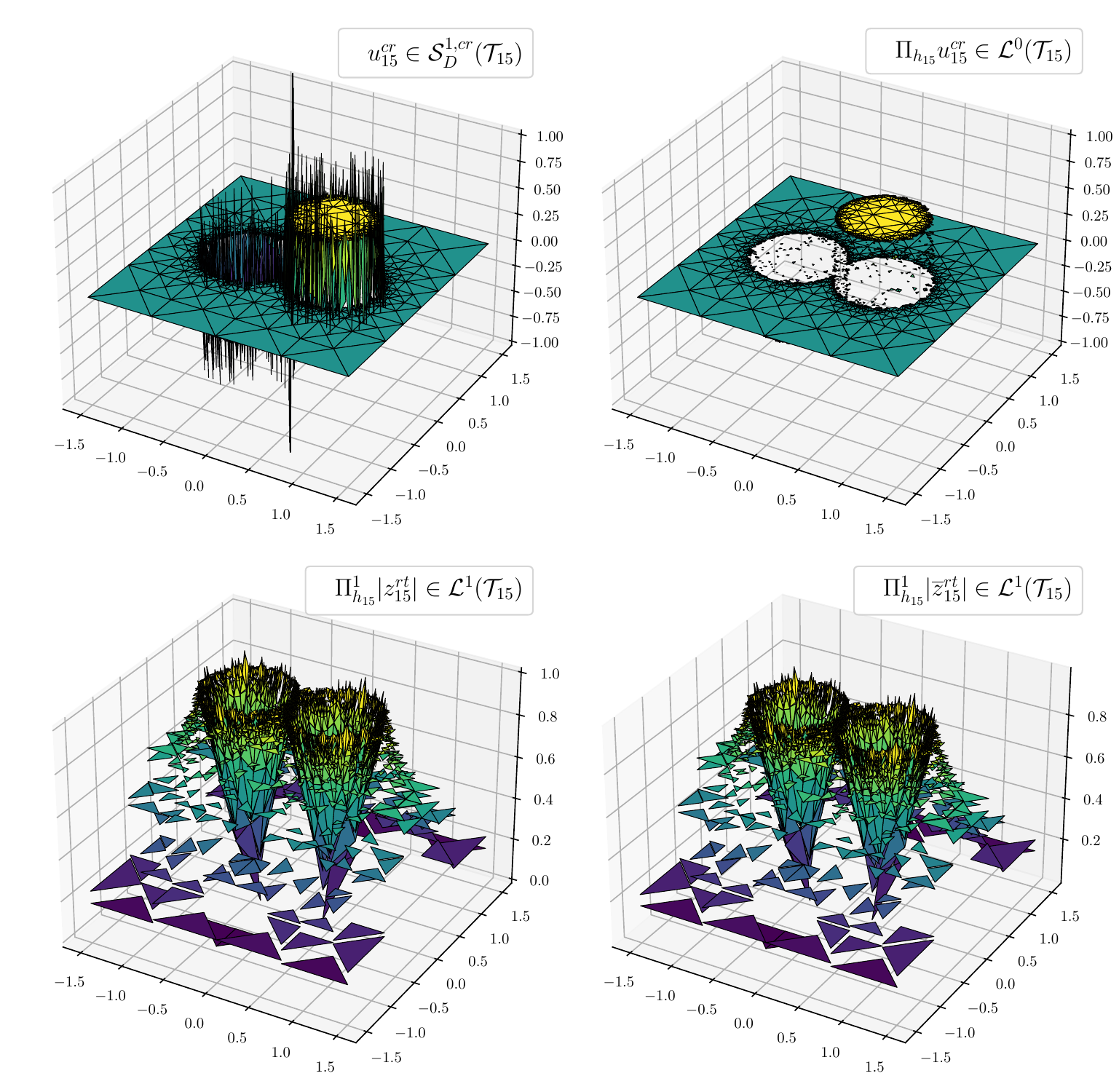}\vspace{-2.5mm}
        \caption{UPPER LEFT: Plot of $u_{15}^{\textit{\textrm{cr}}}\hspace{-0.1em}\in \hspace{-0.1em}\mathcal{S}^{1,\textit{\textrm{cr}}}_D(\mathcal{T}_{15})$, UPPER RIGHT: Plot of $\Pi_{h_{15}}u_{15}^{\textit{\textrm{cr}}}\hspace{-0.1em}\in\hspace{-0.1em} \mathcal{L}^0(\mathcal{T}_{15})$; LOWER LEFT: Plot of  $\Pi_{h_{15}}^1\vert z_{15}^{\textit{\textrm{rt}}}\vert \hspace{-0.1em}\in \hspace{-0.1em}\mathcal{L}^1(\mathcal{T}_{15})$; LOWER RIGHT:  Plot of  $\Pi_{h_{15}}^1\vert \overline{z}_{15}^{\textit{\textrm{rt}}}\vert\in \mathcal{L}^1(\mathcal{T}_{15})$; each obtained using (\protect\hyperlink{local}{local}).}
        \label{fig:TwoDisk_solution}
    \end{figure}\vspace{-5mm}\enlargethispage{7mm}

    \begin{figure}[H]
        \centering
        \hspace*{-2mm}\includegraphics[width=14.5cm]{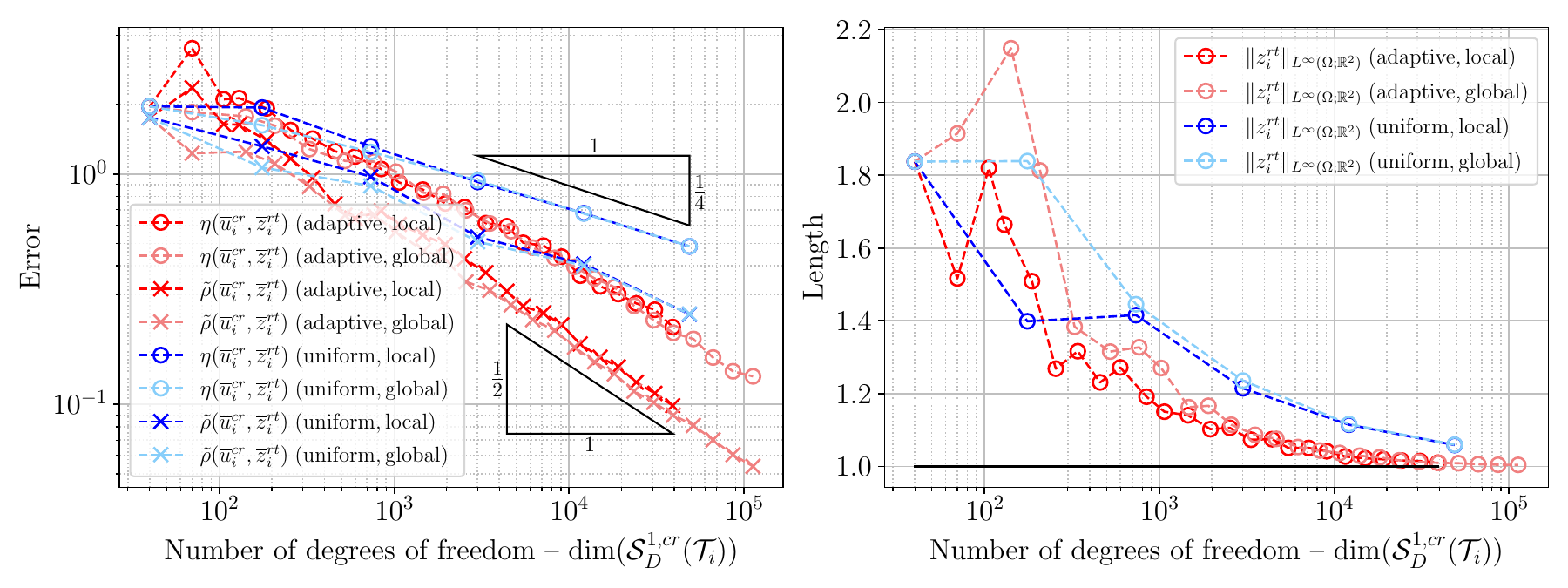}\vspace{-2.5mm}
        \caption{LEFT: Plots of $\eta(\overline{u}_i^{\textit{\textrm{cr}}},\overline{z}_i^{\textit{\textrm{rt}}})$ and $\tilde{\rho}(\overline{u}_i^{\textit{\textrm{cr}}},\overline{z}_i^{\textit{\textrm{rt}}})$ using adaptive mesh refinement for $i=0,\dots,25$ and uniform mesh refinement for $i=0,\dots, 5$; RIGHT: Plots of  $\|\overline{z}_i^{\textit{\textrm{rt}}}\|_{L^\infty(\Omega;\mathbb{R}^2)}$  using adaptive mesh refinement for $i=0,\dots,25$ and uniform mesh refinement for  $i=0,\dots, 5$.}
        \label{fig:TwoDisk_rate}
    \end{figure}

    \subsection{Example with Lipschitz continuous primal solution and Lipschitz continuous dual solution}

    \hspace{5mm}We examine an example from \cite{BDN18}. In this example, we let $\Omega=(-1.5,1.5)^2$, $\Gamma_D=\partial\Omega$,~$\alpha =10$,   $s(t)\coloneqq \sqrt{3t}$ and $r(t)\coloneqq \frac{1}{2}\sqrt{1-4t}$ for $t=0.1$, and $g\in BV(\Omega)\cap L^\infty(\Omega)$ for a.e.\ $x\in \Omega$, be~defined~by\enlargethispage{2mm}
    \begin{align*}
        g(x)\coloneqq \begin{cases}
            1 +\frac{2-\alpha(s(t)^2+t)}{s(t)}&\text{ if }\vert x\vert\leq s(t)\,,\\
            1 +\frac{1-\alpha(\vert x\vert^2+t)}{\vert x\vert}&\text{ if }s(t)<\vert x\vert\leq r(t)\,,\\
            0&\text{ else}\,.
        \end{cases}
    \end{align*}
    Then, the primal solution $u\in BV(\Omega)\cap L^\infty(\Omega)$ and a dual solution $z\in W^2(\textup{div};\Omega)\cap L^\infty(\Omega;\mathbb{R}^2)$ with $\vert z\vert\leq 1$ a.e.\ in $\Omega$, for a.e. $x\in \Omega$ are defined by
    \begin{align}\label{cone_primal_solution}
        u(x)\coloneqq \begin{cases}
            1 - \frac{s(t)^2+t}{s(t)}&\text{ if }\vert x\vert\leq s(t)\,,\\
            1 -\frac{\vert x\vert^2+t}{\vert x\vert}&\text{ if }s(t)<\vert x\vert\leq r(t)\,,\\
            0&\text{ else}\,,
        \end{cases}\quad 
        z(x)\coloneqq\begin{cases}
           -\frac{x}{s(t)}&\text{ if }\vert x\vert\leq s(t)\,,\\
           -\frac{x}{\vert x\vert}&\text{ if }s(t)<\vert x\vert\leq r(t)\,,\\
          -\frac{xr(t)}{\vert x\vert^2}&\text{ else}\,.
        \end{cases}
    \end{align}
    Note that $z\in \smash{W^{1,\infty}(\Omega;\mathbb{R}^2)}$, so that, appealing to \cite{CP20,Bar21}, uniform mesh-refinement is expected to yield the quasi-optimal convergence rate $\mathcal{O}(\smash{h^{\frac{1}{2}}})$.

    The coarsest triangulation $\mathcal{T}_0$ of Figure \ref{fig:Cone_triang} (initial triangulation of Algorithm \ref{alg:afem})~consists~of~$16$ halved squares. More precisely, Figure \ref{fig:Cone_triang} displays
    the triangulations $\mathcal{T}_i$, $i\!\in\! \{0,5,10,15\}$,~generated by Algorithm \ref{alg:afem}
    employing either  $\varepsilon_i\in \mathcal{L}^0(\mathcal{T}_i)$,~cf.~(\hyperlink{local}{local}), or  $\varepsilon_i\coloneqq h_i^2$, cf. (\hyperlink{global}{global}). For both choices,
    a refinement mainly towards~and~on~the~set $\{\vert \nabla u\vert >0\}$  is reported.
    This is also seen in Figure \ref{fig:Cone_solution}, where the regularized, discrete primal solution $u_{15}^{\textit{\textrm{cr}}}\in \mathcal{S}^{1,\textit{\textrm{cr}}}_D(\mathcal{T}_{10})$, the (local)
    $L^2$-projection onto element-wise constant functions 
    $\Pi_{h_{10}} u_{10}^{\textit{\textrm{cr}}}\in \mathcal{L}^0(\mathcal{T}_{10})$, and
    the (local) $L^2$-projections onto element-wise affine functions of 
    the modulus of the regularized, discrete dual solution $z_{10}^{\textit{\textrm{rt}}}\in \mathcal{R}T^0_N(\mathcal{T}_{10})$ and of the scaled regularized, discrete dual solution $\overline{z}_{10}^{\textit{\textrm{rt}}}\in \mathcal{R}T^0_N(\mathcal{T}_{10})$~are~plotted.  Figure \ref{fig:Cone_triang} shows that employing the adaptively modified regularization parameter, cf.\  (\hyperlink{local}{local}), the refinement takes place at and on the set  $\{\vert \nabla u\vert >0\}$. However, in Figure \ref{fig:Cone_rate}, again, it can be seen that (\hyperlink{local}{local}) does not result in an improved error decay, but an error decay comparable to (\hyperlink{global}{global}). In addition, Figure \ref{fig:Cone_rate} demonstrates that  Algorithm \ref{alg:afem} improves the experimental convergence rate of about $\mathcal{O}(h^{\smash{\frac{1}{2}}})$ predicted by \cite{CP20,Bar21} for uniform mesh-refinement to the quasi-optimal rate $\mathcal{O}(h)$, cf.~Remark~\ref{rem:optimal_rate}. In~addition, Figure \ref{fig:Cone_rate} indicates the primal-dual  error estimator is both reliable and efficient~with~respect~to~the~error~quantity~\eqref{eq:reduced_rho}.\vspace{-2mm}

     \begin{figure}[H]
        \centering
        \hspace*{-2mm}\includegraphics[width=15cm]{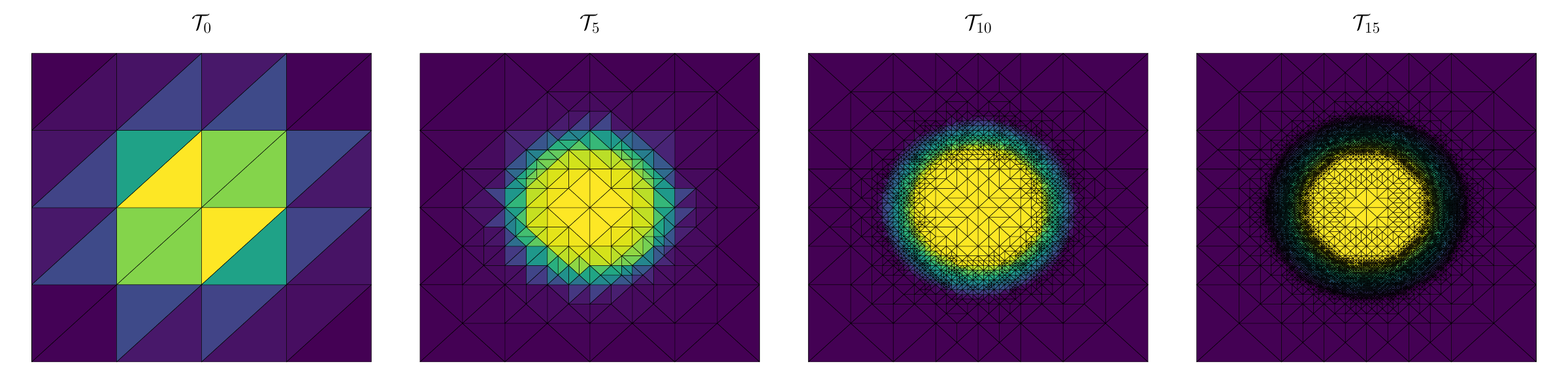}\vspace{-2.5mm}
        \hspace*{-2mm}\includegraphics[width=15cm]{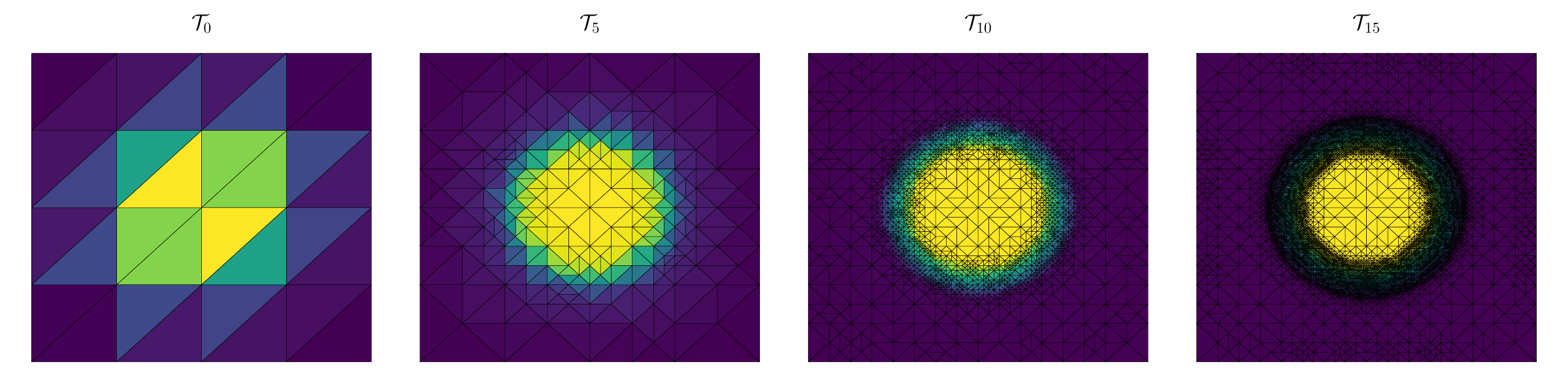}\vspace{-2.5mm}
        \caption{Initial triangulation $\mathcal{T}_0$ and 
        adaptively refined meshes $\mathcal{T}_i$, $i\in \{0,5,10,15\}$, generated by the adaptive Algorithm \ref{alg:afem} (TOP: obtained using (\protect\hyperlink{local}{local}); BOTTOM: obtained using (\protect\hyperlink{global}{global})).}
        \label{fig:Cone_triang}\vspace{-1cm}
    \end{figure}

    \begin{figure}[H]
        \centering
        \hspace*{-5mm}\includegraphics[width=14cm]{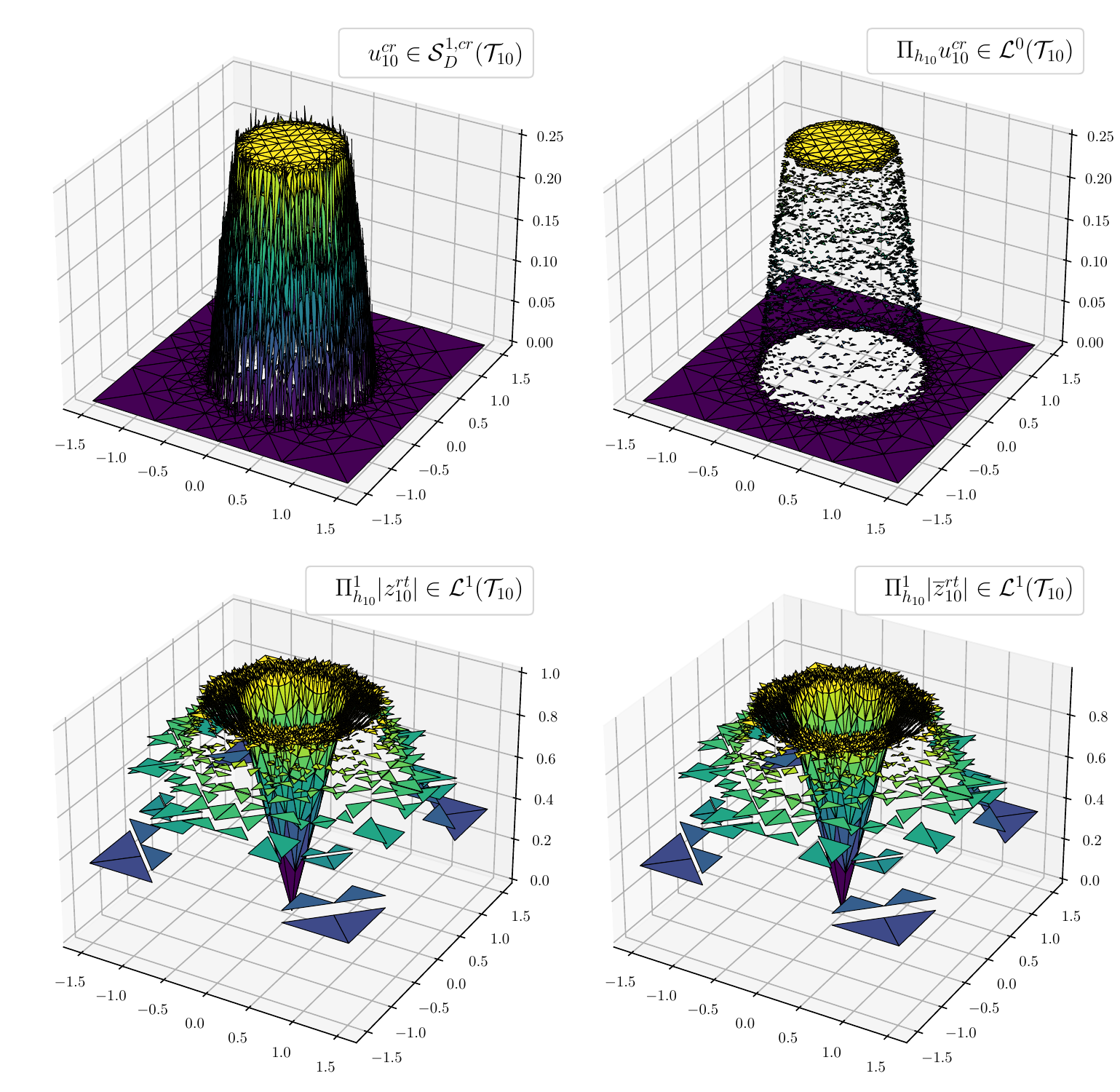}\vspace{-2.5mm}
        \caption{UPPER LEFT: Plot of $u_{10}^{\textit{\textrm{cr}}}\in \mathcal{S}^{1,\textit{\textrm{cr}}}_D(\mathcal{T}_{10})$; UPPER RIGHT: Plot of $\Pi_{h_{10}}^0u_{10}^{\textit{\textrm{cr}}}\in \mathcal{L}^0(\mathcal{T}_{10})$; LOWER LEFT:  Plot of $\Pi_{h_{10}}^1\vert z_{10}^{\textit{\textrm{rt}}}\vert\in \mathcal{L}^1(\mathcal{T}_{10})$; LOWER RIGHT: Plot of $\Pi_{h_{10}}^1\vert \overline{z}_{10}^{\textit{\textrm{rt}}}\vert\in \mathcal{L}^1(\mathcal{T}_{10})$; each obtained using (\protect\hyperlink{local}{local}).}
        \label{fig:Cone_solution}
    \end{figure}\vspace{-5mm}\enlargethispage{7mm}

    \begin{figure}[H]
        \centering
        \hspace*{-2mm}\includegraphics[width=14.5cm]{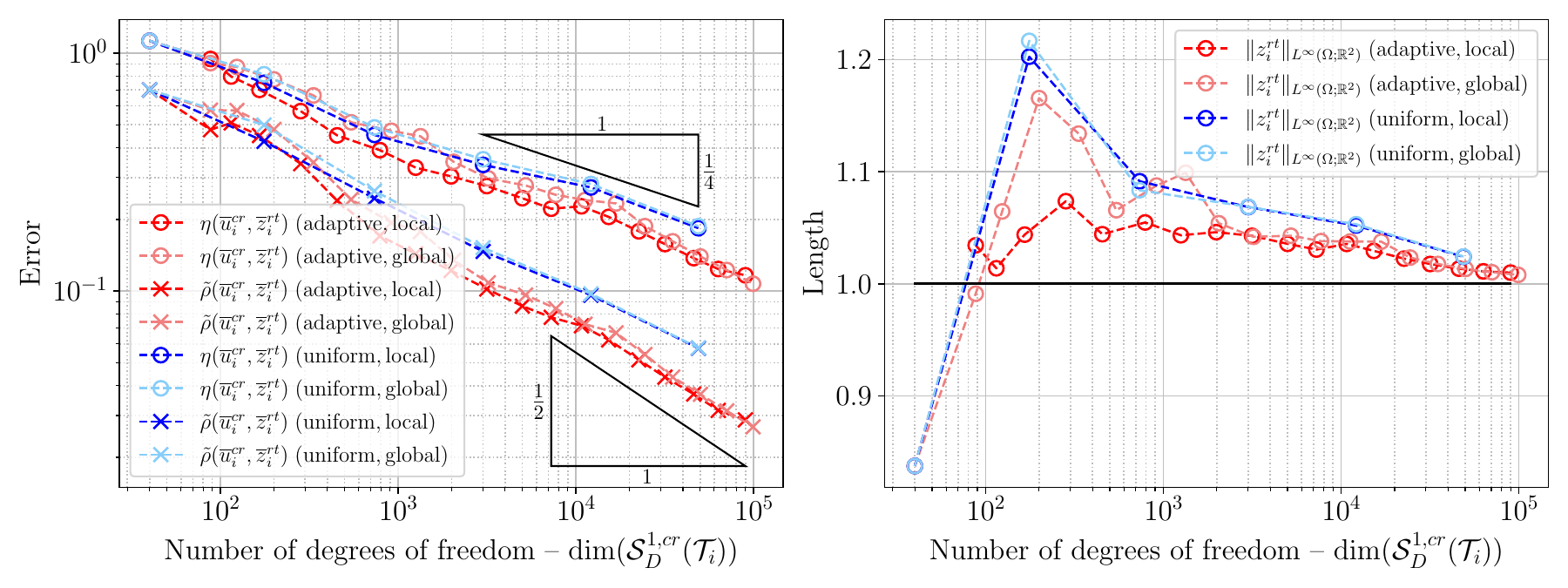}\vspace{-2.5mm}
        \caption{LEFT: Plots of $\eta(\overline{u}_i^{\textit{\textrm{cr}}},\overline{z}_i^{\textit{\textrm{rt}}})$ and $\tilde{\rho}(\overline{u}_i^{\textit{\textrm{cr}}},\overline{z}_i^{\textit{\textrm{rt}}})$ using adaptive mesh refinement for $i=0,\dots,20$ and uniform mesh refinement for $i=0,\dots, 5$; RIGHT: Plots of  $\|\overline{z}_i^{\textit{\textrm{rt}}}\|_{L^\infty(\Omega;\mathbb{R}^2)}$  using adaptive mesh refinement for $i=0,\dots,20$ and uniform mesh refinement for  $i=0,\dots, 5$.}
        \label{fig:Cone_rate}
    \end{figure}

    \subsection{Example without Dirichlet boundary condition and without exact solution}

    \hspace{5mm}We examine an example from \cite{bartels15,BM20}. In this example, we let $\Omega=(-1,1)^2$, $r=\smash{\frac{1}{2}}$, $\Gamma_D=\emptyset$, $\alpha =100$, and $g=\chi_{[-r,r]^2}\in BV(\Omega)\cap L^\infty(\Omega)$. Then, the primal solution 
    and the dual solutions are not known. However, appealing to \cite[Section 5.2]{CP20}, given the regularity of $g\in BV(\Omega)\cap L^\infty(\Omega)$, 
    we can expect the  convergence rate $\mathcal{O}(h^{\frac{1}{4}})$ using uniform mesh refinement.

    The coarsest triangulation $\mathcal{T}_0$ of Figure \ref{fig:OneDisk_triang} (initial triangulation of Algorithm~\ref{alg:afem}) consists of $16$ halved squares. More precisely, Figure \ref{fig:Square_triang} displays
    the triangulations $\mathcal{T}_i$,~${i\in \{0,15,20,25\}}$, generated by Algorithm \ref{alg:afem}
    using either the adaptively modified $\varepsilon_i\in \mathcal{L}^0(\mathcal{T}_i)$, cf.\ (\hyperlink{local}{local}), or the global choice $\varepsilon_i\coloneqq h_i^2$, cf.\ (\hyperlink{global}{global}). For both choices,
    a refinement towards the square $\partial [-r,r]^2$, i.e., the jump set $J_g$ of the data $g\in BV(\Omega)\cap L^\infty(\Omega)$ is reported.
    This behavior is also seen in Figure~\ref{fig:Square_solution}, where the regularized, discrete primal~solution~$u_{15}^{\textit{\textrm{cr}}}\in \mathcal{S}^{1,\textit{\textrm{cr}}}_D(\mathcal{T}_{15})$, the (local)
    $L^2$-projection onto element-wise constant functions 
    $\Pi_{h_{15}} u_{15}^{\textit{\textrm{cr}}}\in \mathcal{L}^0(\mathcal{T}_{15})$, and
    the (local) $L^2$-projections onto element-wise affine functions of 
    the modulus of the regularized, discrete dual solution $z_{15}^{\textit{\textrm{rt}}}\in \mathcal{R}T^0_N(\mathcal{T}_{15})$ and of the projected regularized, discrete dual~solution~${\overline{z}_{15}^{\textit{\textrm{rt}}}\in \mathcal{R}T^0_N(\mathcal{T}_{15})}$ are plotted. Figure \ref{fig:Square_triang}, in addition, shows that using the adaptively modified $\varepsilon_i\in \mathcal{L}^0(\mathcal{T}_i)$, cf.\  (\hyperlink{local}{local}), the refinement is, again, more concentrated at the jump set  $J_g$ of the data $g\in BV(\Omega)\cap L^\infty(\Omega)$. However, in Figure \ref{fig:OneDisk_rate} it can be seen that (\hyperlink{local}{local}) does not result in an improved error decay, but an error decay comparable to  (\hyperlink{global}{global}). In addition, 
    Figure \ref{fig:Square_rate} demonstrates that Algorithm \ref{alg:afem} improves the experimental convergence rate of about $\mathcal{O}(h^{\smash{\frac{1}{4}}})$ predicted by \cite[Section 5.2]{CP20} for uniform mesh-refinement to the value $\mathcal{O}(h^{\smash{\frac{2}{5}}})$.~This,~on~the~one~hand,~confirms~the optimality of the a priori error estimates established in \cite[Section 5.2]{CP20}~and,~on~the~other~hand, appealing to \cite{CP20,Bar21}, let us  expect that there exists no Lipschitz continuous dual solution to the given data $g=\chi_{[-r,r]^2}\in BV(\Omega)\cap L^\infty(\Omega)$. The reported reduced error decay~of~$\mathcal{O}(h^{\smash{\frac{2}{5}}})$~compared~to~\cite{bartels15}, where an error decay of $\mathcal{O}(h^{\smash{\frac{1}{2}}})$ is reported, might only be   pre-asymptotic and due to slight accuracy  losses resulting  due to the global scaling step. This might be due to potential singularities of a dual solution located at the corners of the square $\partial [-r,r]^2$,~as~indicated~in Figure~\ref{fig:Square_solution}. Therefore, it is possible that the error decay $\mathcal{O}(h^{\smash{\frac{1}{2}}})$  in \cite{bartels15} may be reported after surpassing a potential pre-asymptotic regime. 

    \begin{figure}[H]
        \centering
        \hspace*{-2mm}\includegraphics[width=15cm]{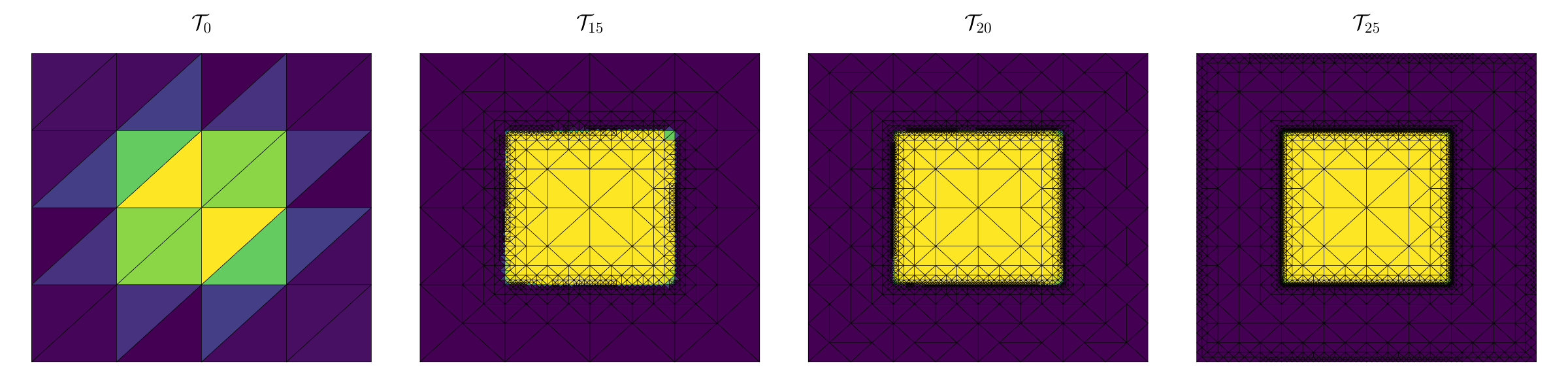}\vspace{-2.5mm}
        \hspace*{-2mm}\includegraphics[width=15cm]{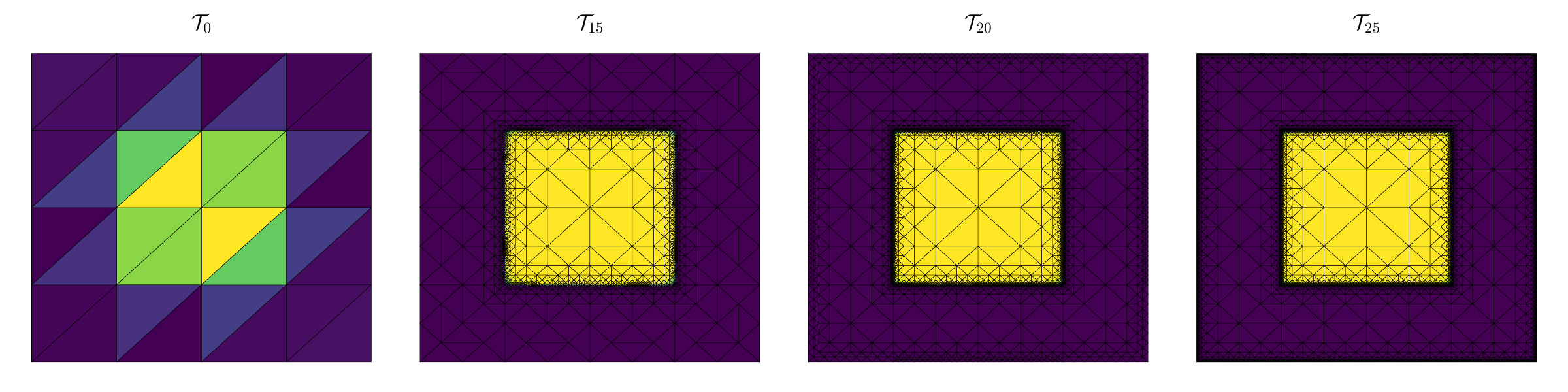}\vspace{-2.5mm}
        \caption{Initial triangulation $\mathcal{T}_0$ and 
        adaptively refined meshes $\mathcal{T}_i$, $i\in \{0,15,20,25\}$, generated by the adaptive Algorithm \ref{alg:afem} (TOP: obtained using (\protect\hyperlink{local}{local}); BOTTOM: obtained using (\protect\hyperlink{global}{global})).}
        \label{fig:Square_triang}
    \end{figure}\newpage

    \begin{figure}[H]
        \centering
        \hspace*{-5mm}\includegraphics[width=14cm]{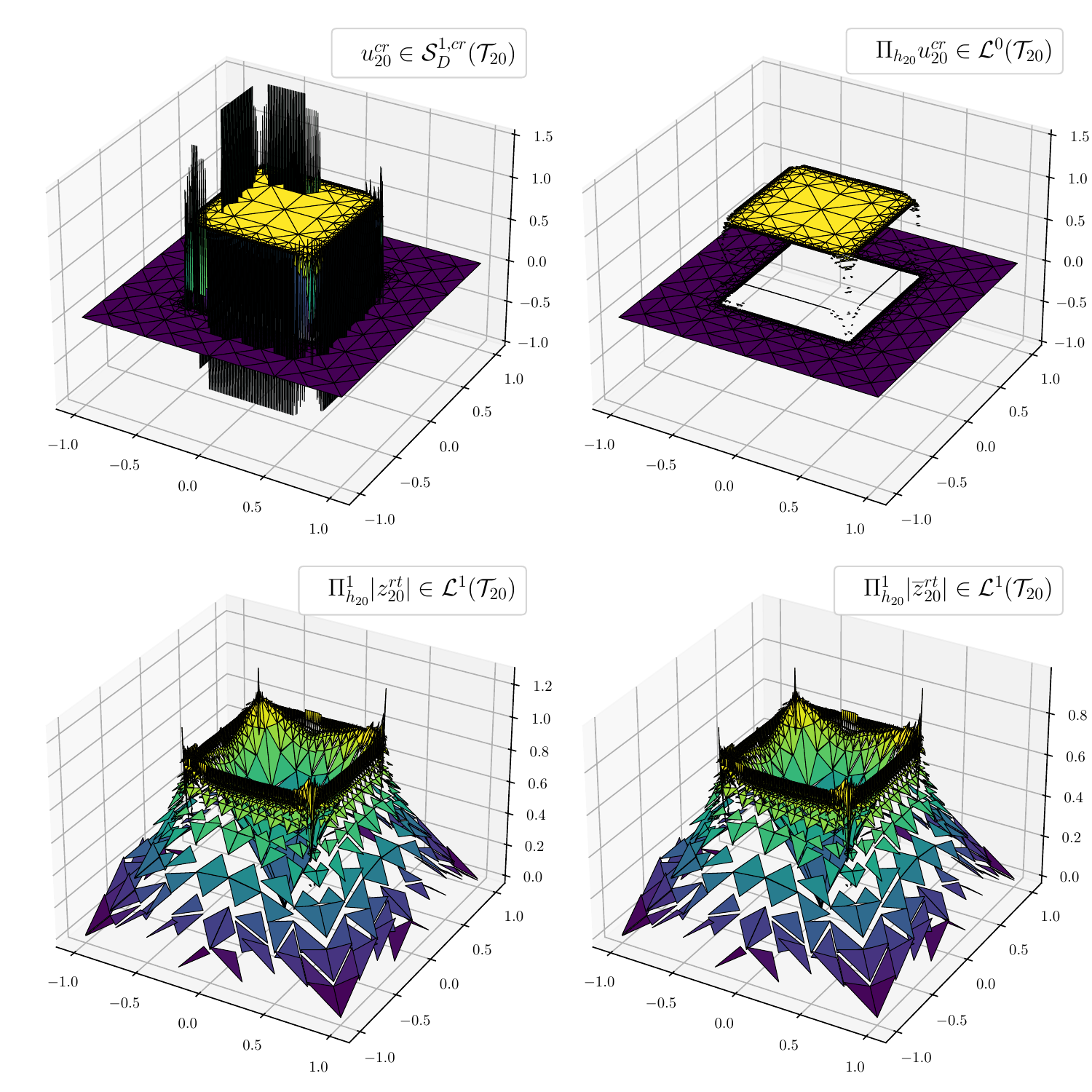}\vspace{-2.5mm}
        \caption{UPPER LEFT: Plot of $u_{20}^{\textit{\textrm{cr}}}\in \mathcal{S}^{1,\textit{\textrm{cr}}}(\mathcal{T}_{20})$; UPPER RIGHT: Plot of $\Pi_{h_{20}}u_{20}^{\textit{\textrm{cr}}}\in \mathcal{L}^0(\mathcal{T}_{20})$; LOWER LEFT: Plot of $\Pi_{h_{20}}^1\vert z_{20}^{\textit{\textrm{rt}}}\vert\in \mathcal{L}^1(\mathcal{T}_{20})$; LOWER RIGHT: Plot of $\Pi_{h_{20}}^1\vert \overline{z}_{20}^{\textit{\textrm{rt}}}\vert\in \mathcal{L}^1(\mathcal{T}_{20})$; each obtained using (\protect\hyperlink{local}{local}).}
        \label{fig:Square_solution}
    \end{figure}\vspace{-5mm}\enlargethispage{10mm}

    \begin{figure}[H]
        \centering
        \hspace*{-2mm}\includegraphics[width=14.5cm]{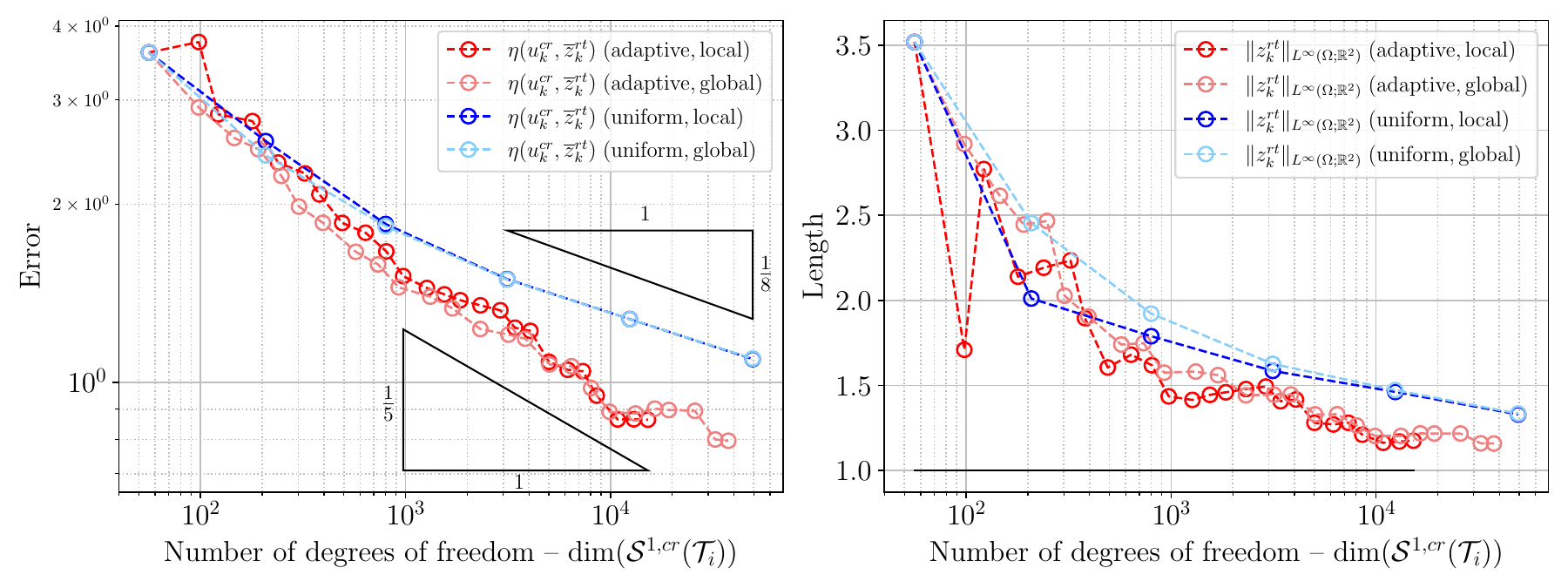}\vspace{-2.5mm}
        \caption{LEFT: Plots of $\eta(u_i^{\textit{\textrm{cr}}},\overline{z}_i^{\textit{\textrm{rt}}})$ and $\tilde{\rho}(u_i^{\textit{\textrm{cr}}},\overline{z}_i^{\textit{\textrm{rt}}})$ using adaptive mesh refinement for $i=0,\dots,25$ and uniform mesh refinement for $i=0,\dots, 5$; RIGHT: Plots of  $\|\overline{z}_i^{\textit{\textrm{rt}}}\|_{L^\infty(\Omega;\mathbb{R}^2)}$  using adaptive mesh refinement for $i=0,\dots,25$ and uniform mesh refinement for  $i=0,\dots, 5$.}
        \label{fig:Square_rate}
    \end{figure}

   \subsection{Numerical experiments with application to image processing}

   \hspace{5mm}In order to benchmark the performance of the proposed numerical scheme (cf. Algorithm~\ref{algorithm} and Algorithm \ref{alg:afem})
   in a problem related to image processing,  we examine a standard  example from the field of image processing (cf. Section \ref{subsubsec:cameraman}) and a new example (cf.~Section~\ref{subsubsec:merle}).\vspace{-1mm}\enlargethispage{11mm}
   
   \subsubsection{The Cameraman image}\label{subsubsec:cameraman}
   
   \hspace{5mm}We examine the cameraman image, which in a similar context has been considered~in~\cite{bartels15}. In this example, 
   we let $\Omega\coloneqq (0,1)^2$, $\Gamma_D=\emptyset$, $\alpha=1\mathrm{e}{+}4$,~and~${g\in BV(\Omega)\cap L^\infty(\Omega)}$ a piece-wise constant function taking its values in  the interval $[0,1]$, representing the cameraman image on a uniform triangulation with $66.049$
 nodes, cf.\ Figure \ref{fig:Cameraman_solution}.\ The adaptive algorithm~(cf.~Algorithm~\ref{alg:afem}), employed as coarsening strategy, reduces
the number of nodes within $30$ iteration steps to $25.059$ nodes which corresponds to $38.0$\% of the initial number of nodes,~which~results in a squared $L^2$-error of  $\|u_{30}^{cr}-g\|_{L^2(\Omega)}^2\hspace*{-0.1em}\approx \hspace*{-0.1em} 2.211\mathrm{e}{-}3$. The resulting coarsened image,~\mbox{represented}~by~${u_{30}^{cr}\hspace*{-0.1em}\in \hspace*{-0.1em}\mathcal{S}^{1,cr}(\mathcal{T}_{30})}$, is shown in~Figure~\ref{fig:Cameraman_solution}. The underlying grid $\mathcal{T}_{30}$ shown in Figure \ref{fig:Cameraman_triangulation} reveals the expected coarsening of the triangulation away from the edges.\vspace{-1mm}

  \begin{figure}[H]
        \centering
        \includegraphics[width=5.5cm]{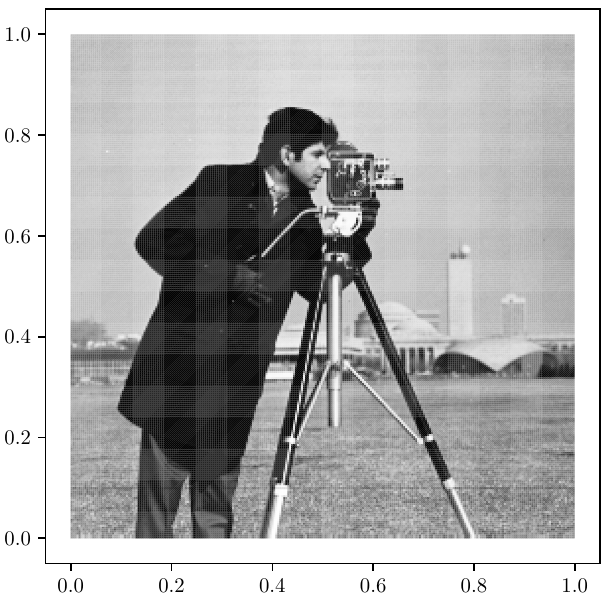}\hspace*{10mm}\includegraphics[width=5.5cm]{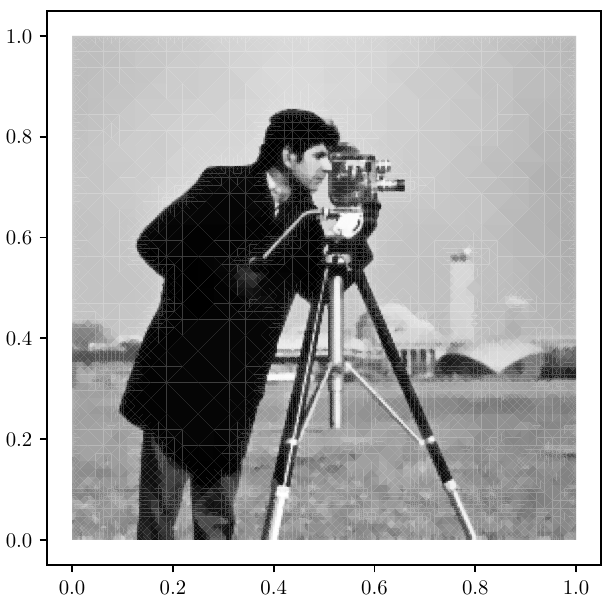}
        \caption{LEFT: Plot of the cameraman image $g\in BV(\Omega)\cap L^\infty(\Omega)$ on a grid with $66.049$~nodes; RIGHT:
        Plot of coarsened image $u_{30}^{cr}\in \mathcal{S}^{1,cr}(\mathcal{T}_{30})$ on $\mathcal{T}_{30}$ with $25.059$ nodes, cf. Figure \ref{fig:Cameraman_triangulation}.}
        \label{fig:Cameraman_solution}
\end{figure}\vspace{-0.5cm}


 \begin{figure}[H]
        \centering
      \includegraphics[width=8cm]{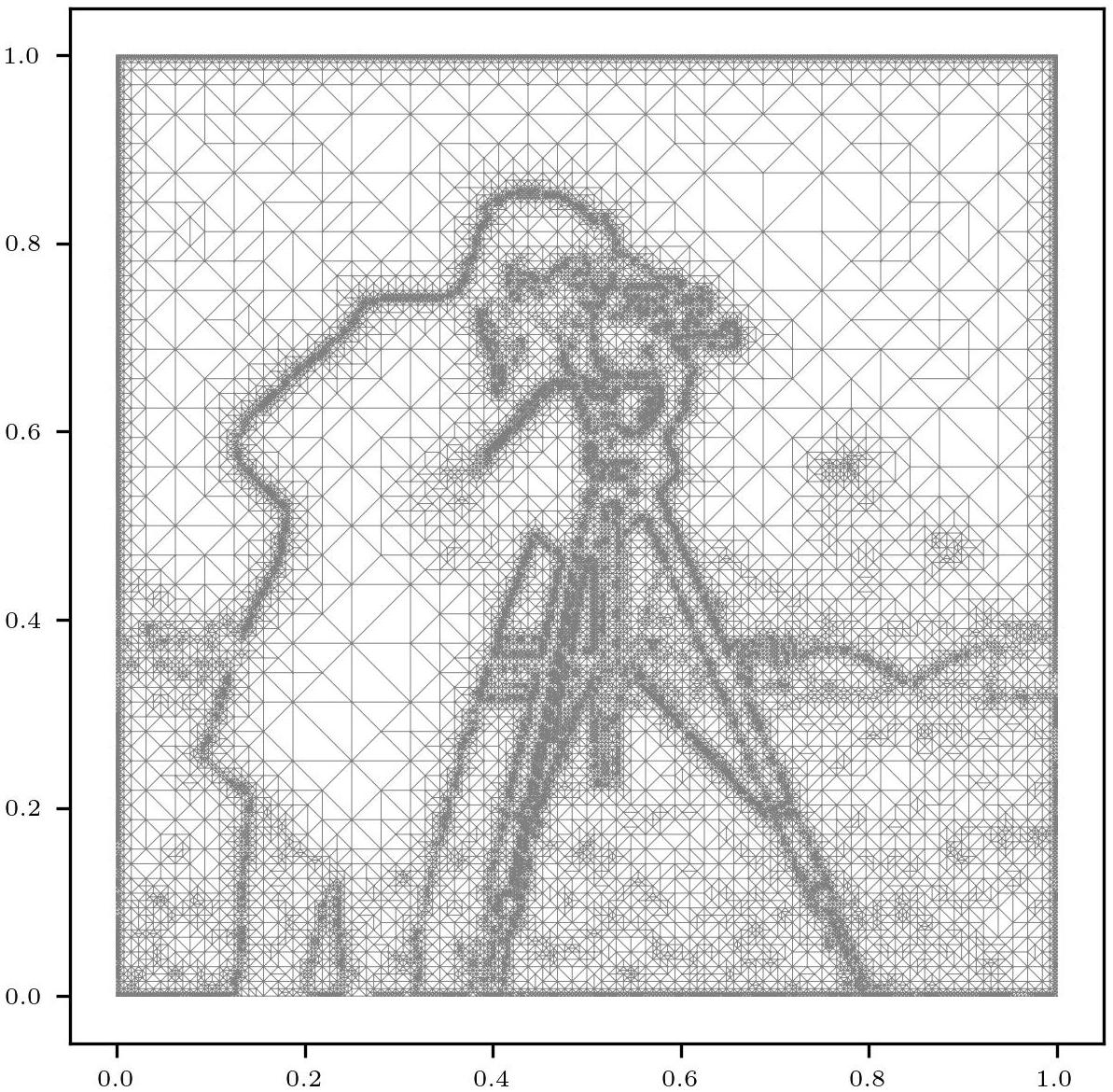}
        \caption{Triangulation $\mathcal{T}_{30}$ in the coarsened image $u_{30}^{cr}\in \mathcal{S}^{1,cr}(\mathcal{T}_{30})$ on the right~of~Figure~\ref{fig:Cameraman_solution}.}
        \label{fig:Cameraman_triangulation}
    \end{figure}

    \subsubsection{The Merle image}\label{subsubsec:merle}\enlargethispage{10mm}
   
   \hspace{5mm}We examine an image of Merle, the male cat of the second author. In this example, 
    we let $\Omega\coloneqq (0,1)^2$, $\Gamma_D=\emptyset$, $\alpha=1\mathrm{e}{+}4$, and 
    $g\in BV(\Omega)\cap L^\infty(\Omega)$ a piece-wise constant function taking its values in  the interval $[0,1]$, representing the Merle image on a uniform triangulation with $140.625$
    nodes, cf.  Figure \ref{fig:Merle_solution}. The adaptive algorithm (cf.~Algorithm~\ref{alg:afem}), employed as coarsening strategy, reduces
    the number of nodes within $30$ iteration steps to $41.749$ nodes which is $30.0$\% of the initial number of nodes, which results in a squared $L^2$-error of  $\|u_{30}^{cr}-g\|_{L^2(\Omega)}^2\approx 2.162\mathrm{e}{-}3$. The resulting  coarsened image, represented by $u_{30}^{cr}\in \mathcal{S}^{1,cr}(\mathcal{T}_{30})$,  is shown in  Figure \ref{fig:Merle_solution}. The underlying grid $\mathcal{T}_{30}$ shown in Figure \ref{fig:Merle_triangulation} reveals the expected coarsening of the triangulation away from the edges.\vspace{-0.15cm}

    \begin{figure}[H]
        \centering
        \includegraphics[width=5.5cm]{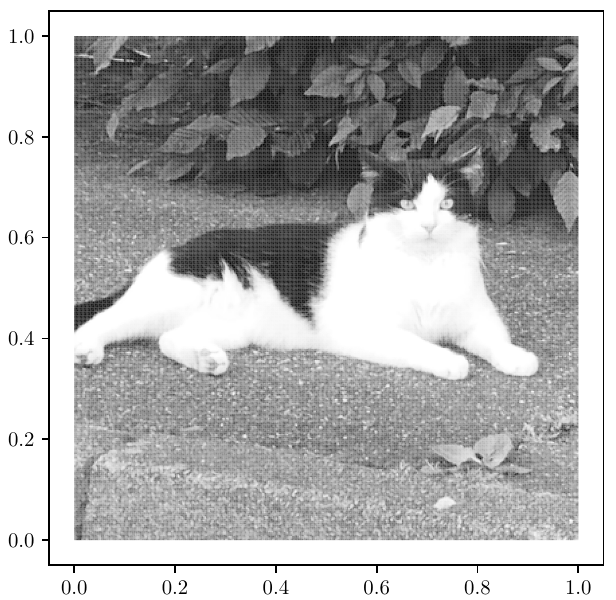}\hspace{10mm}\includegraphics[width=5.5cm]{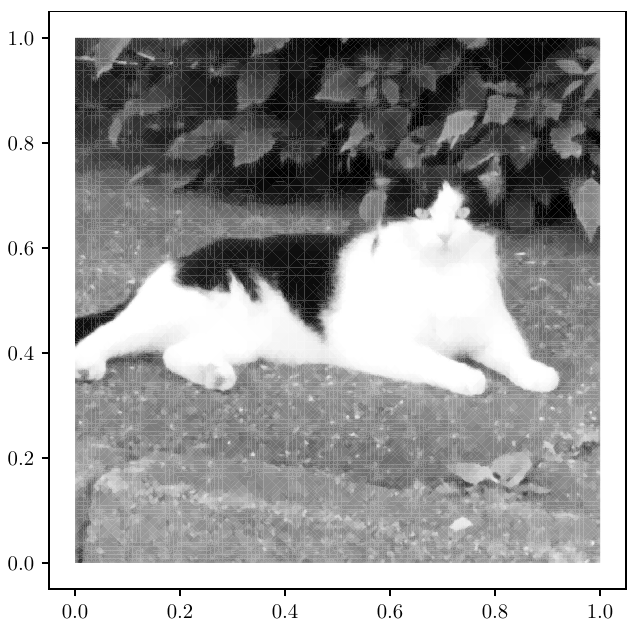}
        \caption{LEFT: Plot of the Merle image $g\in BV(\Omega)\cap L^\infty(\Omega)$ on a grid with $140.625$ nodes; RIGHT:
        Plot of coarsened image $u_{30}^{cr}\in \mathcal{S}^{1,cr}(\mathcal{T}_{30})$ on $\mathcal{T}_{30}$ with $41.749$ nodes, cf. Figure \ref{fig:Merle_triangulation}.}
        \label{fig:Merle_solution}
\end{figure}\vspace{-0.5cm}\enlargethispage{5mm}

    \begin{figure}[H]
        \centering
    \includegraphics[width=8cm]{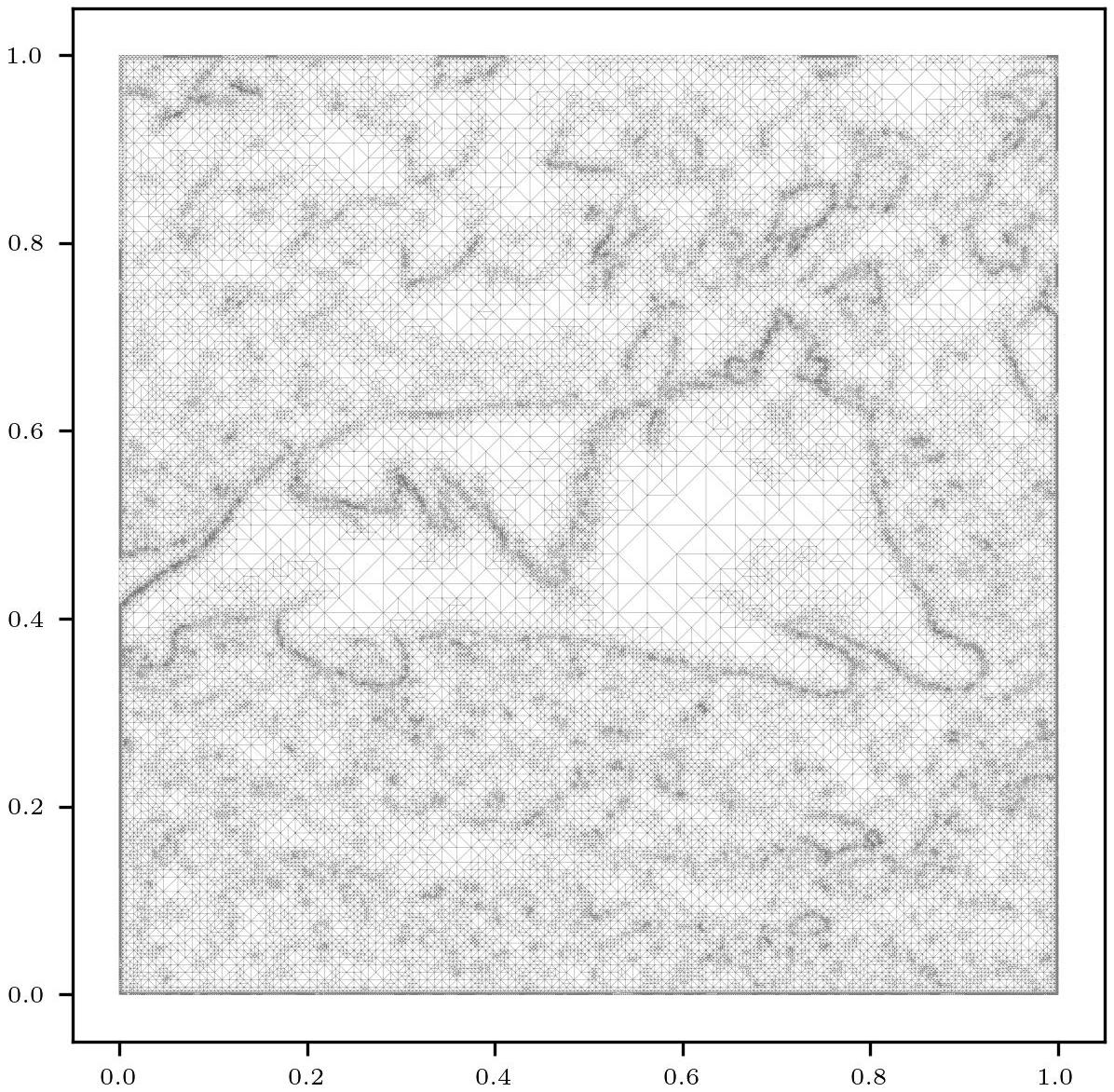}
        \caption{Triangulation $\mathcal{T}_{30}$ in the coarsened image of $u_{30}^{cr}\in \mathcal{S}^{1,cr}(\mathcal{T}_{30})$ on the right~of~Figure~\ref{fig:Merle_solution}.}
        \label{fig:Merle_triangulation}
    \end{figure}

{\setlength{\bibsep}{0pt plus 0.0ex}\small
 \providecommand{\MR}[1]{}
\providecommand{\bysame}{\leavevmode\hbox to3em{\hrulefill}\thinspace}
\providecommand{\noopsort}[1]{}
\providecommand{\mr}[1]{\href{http://www.ams.org/mathscinet-getitem?mr=#1}{MR~#1}}
\providecommand{\zbl}[1]{\href{http://www.zentralblatt-math.org/zmath/en/search/?q=an:#1}{Zbl~#1}}
\providecommand{\jfm}[1]{\href{http://www.emis.de/cgi-bin/JFM-item?#1}{JFM~#1}}
\providecommand{\arxiv}[1]{\href{http://www.arxiv.org/abs/#1}{arXiv~#1}}
\providecommand{\doi}[1]{\url{https://doi.org/#1}}
\providecommand{\MR}{\relax\ifhmode\unskip\space\fi MR }
\providecommand{\MRhref}[2]{%
  \href{http://www.ams.org/mathscinet-getitem?mr=#1}{#2}
}
\providecommand{\href}[2]{#2}

 }

\end{document}